\documentclass[letterpaper,PhD,10pt,reqno]{usydthesis}

\usepackage{amsmath,amssymb,amsthm,latexsym,mathrsfs,euscript,verbatim}

\usepackage{graphicx,mathrsfs,srcltx,xtab,fancyhdr,fancybox}
\usepackage{hyperref}
\usepackage{pstricks,amscd}
\usepackage[margin=2.5cm]{geometry}
\usepackage[Conny]{fncychapthesis} %edited fncychap.sty; old and new version in C:\Program Files\MiKTeX 2.7\tex\latex\fncychap

\newtheoremstyle{thm}
{10pt}% this gives how much space is left above theorem
{6pt}% space below theorem
{\itshape}% font used (italics for theorems)
{}% Indent amount (empty = no indent, \parindent = paragraph indent)
{\bf}% Tital font
{.}% Punctuation after thm head
{.5em}% Space after thm head
{} % This one should just be left empty

\theoremstyle{thm}

\newtheorem{thm}{Theorem}[chapter]
\newtheorem{prop}[thm]{Proposition}
\newtheorem{lemma}[thm]{Lemma}
\newtheorem{cor}[thm]{Corollary}

\theoremstyle{definition}
\newtheorem{defn}[thm]{Definition}

\DeclareMathOperator{\im}{Im} 
\DeclareMathOperator{\End}{End}
\newcommand{\nc}{\newcommand}
\newcommand{\rnc}{\renewcommand}
\renewcommand{\l}{\lambda}
\renewcommand{\d}{\delta}
\rnc{\k}{\kappa}

\newcommand{\e}{\epsilon}
\newcommand{\noi}{\noindent}
\nc{\bmw}{\mathscr{C}}
\nc{\ckt}[1]{\mathbb{KT}_{#1}^k}
\nc{\akt}[1]{\widehat{\mathbb{KT}}_{#1}}
\nc{\cl}[1]{\mathrm{cl}_{#1}}
\newcommand{\bnk}{\mathscr{B}_n^k}
\newcommand{\bmk}{\mathscr{B}_{n-1}^k}
\newcommand{\bt}{\mathscr{B}_2^k}
\rnc{\b}[1]{\,\overline{\!#1}}
\nc{\blk}{\mathscr{B}_l^k}
\nc{\bl}[1]{\widetilde{\mathscr{B}}_{#1}^k}
\nc{\ak}[1]{\mathfrak{h}_{#1,k}}
\newcommand{\At}[1]{\widetilde{\mathfrak{h}}_{#1,k}}
\newcommand{\ix}{X_i^{-1}}
\newcommand{\iy}{Y^{-1}}
\nc{\kum}[1]{\sum_{#1=0}^{k-1}}
\nc{\kkum}[1]{\sum_{#1=0}^{k}}
\nc{\pum}[1]{\sum_{#1=1}^p}
\nc{\ppum}[1]{\sum_{#1=0}^{p-1}}
\nc{\tw}[1]{0 \leq #1 \leq k-1}
\nc{\tf}{\text{,\quad for all }}
\nc{\iif}{\text{,\quad if }}
\nc{\la}[1]{\left\langle #1 \right\rangle}
\nc{\rel}[1]{\stackrel{(\ref{#1})}{=}}
\newcommand{\stack}[1]{\stackrel{#1}{=}}
\nc{\sstack}[1]{\stackrel{#1}{\subseteq}}
\nc{\f}{&\hphantom{=}\,\,}
\nc{\p}{\stackrel{\hphantom{(1)}}{=}}
\nc{\h}{\stackrel{\hphantom{(19)}}{=}}
\newcommand{\y}[2]{Y^{\prime \,#2}_{#1}}

\rnc{\a}[2]{\alpha_{#1}^{#2}}
\rnc{\o}{\ldots}
\nc{\floor}[1]{\left\lfloor #1 \right\rfloor}
\nc{\ceil}[1]{\left\lceil \frac{#1}{2} \right\rceil}
\rnc{\ll}[1]{\left( #1 \right)}
\nc{\s}{\subseteq}
\nc{\Z}{\mathbb{Z}}
\nc{\dg}{\psi} %map from \bnk to CKT
\nc{\wh}[1]{\widehat{#1}}
\newcommand{\B}{\mathbb B}
\nc{\M}{\mathbb{M}}
\nc{\X}{\mathfrak{W}}

\def\be{\begin{align*}}
\def\ee{\end{align*}}
\rnc{\=}{&=&}
\rnc{\-}{&=}
\newcommand{\beq}{\begin{align}}
\newcommand{\eeq}{\end{align}}

\newcommand{\commentout}[1]{}
\rnc{\sectionmark}[1]{%
\markright{\thesection\ #1}}
\fancypagestyle{plain}{%
\fancyhead{} % get rid of headers on plain pages
} \makeindex 

\begin{document}
\author{\vskip 1cm Shona Huimin Yu}
\title{The Cyclotomic Birman-Murakami-Wenzl Algebras}
%\date{September 2007}
\maketitle
%\pagenumbering{roman}
\setcounter{tocdepth}{2}
\pagebreak

{\thispagestyle{plain}
\begin{center} 
\textbf{\Large Abstract \\ (2008)}
\vspace{0.2cm}
\end{center}
\vskip -0.2cm

This thesis presents a study of the cyclotomic BMW (Birman-Murakami-Wenzl) algebras, introduced by H\"aring-Oldenburg as a generalization of the BMW algebras associated with the cyclotomic Hecke algebras of type $G(k,1,n)$ (also known as Ariki-Koike algebras) and type $B$ knot theory involving affine/cylindrical tangles.

The motivation behind the definition of the BMW algebras may be traced back to an important problem in knot theory; namely, that of classifying knots (and links) up to isotopy.
The algebraic definition of the BMW algebras uses generators and relations originally inspired by the Kauffman link invariant. They are closely connected with the Artin braid group of type~$A$, Iwahori-Hecke algebras of type $A$, and with many diagram algebras, such as the Brauer and Temperley-Lieb algebras.
Geometrically, the BMW algebra is isomorphic to the Kauffman Tangle algebra. The representations and the
cellularity of the BMW algebras have now been extensively studied in the literature. These algebras also feature in the theory of quantum groups, statistical mechanics, and topological quantum field theory.

In view of these relationships between the BMW algebras and several objects of ``type $A$'', several authors have since naturally generalized the BMW algberas for other types of Artin groups.
Motivated by knot theory associated with the Artin braid group of type $B$, H\"aring-Oldenburg introduced the cyclotomic BMW algebras $\mathscr{B}_n^k$ as a generalization of the BMW algebras such that the Ariki-Koike algebra $\mathfrak{h}_{n,k}$ is a quotient of $\mathscr{B}_n^k$, in the same way the Iwahori-Hecke algebra of type $A$ is a quotient of the BMW algebra.

In this thesis, we investigate the structure of these algebras and show they have a topological realization as a certain cylindrical analogue of the Kauffman Tangle algebra. In particular, they are shown to be $R$-free of rank $k^n (2n-1)!!$ and
bases that may be explicitly described both algebraically and diagrammatically in terms of cylindrical tangles are obtained.
Unlike the BMW and Ariki-Koike algebras, one must impose extra so-called ``admissibility conditions'' on the parameters of the ground ring in order for these results to hold. This is due to potential torsion caused by the polynomial relation of order $k$ imposed on one of the generators of $\mathscr{B}_n^k$. It turns out that the representation theory of $\mathscr{B}_2^k$ is crucial in determining these conditions precisely. The representation theory of $\mathscr{B}_2^k$ is analysed in detail in a joint preprint with Wilcox in \cite{WY06}} \thispagestyle{plain}{(http://arxiv.org/abs/math/0611518). The admissibility conditions and a universal ground ring with admissible parameters are given explicitly in Chapter 3.

The admissibility conditions are also closely related to the existence of a non-degenerate Markov trace function of $\mathscr{B}_n^k$ which is then used together with the cyclotomic Brauer algebras in the linear independency arguments contained in Chapter 4.

Furthermore, in Chapter 5, we prove the cyclotomic BMW algebras are cellular, in the sense of Graham and Lehrer. The proof uses the cellularity of the Ariki-Koike algebras (Graham-Lehrer \cite{GL96} and Dipper-James-Mathas \cite{DJM98}) and an appropriate ``lifting'' of a cellular basis of the Ariki-Koike algebras into $\bnk$, which is compatible with a certain anti-involution of $\mathscr{B}_n^k$.

When $k = 1$, the results in this thesis specialize to those previously established for the BMW algebras by Morton-Wasserman \cite{MW89}, Enyang \cite{E04}, and Xi \cite{X00}. \\
\\
\textbf{Remarks:}

During the writing of this thesis, Goodman and Hauschild-Mosley also attempt similar arguments to establish the freeness and diagram algebra results mentioned above. However, they withdrew their preprints \cite{GH107,GH207}, due to issues with their generic ground ring crucial to their linear independence arguments. A similar strategy to that proposed in \cite{GH107}, together with different trace maps and the study of rings with admissible parameters in Chapter 3, is used in establishing linear independency of our basis in Chapter 4. 

Since the submission of this thesis, new versions of these preprints have been released in which Goodman and Hauschild-Mosley use alternative topological and Jones basic construction theory type arguments to establish freeness of $\mathscr{B}_n^k$ and an isomorphism with the cyclotomic Kauffman Tangle algebra. However, they require their ground rings to be an integral domain with parameters satisfying the (slightly stronger) admissibility conditions introduced by Wilcox and the author in \cite{WY06}.
Also, under these conditions, Goodman has obtained cellularity results.}

Rui and Xu have also proved freeness and cellularity results when $k$ is odd, and later Rui and Si for general $k$, under the assumption that $\delta$ is invertible and using another stronger condition called ``$u$-admissibility''. The methods and arguments employed are strongly influenced by those used by Ariki, Mathas and Rui \cite{AMR06} for the cyclotomic Nazarov-Wenzl algebras and involve the construction of seminormal representations; their preprints have recently been released on the arXiv.

\thispagestyle{plain}
{
It should also be noted there are slight differences between the definitions of cyclotomic BMW algebras and ground rings used, as explained partly above.
Furthermore, Goodman and Rui-Si-Xu use a weaker definition of cellularity, to bypass a problem discovered in their original proofs relating to the anti-involution axiom of the original Graham-Lehrer definition.
}
\pagebreak
\tableofcontents
\setcounter{chapter}{-1}
\pagenumbering{arabic}

\newpage

{\thispagestyle{plain}

{\mbox{}\vfill
  \begin{center}
{\em I dedicate this to my family, who have endured me for the longest. \\ To my parents and my little sister Nina, this is for you.} 
  \end{center}
\mbox{}\vfill}

\newpage

{\thispagestyle{plain}
\begin{center} 
\textbf{\Large Acknowledgements}
\end{center}
Firstly, I thank my supervisor Bob Howlett for his patience, encouragement and help during my degree. Thank you for initiating my studies at the University of Sydney and exposing me to various areas of  mathematics. Thank you also for providing the pstricks diagrams in this thesis.
I thank my co-supervisor Gus Lehrer for his assistance and Andrew Mathas for always being generous with his time and giving me helpful advice and criticism on my work.

Also, I am very grateful to Bob, Gus and Andrew for their financial  support throughout my degree. Thank you for your generosity when my scholarships had expired and, in particular, for giving me invaluable experiences by allowing me to attend international conferences and meet and learn from experts in the area.

My special thanks to Stewart Wilcox, a best friend and colleague, without  whom this thesis would not have been completed. Thank you for always encouraging me and for sharing your mathematical enthusiasm, thoughts and  ideas with me. I have benefited from our many, both mathematical and non-mathematical, discussions. Thank you also for discovering and remedying errors in past literature pertaining to this thesis and for being a meticulous proofreader.

To all my friends from University of Sydney, thank you  for providing such a crazy jovial environment.
To the past and present residents of the Carslaw 4$^\mathrm{th}$ floor corridor, thank you for being a reliable source of procrastination, entertainment, deeply philosophical discussions, and an ``uncountably  finite'' number of corny maths jokes. 
The good times we have shared, even the indefatigably constant teasing, have enriched my life. Thank you all for your constant support and, in particular, Stewart Wilcox, Anna Wang, Nicholas  Wilson, Neil Saunders, Timothy Schaerf, Tanya Wong, Chris Ormerod, Janet  Thomas, Justin Koonin, Erwin Lobo, Hai Ho, Leanne Rylands, and Sonia Morr, for your attempts to keep me calm and see beyond my pessimism and cynicism.
Thank you also to the staff of the School of Mathematics and Statistics at the University of Sydney for always being friendly and helpful.
 
I also gratefully acknowledge the financial support of the University Postgraduate Award, the PRSS, and the Algebra group at The University of Sydney.

Last but not least, I must thank my family, my parents and my little sister  Nina, for their continual love and support and for standing by me over the years. Sorry it took me so long. Words cannot express how thankful I am for the sacrifices you have all made these last few years so that I could pursue my studies abroad. Thank you for providing me with a home away from home and for putting up with me. I love you all. Thank you God for my family and friends.
}

\setcounter{page}{0}
\chapter{Introduction}

The Birman-Murakami-Wenzl (BMW) algebras, conceived independently by Birman and Wenzl \cite{BW89} and Murakami \cite{M87}, are defined by generators and relations originally inspired by the Kauffman link invariant of \cite{K90}. The BMW algebras are closely connected with Artin braid groups, Iwahori-Hecke algebras of the symmetric group, and Brauer algebras. In fact, they may be construed as deformations of the Brauer algebras obtained by replacing the symmetric group algebras with the corresponding Iwahori-Hecke algebras.
\vspace{-0.13cm}
\begin{defn} \label{defn:bmw}
Fix a natural number $n$. Let $R$ be a unital commutative ring 
containing units $A_0,q,\l$ such that $\l - \l^{-1} = \d(1-A_0)$ holds, where $\d := q - q^{-1}$. The \emph{\textbf{BMW algebra}} 
$\bmw_n := \bmw_n(q,\l,A_0)$ is defined to be the unital associative $R$-algebra 
generated by $X_1^{\pm 1}, \ldots, X_{n-1}^{\pm 1}$ and 
$e_1, \ldots, e_{n-1}$ subject to the following relations, which hold for all possible values of $i$ unless otherwise stated.
\[\begin{array}{rcll}
\hspace{2cm} X_i - X_i^{-1} &=& \d(1-e_i)& \\
X_iX_j &=& X_jX_i &\qquad \text{for } |i-j| \geq 2 \\
X_iX_{i+1}X_i &=& X_{i+1}X_iX_{i+1}& \\
X_ie_j &=& e_jX_i &\qquad \text{for } |i-j| \geq 2 \\
e_ie_j &=& e_je_i &\qquad \text{for } |i-j| \geq 2 \\
X_ie_i &=& e_iX_i \,\,\,=\,\,\, \l e_i& \\
X_i X_j e_i &=& e_j e_i \,\,\,=\,\,\, e_j X_i X_j &\qquad \text{for } |i-j| = 1 \\
e_i e_{i\pm 1} e_i &=& e_i& \\
e_i^2 &=& A_0 e_i.& 
\end{array}
\]
\end{defn}
%\pagebreak
The relations above are an algebraic version of geometric relations satisfied by certain tangle diagrams in the \emph{Kauffman tangle algebra} $\mathbb{KT}_{n}$, an algebra of (regular isotopy equivalence classes of) tangles on $n$ strands in the disc cross the interval (that is, a solid cylinder) modulo the Kauffman skein relation; see Kauffman \cite{K90} and Morton and Traczyk \cite{MT90}. In particular, the relations $X_i - X_i^{-1} = \d(1-e_i)$ reflects the Kauffman skein relation which is typically presented as:
  \begin{center}
    \includegraphics{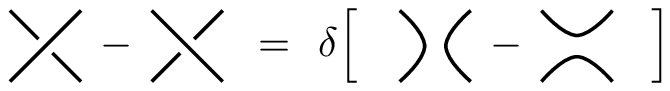}
  \end{center}
Naturally, one expects the BMW algebras to be isomorphic to the Kauffman tangle algebras and, indeed, Morton and Wasserman establish this isomorphism (illustrated below) in \cite{MW89} and, as a result, show the algebra is free of rank $(2n-1)!! = (2n-1) \cdot (2n-3) \cdots 3 \cdot 1$.
\vspace{0.2cm}
  \begin{center}
 \includegraphics{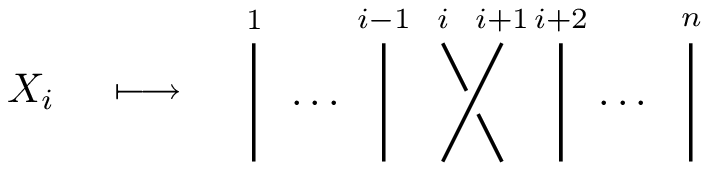} 
    \vskip 0.7cm
    \hspace*{-0.07 cm} 
    \includegraphics{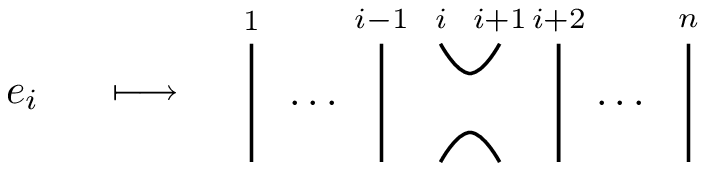}
\end{center}

The representation theory of the BMW algebras has been studied by various authors. Birman and Wenzl \cite{BW89} construct a nondegenerate Markov trace on the algebra using the Kauffman link invariant. The existence of this nondegenerate trace together with tools from the Jones Basic Construction theory (see Jones \cite{J83} and Wenzl \cite{W88}) allow them to derive the structure of the algebra in the generic semisimple setting. They prove that the algebra $\bmw_n$ is generically semisimple with irreducible representations indexed by Young diagrams of size $n-2f$, where $0 \leq f \leq \floor{\frac{n}{2}}$. Wenzl \cite{W90} provides sufficient conditions for the BMW algebras to be semisimple. Moreover, Rui and Si \cite{RS06} recently produced a criterion for semisimplicity of the BMW algebras over an arbitrary field. 

The Iwahori-Hecke algebra of the symmetric group is a quotient of the BMW algebra. Using this connection, several authors have determined analogues of results about the representations and characters of the Iwahori-Hecke algebras for the BMW algebras. For example, Enyang \cite{E04} and Xi \cite{X00} both exploit the fact the Iwahori-Hecke algebra of the symmetric group is cellular, in the sense of Graham and Lehrer \cite{GL96}, to investigate the cellularity of the BMW algebra. Xi shows that certain analogues of the Kazhdan-Lusztig basis for BMW algebras, studied by Fishel and Grojnowski \cite{FG95}, Morton and Trazcyk \cite{MT90} and Morton and Wasserman \cite{MW89}, are in fact cellular. Xi's basis is constructed using certain diagrams called \emph{dangles} and a basis of the Iwahori-Hecke algebra. On the other hand, Enyang produces a basis indexed by certain bitableaux and gives an explicit combinatorial description of his cellular basis. Further results can also be found in Halverson and Ram \cite{HR95} and Leduc and Ram \cite{LR97}.

It is not surprising that several authors have since generalised the BMW algebras for arbitrary simply laced Artin groups (see Cohen et al. \!\!\cite{CGW05}), and defined affine (see Goodman and Hauschild \cite{GH06}) and cyclotomic (see H\"aring-Oldenburg \cite{HO01}) versions. Also, degenerate versions of these algebras exist in the literature; recently, Ariki, Mathas and Rui \cite{AMR06} defined and studied the representation theory of ``cyclotomic Nazarov-Wenzl algebras'', which are quotients of Nazarov's degenerate affine BMW algebras \cite{N96}.
Other quotients and specialisations of the affine and cyclotomic BMW algebras have also appeared in the literature, such as the cyclotomic Brauer (see Rui and Yu \cite{RY04}) and cyclotomic Temperley-Lieb algebras (see Rui and Xi \cite{RX04}). 
The BMW algebras and the algebras we have mentioned above also play a role in the study of quantum groups, quantum field theory, subfactors and statistical mechanics.

Motivated by type $B$ knot theory and the Ariki-Koike algebras, H\"aring-Oldenburg  introduced the ``cyclotomic BMW algebras'' in \cite{HO01}. They are so named because the Ariki-Koike algebras~\cite{AK94,BM93}, which are also known as cyclotomic Hecke algebras of type $G(k,1,n)$, arise as quotients of cyclotomic BMW algebras in the same way as the Iwahori-Hecke algebras arise as quotients of BMW algebras. They are obtained from the original BMW algebras by adding an extra generator $Y$ satisfying a polynomial relation of finite order $k$ and imposing several further relations modelled on type $B$ knot theory. For example, $Y$ satisfies the Artin braid relations of type $B$ with the generators $X_1$, \o, $X_{n-1}$ of the ordinary BMW algebra. The cyclotomic BMW algebras and its representations in the generic case were studied by H\"aring-Oldenburg \cite{HO01}, Orellana and Ram \cite{OR04}, Goodman and Hauschild Mosley \cite{GH107,GH207}. 

When this $k^\mathrm{th}$ order relation on the generator $Y$ is omitted, one obtains the infinite dimensional affine BMW algebras, studied by Goodman and Hauschild in \cite{GH06}. This extra affine generator may be visualised as the cylindrical braid of type $B$ illustrated below.
\begin{figure}[h!]
\begin{center}
  \includegraphics{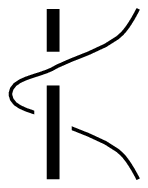}
\end{center}
\end{figure}

Given what has already been established for BMW algebras, it is then conceivable that the cyclotomic and affine BMW algebras be isomorphic to appropriate analogues of the
Kauffman tangle algebras. Indeed, by utilising the results and techniques of Morton and Wasserman \cite{MW89} for the ordinary BMW algebras, this was shown to be the case for the affine version, over an arbitrary ring, by Goodman and Hauschild in \cite{GH06}. The topological realisation of the affine BMW algebra is as an algebra of (regular isotopy equivalence classes of) ``affine'' tangles on $n$ strands in the annulus cross the interval (that is, the solid torus) modulo the Kauffman skein relations. A precise definition is given in Chapter \ref{chap:basis}. In proving this isomorphism, they also obtain a basis analogous to a well-known basis of the affine Hecke algebras.

This thesis is concerned with the study of H\"aring-Oldenburg's cyclotomic BMW algebras. In Chapter \ref{chap:bnkintro}, we introduce these algebras and derive some straightforward identities and formulas pertinent to the next chapter.
A natural problem to address first is whether these algebras are always free. We expect that the cyclotomic BMW algebras $\bnk$ should be free of rank $k^n(2n-1)!!$.
Chapter \ref{chap:spanning} is concerned with obtaining a spanning set of $\bnk$ of this cardinality.

Many difficulties, in particular regarding linear independency, arise in the case of the cyclotomic BMW algebras. Due to the $k^\mathrm{th}$ order polynomial relation imposed on the extra generator $Y$, one can easily obtain torsion on elements associated with certain tangles on two strands. In order to fix this problem, this suggests additional ``admissibility'' assumptions should be imposed on the parameters of our ground ring. In Chapter \ref{chap:admissible}, we devote our study to the representation theory of $\bnk$ at the $n=2$ level to determine precisely the form of these assumptions.  We give these admissibility conditions explicitly and provide a ``generic'' (or universal) ground ring $R_0$, in the sense that for any ring $R$ with admissible parameters (see Definition \ref{defn:adm}) there is a unique map $R_0 \rightarrow R$ which respects the parameters. 
Chapter \ref{chap:admissible} contains results specific to the $n=2$ algebra $\bt$ proven in \cite{WY06} by Wilcox and the author, under a slightly stronger notion of admissibility. A particular result in \cite{WY06} shows that admissibility ensures the freeness of the algebra $\bt$ over R. These results are stated but incompletely proved in H\"aring-Oldenburg \cite{HO01}.

Goodman and Hauschild Mosley \cite{GH107} attempt to follow the same type of arguments they used in \cite{GH06} to establish freeness of $\bnk$ and show that the cyclotomic BMW algebra is isomorphic to a cyclotomic version of the Kauffman tangle algebra. Using diagrammatical arguments, they obtain a second basis of the affine BMW algebra which then restricts naturally onto a spanning set of $\bnk$, different to ours, also of cardinality $k^n (2n-1)!!$. However, they withdrew their preprint due to issues with their generic ground ring crucial to their linear independence arguments. It turns out that the generic ground ring $R_0$ constructed in Chapter \ref{chap:admissible} fulfills the properties required for their arguments. We also discuss in Chapter \ref{chap:admissible} the relationship between our notion of admissibility and that used by Goodman and Hauschild Mosley \cite{GH107}.

In Chapter \ref{chap:basis}, we then follow a similar path to that in \cite{GH107} to prove freeness of $\bnk$ over our alternative ground ring $R_0$ and a basis theorem for $\bnk$ over a general ring $R$ with admissible parameters, hence showing $\bnk$ is $R$-free of rank $k^n (2n-1)!!$. In proving this, we also establish an isomorphism between the cyclotomic BMW algebras and the cyclotomic Kauffman tangle algebras defined in Chapter \ref{chap:basis}.

In Chapter \ref{chap:cellularity}, we investigate the cellularity of $\bnk$.
The Brauer, Iwahori-Hecke, Ariki-Koike and BMW algebras are all cellular algebras, in the sense of Graham and Lehrer \cite{GL96}. The theory of cellular algebras provides a unified axiomatic framework for understanding several important algebras, including the (non-semisimple specialisations of) Iwahori-Hecke algebras. In \cite{GL96}, Graham and Lehrer show that cellular algebras have naturally defined \emph{cell representations} whose structure depends on certain symmetric bilinear forms. Given a cellular algebra, they obtain a general description of its irreducible representations and block theory as well as a criterion for semisimplicity.
Using a known cellular basis of the Ariki-Koike algebras $\ak{n}$, we first obtain an appropriate ``lifting'' of a slight modification of this basis into $\bnk$ which is compatible with a certain anti-involution of $\bnk$. Then using this new basis we prove $\bnk$ is a cellular algebra.

\chapter{The Cyclotomic BMW Algebras} \label{chap:bnkintro}

In this chapter we introduce the cyclotomic Birman-Murakami-Wenzl algebras $\bnk$, as defined by H\"aring-Oldenburg in \cite{HO01}. As seen in Definition \ref{defn:bnk} below, the defining relations of the algebra $\bnk$ consist of the defining relations of the BMW algebra $\bmw_n$ (see Definition \ref{defn:bmw}) and further relations involving an extra generator $Y$ which satisfies a polynomial relation of order $k$. 
Through straightforward calculations and induction arguments, we establish several useful formulas and identities between special elements of the algebra. These results will then be used extensively in the next chapter, which will involve many lengthy manipulations of certain products in the algebra. Throughout let us fix natural numbers $n$ and $k$. 
\vspace{-2mm}
\begin{defn} \label{defn:bnk}
Let $R$ be a unital commutative ring 
containing units $A_0, q_0, q,\l$ and further 
elements $q_1, \ldots, q_{k-1}$ and $A_1, \ldots, A_{k-1}$ such that $\l - \l^{-1} = \d(1-A_0)$ holds, where $\d := q - q^{-1}$. \\ The \emph{\textbf{cyclotomic BMW algebra}} 
$\bnk := \bnk(q,\l,A_i,q_i)$ is the unital associative $R$-algebra 
generated by $Y^{\pm 1}, X_1^{\pm 1}, \ldots, X_{n-1}^{\pm 1}$ and 
$e_1, \ldots, e_{n-1}$ subject to the following relations, which hold for all possible values of $i$ unless otherwise stated.
\begin{eqnarray}
X_i - X_i^{-1} &=& \d(1-e_i) \label{eqn:ix} \\
X_iX_j &=& X_jX_i \qquad\qquad\qquad\quad \mspace{-3mu} \text{for } |i-j| \geq 2 \label{eqn:braid2} \\
X_iX_{i+1}X_i &=& X_{i+1}X_iX_{i+1}  \label{eqn:braid1}\\
X_ie_j &=& e_jX_i \ \qquad\qquad\qquad\quad \text{for } |i-j| \geq 2 \label{eqn:xecomm} \\
e_ie_j &=& e_je_i \qquad\qquad\qquad\quad\,\,\,\mspace{2mu} \text{for } |i-j| \geq 2 \label{eqn:eecomm}
\end{eqnarray}
\begin{eqnarray}
X_ie_i &=& e_iX_i \,\,\,=\,\,\, \l e_i \label{eqn:untwist} \\
X_i X_j e_i &=& e_j e_i \,\,\,=\,\,\, e_j X_i X_j \mspace{1mu}\,\qquad \text{for } |i-j| = 1 \label{eqn:xxe} \\
e_i e_{i\pm 1} e_i &=& e_i \label{eqn:eee} \\
e_i^2 &=& A_0e_i \\
Y^k \= \sum_{i=0}^{k-1} q_i Y^i \label{eqn:ycyclo} \\
X_1YX_1Y &=& YX_1YX_1  \label{eqn:braidb}\\
YX_i &=& X_iY \qquad\qquad\qquad\quad\mspace{5mu} \text{for } i > 1 \label{eqn:ycommx}\\
Ye_i &=& e_iY \qquad\qquad\qquad\quad\mspace{12mu} \text{for } i > 1 \label{eqn:ycomme}\\
YX_1Ye_1 &=& \lambda^{-1}e_1 \,\,\,=\,\,\, e_1YX_1Y \label{eqn:eyxy} \\
e_1Y^me_1 &=& A_m e_1 \qquad\qquad\qquad\quad\mspace{-1mu} \text{for } 0 \leq m \leq k-1. \label{eqn:eye}
\end{eqnarray}
\end{defn}
\textbf{Remark:} Observe that, by relations (\ref{eqn:ix}) and (\ref{eqn:ycyclo}), it is unnecessary to include the inverses of $Y$ and $X$ as generators of $\bnk$ in Definition \ref{defn:bnk}. 

The definition of $\bnk$ given here is a slight modification of the original definition given by H\"aring-Oldenburg \cite{HO01}, in which the $k^\mathrm{th}$ order polynomial relation (\ref{eqn:ycyclo}) on the generator $Y$ is $\prod_{i=0}^{k-1}(Y-p_i) = 0$, where the $p_i$ are units in the ground ring $R$. Under this stronger relation, the $q_i$ in relation (\ref{eqn:ycyclo}) would then be the signed elementary symmetric polynomials in the $p_i$, where $q_0 = (-1)^{k-1} \prod_i p_i$ is invertible.
However, we need not impose this stronger polynomial relation on $Y$ (that is, it is not necessary for us to assume that $\sum_{j=0}^k q_j y^j$ splits in $R$), until Chapter \ref{chap:cellularity}.

Define $q_k := -1$. Then $\sum_{j=0}^k q_j Y^j = 0$ and the inverse of $Y$ may then be expressed as a linear combination of non-negative powers of $Y$ as follows:
\[
Y^{-1} = - q_0^{-1} \sum_{i=0}^{k-1} q_{i+1}  Y^i.
\]

Using the defining $k^\mathrm{th}$ order relation on $Y$ and (\ref{eqn:eye}), there exists elements $A_{m}$ of $R$, \emph{for all} $m \in \mathbb{Z}$, such that 
\begin{equation} \label{eqn:negeye}
e_1 Y^{m} e_1 = A_{m} e_1.
\end{equation}
We will see later that, in order for our algebras to be ``well-behaved'', the $A_m$ cannot be chosen independently of the other parameters of the algebra.

\noi Observe that there is an unique anti-involution $^*: \bnk \rightarrow \bnk$ such that
\renewcommand{\theequation}{$\ast$}
\begin{equation} \label{eqn:star}
  Y^* = Y, \quad X_i^* = X_i \quad \text{and} \quad e_i^* = e_i,
\end{equation}
for every $i = 1,\ldots,n-1$. Here an anti-involution always means an involutary $R$-algebra anti-automorphism. \\
Also the following defines an $R$-algebra isomorphism \[\bnk(q,\l,A_i,q_i) \rightarrow \bnk(q^{-1},\l^{-1},A_{-i},-q_{k-i}q_0^{-1}):\]
\[  
Y \mapsto \iy, \quad X_i \mapsto \ix, \quad e_i \mapsto e_i.
\]

\noi For all $i = 1,\ldots, n$, define the following elements of $\bnk$:
\[ Y_i':= X_{i-1} \ldots X_2 X_1 Y X_1 X_2 \ldots X_{i-1}.\]
Observe that these elements are fixed under the (\ref{eqn:star}) anti-involution.

\renewcommand{\theequation}{\arabic{equation}}
\addtocounter{equation}{-1}
We now establish several identities in the algebra which will be used frequently in future proofs.
Let us fix $n$ and $k$. The following calculations are valid over a general ring $R$ with any choice of the above parameters $A_0, \o, A_{k-1}, q_0, \o, q_{k-1}, q, \l$. 

\begin{prop} \label{prop:conseq}
The following relations hold in $\bnk$: 
\begin{enumerate}
\item[(a)]
For all $i$,
\begin{equation} \label{eqn:xi2}
  X_i^2 = 1+ \d X_i - \d \l e_i
\end{equation}
and
\begin{equation} \label{eqn:exe} 
e_iX_{i\pm 1}e_i = \l^{-1} e_i. 
\end{equation}
\item[(b)]
For all $j$ and $i \neq j$ or $j-1$, 
\begin{equation} \label{eqn:prop1b}
  X_i Y_j' = Y_j' X_i \quad\text{and}\quad e_i Y_j' = Y_j' e_i. 
\end{equation}
\item[(c)]
For all $i$ and $j$, 
\begin{equation} \label{eqn:prop1c}
Y_i' Y_j' = Y_j' Y_i'.
\end{equation} 
\item[(d)] 
For all $i$, 
\begin{equation} \label{eqn:prop1d}
Y_i' X_i Y_i' e_i = \l^{-1} e_i = e_i Y_i' X_i Y_i'.
\end{equation}
\item[(e)]
For all $i$ and $p$,
  \begin{equation} \label{eqn:m9}
    e_i \y{i+1}{p} = e_i \y{i}{-p} \quad \text{ and } \quad \y{i+1}{p} e_i = \y{i}{-p} e_i. 
  \end{equation}
\end{enumerate}
\end{prop}

\begin{proof}
The quadratic relation in part (a) follows by multiplying relation (\ref{eqn:ix}) by $X_i$ and applying relation (\ref{eqn:untwist}) to simplify. Equation (\ref{eqn:exe}) is proved below.
\begin{align*}
  e_i X_{i\pm 1} e_i &\rel{eqn:untwist} \l^{-1} e_i X_{i+1} X_i e_i \\
&\rel{eqn:xxe} \l^{-1} e_i e_{i+1} e_i \\
&\rel{eqn:eee} \l^{-1} e_i.
\end{align*}

\noi The first equation in part (b) follows from the braid relations (\ref{eqn:braid2}), (\ref{eqn:braid1}) and (\ref{eqn:ycommx}) and the second follows from relations (\ref{eqn:xecomm}), (\ref{eqn:xxe}) and (\ref{eqn:ycomme}).

\noi Part (c) follows from part (b) and the braid relation (\ref{eqn:braidb}). 

\noi We prove (d) by induction on $i \geq 1$. The case where $i=1$ is simply relation (\ref{eqn:eyxy}). Now assume (d) holds for a fixed $i$. Then
\begin{eqnarray*}
Y_{i+1}' X_{i+1} Y_{i+1}' e_{i+1} 
&=& X_i Y_i' X_i X_{i+1} X_i Y_i' X_i e_{i+1} \\
&\stackrel{(\ref{eqn:braid1})}{=}&  X_i Y_i' X_{i+1} X_i X_{i+1} Y_i' X_i e_{i+1} \\
&\rel{eqn:prop1b}&  X_i X_{i+1} Y_i' X_i Y_i' X_{i+1} X_i e_{i+1} \\
&\rel{eqn:xxe}& X_i X_{i+1} Y_i' X_i Y_i' e_i e_{i+1} \\
&\stackrel{\text{ind. hypo.}}{=}& \l^{-1} X_i X_{i+1} e_i e_{i+1} \\
&\stackrel{(\ref{eqn:xxe}),(\ref{eqn:eee})}{=}& \l^{-1} e_{i+1}. 
\end{eqnarray*}
The second equality of part (d) now follows immediately by applying the anti-involution (\ref{eqn:star}) to the first.
Moreover, part (e) follows from parts (c) and (d), remembering that $Y_{j+1}' = X_j Y_j' X_j$.
\end{proof}
The most important and useful property of the $Y_i'$ is their pairwise commutativity. 
Here are some useful identities involving the $Y_i'$, $X_i$ and $e_i$ which shall be used extensively throughout later proofs.
\begin{prop} \label{prop:XXp1}
The following equations hold for all $i$:
\begin{eqnarray}
e_i e_{i+1} e_{i+2} \gamma_i \= \gamma_{i+2} e_i e_{i+1} e_{i+2}, \quad \text{where $\gamma_i = X_i,e_i$ or $Y_i'$}; \label{eqn:ep2} \\
X_i X_{i+1} \gamma_i  \= \gamma_{i+1} X_i X_{i+1}, \quad \text{where $\gamma_i = X_i$ or $e_i$}. \label{eqn:XXp1}
\end{eqnarray} 
\end{prop}
\begin{proof}
We prove (\ref{eqn:ep2}) in the three stated cases.
\begin{eqnarray*}
e_i e_{i+1} e_{i+2} X_i &\rel{eqn:xecomm}& e_i e_{i+1} X_i e_{i+2} \\
&\stack{(\ref{eqn:xxe}),(\ref{eqn:eee})}& e_i X_{i+1}^{-1} e_{i+2} \\
&\stack{(\ref{eqn:xxe}),(\ref{eqn:eee})}& e_i X_{i+2} e_{i+1} e_{i+2} \\
&\rel{eqn:xecomm}& X_{i+2} e_i e_{i+1} e_{i+2}. 
\end{eqnarray*}
\begin{align*}
e_i e_{i+1} e_{i+2} e_i &\rel{eqn:eecomm} e_i e_{i+1} e_i e_{i+2} \\
&\rel{eqn:eee} e_i e_{i+2} \\
&\rel{eqn:eee} e_i e_{i+2} e_{i+1} e_{i+2} \\
&\rel{eqn:eecomm} e_{i+2} e_i e_{i+1} e_{i+2}. 
\end{align*}
\begin{align*}
e_i e_{i+1} e_{i+2} Y_i' &\rel{eqn:prop1b} e_i Y_i' e_{i+1} e_{i+2} \\
&\rel{eqn:m9} e_i \y{i+1}{-1} e_{i+1}e_{i+2} \\
&\rel{eqn:m9} e_i Y_{i+2}' e_{i+1}e_{i+2} \\
&\rel{eqn:prop1b} Y_{i+2}' e_i e_{i+1} e_{i+2}. 
\end{align*}
Equation (\ref{eqn:XXp1}) follows clearly from relations (\ref{eqn:braid1}) and (\ref{eqn:xxe}). 
\end{proof}
%Here are some more useful properties involving the $Y_i'$.
\begin{lemma}\label{lemma:mult}
The following hold for any $i$ and non-negative integer $p$: 
\begin{align}
X_i \y{i}{p} &= \y{i+1}{p} X_i - \d \pum{s} \y{i+1}{s} \y{i}{p-s} + \d \pum{s} \y{i+1}{s} e_i \y{i}{p-s} \label{eqn:magic} \\
X_i  \y{i}{-p} &= \y{i+1}{-p} X_i + \d \pum{s} \y{i+1}{s-p} \y{i}{-s} - \d \pum{s} \y{i+1}{s-p} e_i \y{i}{-s} \label{eqn:m2} \\
\ix \y{i}{p} &= \y{i+1}{p} \ix - \d \pum{s} \y{i+1}{p-s} \y{i}{s} + \d \pum{s} \y{i+1}{p-s} e_i \y{i}{s} \label{eqn:m3} \\
\ix \y{i}{-p} &= \y{i+1}{-p} \ix + \d \pum{s} \y{i+1}{-s} \y{i}{-(p-s)} - \d \pum{s} \y{i+1}{-s} e_i \y{i}{-(p-s)} \label{eqn:m4} \\
X_i \y{i+1}{p} &= \y{i}{p} X_i + \d \pum{s} \y{i}{p-s} \y{i+1}{s} - \d \pum{s} \y{i}{p-s} e_i \y{i+1}{s} \label{eqn:m5} 
\end{align}
\begin{align}
X_i \y{i+1}{-p} &= \y{i}{-p} X_i - \d \pum{s} \y{i}{-s} \y{i+1}{s-p} + \d \pum{s} \y{i}{-s} e_i \y{i+1}{s-p} \label{eqn:m6} \\
\ix \y{i+1}{p} &= \y{i}{p} \ix + \d \pum{s} \y{i}{s} \y{i+1}{p-s} - \d \pum{s} \y{i}{s} e_i \y{i+1}{p-s} \label{eqn:m7} \\
\ix \y{i+1}{-p} &= \y{i}{-p} \ix - \d \pum{s} \y{i}{-(p-s)} \y{i+1}{-s} + \d \pum{s} \y{i}{-(p-s)} e_i \y{i+1}{-s} \label{eqn:m8} \\
X_i \y{i}{p} X_i &= \y{i+1}{p} - \d \sum_{s=1}^{p-1} \y{i+1}{s} \y{i}{p-s} X_i + \d \sum_{s=1}^{p-1} \y{i+1}{s} e_i \y{i}{p-s} X_i \label{eqn:m10} \\
X_i \y{i}{p} X_i &= \y{i+1}{p} - \d \sum_{s=1}^{p-1} X_i \y{i}{s} \y{i+1}{p-s} + \d \sum_{s=1}^{p-1} X_i \y{i}{s} e_i \y{i+1}{p-s} \label{eqn:m11} \\
X_i \y{i}{-p} X_i &= \y{i+1}{-p} + \d \sum_{s=0}^{p} \y{i+1}{s-p} \y{i}{-s} X_i - \d \sum_{s=0}^{p} \y{i+1}{s-p} e_i \y{i}{-s} X_i \label{eqn:m12} \\
X_i \y{i}{-p} X_i &= \y{i+1}{-p} + \d \sum_{s=0}^{p} X_i \y{i}{-s} \y{i+1}{s-p} - \d \sum_{s=0}^{p} X_i \y{i}{-s} e_i \y{i+1}{s-p}. \label{eqn:m13} 
\end{align}
\end{lemma}
\begin{proof}
We obtain the first equation through the following straightforward calculation. For all $p \geq 0$,
\be 
  X_i \y{i}{p} &\p Y'_{i+1} \ix \y{i}{p-1} \\ 
  &\rel{eqn:ix} Y'_{i+1} X_i \y{i}{p-1} - \d Y'_{i+1} \y{i}{p-1} + \d Y'_{i+1} e_i \y{i}{p-1} \\ 
&\p \y{i+1}{2}X_i \y{i}{p-2} - \d \y{i+1}{2} \y{i}{p-2} + \d \y{i+1}{2} e_i \y{i}{p-2} \\
&\phantom{\p} \phantom{\y{i+1}{2}X_i \y{i}{p-2}} \,\,\,- \d Y'_{i+1} \y{i}{p-1} + \d Y'_{i+1} e_i \y{i}{p-1} \\
&\p \ldots \, = \\
&\p \y{i+1}{p-1} X_i Y_i' - \d \sum_{s=1}^{p-1} \y{i+1}{s} \y{i}{p-s} + \d \sum_{s=1}^{p-1} \y{i+1}{s} e_i \y{i}{p-s} \\
&\p \y{i+1}{p} X_i - \d \sum_{s=1}^{p} \y{i+1}{s} \y{i}{p-s} + \d \sum_{s=1}^{p} \y{i+1}{s} e_i \y{i}{p-s},
\end{align*}
proving equation (\ref{eqn:magic}).
Multiplying equation (\ref{eqn:magic}) on the left by $\y{i+1}{-p}$ and the right by $\y{i}{-p}$ and rearranging gives equation (\ref{eqn:m2}). 
Applying (\ref{eqn:star}) to equations (\ref{eqn:magic}) and (\ref{eqn:m2}) and rearranging  then produces equations (\ref{eqn:m5}) and (\ref{eqn:m6}), respectively. By using (\ref{eqn:ix}) and a simple change of the summation index $s \mapsto p-s$, one easily obtains equations (\ref{eqn:m3}), (\ref{eqn:m4}), (\ref{eqn:m7}) and (\ref{eqn:m8}) from equations (\ref{eqn:magic}), (\ref{eqn:m2}), (\ref{eqn:m5}) and (\ref{eqn:m6}), respectively.

\noi Using equations (\ref{eqn:magic}) and (\ref{eqn:xi2}), we obtain
\begin{align*}
X_i \y{i}{p} X_i &= \y{i+1}{p} X_i^2 - \d \pum{s} \y{i+1}{s} \y{i}{p-s} X_i + \d \pum{s} \y{i+1}{s} e_i \y{i}{p-s} X_i \\
&= \y{i+1}{p} - \d \sum_{s=1}^{p-1} \y{i+1}{s} \y{i}{p-s} X_i + \d \sum_{s=1}^{p-1} \y{i+1}{s} e_i \y{i}{p-s} X_i,
\end{align*}
proving equation (\ref{eqn:m10}). Applying (\ref{eqn:star}) to (\ref{eqn:m10}) and a straightforward change of summation now gives equation (\ref{eqn:m11}). Similarly, using equation (\ref{eqn:m2}) and (\ref{eqn:star}), one obtains equations (\ref{eqn:m12}) and (\ref{eqn:m13}).
\end{proof}

\begin{lemma} \label{lemma:eype}
For all integers $p$, the following hold:
\begin{enumerate}
  \item[(I)] $e_i \y{i}{p} e_i \in \la{Y^{s_1} \y{2}{s_2} \ldots \y{i-1}{s_{i-1}} e_i}$;
  \item[(II)] $X_i \y{i}{p} e_i \in \la{Y^{s_1} \y{2}{s_2} \ldots \y{i-1}{s_{i-1}} \y{i}{s_i} e_i \,\big|\,  |s_i| \leq |p|\,}$;
  \item[(III)] $e_i \y{i}{p} X_i \in \la{e_i Y^{s_1} \y{2}{s_2} \ldots \y{i-1}{s_{i-1}} \y{i}{s_i} \,\big|\,  |s_i| \leq |p|\,}$, where $\la{J}$ denotes the $R$-span of $J$.
\end{enumerate}
\end{lemma}

\begin{proof}
The relation (\ref{eqn:eye}) and the $k^\mathrm{th}$ order relation on $Y$ tells us that, for any integer $p$, $e_1 Y^{p} e_1$ is always a scalar multiple of $e_1$, hence showing part (I) of the lemma for the case $i=1$.

\noi Now, for all $p \geq 0$, equation (\ref{eqn:magic}) implies that
\begin{eqnarray*}
X_1 Y^p e_1 &=& \y{2}{p} X_1 e_1 - \d \pum{s} \y{2}{s}Y^{p-s} e_1 + \d \pum{s} \y{2}{s} e_1 Y^{p-s} e_1 \\
&\rel{eqn:untwist}& \l \y{2}{p} e_1 - \d \pum{s} \y{2}{s}Y^{p-s} e_1 + \d \pum{s} \y{2}{s} e_1 Y^{p-s} e_1 \\
&\stack{(\ref{eqn:m9}),(\ref{eqn:eye})}& \l Y^{-p} e_1 - \d \pum{s} Y^{p-2s} e_1 + \d \pum{s} A_{p-s} Y^{-s} e_1.
\end{eqnarray*}

\noi Similarly, by equations (\ref{eqn:m2}), (\ref{eqn:untwist}) and (\ref{eqn:negeye}),
\[
X_1 Y^{-p} e_1 = \l Y^p e_1 + \d \pum{s} Y^{p-2s} e_1 - \d \pum{s} A_{-s} Y^{p-s} e_1 \text{, for all }p \geq 0.
\]

\noi Observe that $|\!-\! p\kern1.1pt| = |p|$ and when $1 \leq s \leq p$, we have $|s|$,$|p-s|$, $|p-2s| \leq |p|$. Hence $X_1 Y^p e_1 \in \la{Y^m e_1 \,\big|\, |m| \leq |p|\,}$, for all $p \in \mathbb{Z}$, proving part (II) of the lemma for the case $i=1$.

We are now able to prove (I) and (II), for all integers $p \geq 0$, together by induction on $i$, which will in turn involve inducting on $p \geq 0$. Both hold trivially for $p = 0$, since $e_i^2 = A_0 e_i$ and $X_i e_i = \l e_i$, for all $i$. \\ Now let us assume that: $X_{i-1} \y{i-1}{r} e_{i-1} \in \la{Y^{s_1} \y{2}{s_2} \ldots \y{i-2}{s_{i-2}} \y{i-1}{s_{i-1}} e_{i-1} \,\big|\,  |s_{i-1}| \leq |r|\,}$ and \\ $e_{i-1} \y{i-1}{r} e_{i-1} \in \la{Y^{s_1} \y{2}{s_2} \ldots \y{i-2}{s_{i-2}} e_{i-1}}$, for all $r \geq 0$, and
$e_{i} \y{i}{r} e_{i} \in \la{Y^{s_1} \y{2}{s_2} \ldots \y{i-1}{s_{i-1}} e_{i}}$ and $X_i \y{i}{r} e_i \in \la{Y^{s_1} \y{2}{s_2} \ldots \y{i-1}{s_{i-1}} \y{i}{s_i} e_i \,\big|\,  |s_i| \leq |r|\,}$, for all $r < p$. 

\noi For all $p > 0$,
\begin{align}
e_i \y{i}{p} e_i &\h e_i X_{i-1} Y_{i-1}' X_{i-1} \y{i}{p-1} e_i \nonumber \\
&\rel{eqn:m5} e_i X_{i-1} \y{i-1}{p} X_{i-1} e_i + \d \sum_{s=1}^{p-1} e_i X_{i-1} \y{i-1}{p-s} \y{i}{s} e_i - \d \sum_{s=1}^{p-1} e_i X_{i-1} \y{i-1}{p-s} e_{i-1} \y{i}{s} e_i \nonumber \\
&\kern.18em\rel{eqn:ix} e_i X_{i-1} \y{i-1}{p} X_{i-1}^{-1} e_i + \d \ppum{s} e_i X_{i-1} \y{i-1}{p-s} \y{i}{s} e_i - \d \ppum{s} e_i X_{i-1} \y{i-1}{p-s} e_{i-1} \y{i}{s} e_i \nonumber \\
&\kern.18em\rel{eqn:xxe} e_i e_{i-1} \ix \y{i-1}{p} X_{i-1}^{-1} e_i + \d \ppum{s} e_i X_{i-1} \y{i-1}{p-s} \y{i}{s} e_i - \d \ppum{s} e_i X_{i-1} \y{i-1}{p-s} e_{i-1} \y{i}{s} e_i \nonumber \\
&\rel{eqn:prop1b} e_i e_{i-1} \y{i-1}{p} \ix X_{i-1}^{-1} e_i + \d \ppum{s} e_i X_{i-1} \y{i-1}{p-s} \y{i}{s} e_i - \d \ppum{s} e_i X_{i-1} \y{i-1}{p-s} e_{i-1} \y{i}{s} e_i \nonumber \\
&\h e_i e_{i-1} \y{i-1}{p} e_{i-1} e_i + \d \ppum{s} e_i X_{i-1} \y{i}{s} e_i \y{i-1}{p-s}  - \d \ppum{s} e_i X_{i-1} \y{i-1}{p-s} e_{i-1} \y{i}{s} e_i, \label{eqn:eype}
\end{align}
by relation (\ref{eqn:xxe}) and Proposition \ref{prop:conseq}.

\noi Let us consider the first term in the latter equation above. By induction on $i$, 
\[
e_{i-1} \y{i-1}{p} e_{i-1} \in \la{Y^{s_1} \y{2}{s_2} \ldots \y{i-2}{s_{i-2}} e_{i-1}}.
\]
Therefore, by relation (\ref{eqn:eee}),
\[
e_i e_{i-1} \y{i-1}{p} e_{i-1} e_i \in \la{Y^{s_1} \y{2}{s_2} \ldots \y{i-2}{s_{i-2}} e_i}.
\]
Now let us consider the second term in the RHS of (\ref{eqn:eype}). Fix $0 \leq s \leq p-1$.

\begin{align*}
e_i X_{i-1} \y{i}{s} e_i \y{i-1}{p-s}
&\rel{eqn:xxe} e_i e_{i-1} \ix \y{i}{s} e_i \y{i-1}{p-s} \\
&\rel{eqn:ix} e_i e_{i-1} X_i \y{i}{s} e_i \y{i-1}{p-s} - \d e_i e_{i-1} \y{i}{s} e_i \y{i-1}{p-s} + \d e_i e_{i-1} e_i \y{i}{s} e_i \y{i-1}{p-s}. 
\end{align*}

\noi By induction on $p$ and equation (\ref{eqn:prop1b}),
\begin{align*}
& e_i e_{i-1} X_i \y{i}{s} e_i \y{i-1}{p-s} \\
\,\,&\stackrel{\phantom{(19)}}{\in} \la{e_i e_{i-1} Y^{m_1} \y{2}{m_2} \ldots \y{i-1}{m_{i-1}} \y{i}{m_i} e_i \y{i-1}{p-s} \,\big|\, |m_i| \leq |s|\,} \\
&\h \la{e_i Y^{m_1} \y{2}{m_2} \ldots \y{i-2}{m_{i-2}} e_{i-1} \y{i-1}{m_{i-1}} \y{i}{m_i} e_i \y{i-1}{p-s} \,\big|\, |m_i| \leq |s|\,} \\
&\rel{eqn:m9} \la{e_i Y^{m_1} \y{2}{m_2} \ldots \y{i-2}{m_{i-2}} e_{i-1} \y{i-1}{m_{i-1}-m_i} e_i \y{i-1}{p-s} \,\big|\, |m_i| \leq |s|\,} \\
&\h \la{Y^{m_1} \y{2}{m_2} \ldots \y{i-2}{m_{i-2}} e_i e_{i-1} e_i \y{i-1}{p-s+m_{i-1}-m_i} \,\big|\, |m_i| \leq |s|\,}.
\end{align*}

\noi Therefore
\[
e_i e_{i-1} X_i \y{i}{s} e_i \y{i-1}{p-s} \in \la{Y^{m_1} \y{2}{m_2} \ldots \y{i-2}{m_{i-2}} \y{i-1}{p-s+m_{i-1}-m_i} e_i}.
\]

\noi Also, by (\ref{eqn:m9}), (\ref{eqn:prop1b}) and (\ref{eqn:eee}),
\[
e_i e_{i-1} \y{i}{s} e_i \y{i-1}{p-s} = e_i \y{i-1}{p-2s}.
\]

\noi Moreover, by induction on $p$,
\begin{align*}
e_i e_{i-1} e_i \y{i}{s} e_i \y{i-1}{p-s} 
&\rel{eqn:eee} e_i \y{i}{s} e_i \y{i-1}{p-s} \\ 
&\stackrel{\phantom{(1)}}{\in} \la{Y^{m_1} \y{2}{m_2} \ldots \y{i-1}{m_{i-1}+p-s} e_i}.
\end{align*}

\noi Thus, for all $0 \leq s \leq p-1$,
\[
e_i X_{i-1} \y{i}{s} e_i \y{i-1}{p-s} \in \la{Y^{s_1} \y{2}{s_2} \ldots \y{i-1}{s_{i-1}} e_i}.
\]
Hence the second term in the RHS of equation (\ref{eqn:eype}) is in $\la{Y^{s_1} \y{2}{s_2} \ldots \y{i-1}{s_{i-1}} e_i}$.

Finally, by induction on $i$ and using (\ref{eqn:m9}), (\ref{eqn:prop1b}) and (\ref{eqn:eee}),
\be
e_i X_{i-1} \y{i-1}{p-s} e_{i-1} \y{i}{s} e_i
&\in \la{e_i Y^{m_1} \y{2}{m_2} \ldots \y{i-1}{m_{i-1}} e_{i-1} \y{i}{s} e_i \,\big|\, |m_{i-1}| \leq |p-s|\,} \\
&\in \la{Y^{m_1} \y{2}{m_2} \ldots \y{i-1}{m_{i-1}-s} e_i \,\big|\, |m_{i-1}| \leq |p-s|\,}.
\end{align*}
Thus the third term in the RHS of equation (\ref{eqn:eype}) is in $\la{Y^{s_1} \y{2}{s_2} \ldots \y{i-1}{s_{i-1}} e_i}$.

Also, for all $p \geq 0$, equation (\ref{eqn:magic}) implies that 
\begin{eqnarray*}
  X_i \y{i}{p} e_i \= \y{i+1}{p} X_i e_i - \d \pum{s} \y{i+1}{s} \y{i}{p-s} e_i + \d \pum{s} \y{i+1}{s} e_i \y{i}{p-s} e_i \\ 
&\stackrel{(\ref{eqn:untwist}),(\ref{eqn:m9})}{=}& \l \y{i}{-p} e_i - \d \pum{s} \y{i}{p-2s} e_i + \d \pum{s} \y{i}{-s} e_i \y{i}{p-s} e_i.
\end{eqnarray*}
The first term above is clearly in $\la{Y^{s_1} \y{2}{s_2} \ldots \y{i-1}{s_{i-1}} \y{i}{s_i} e_i \,\big|\, |s_i| \leq |p|\,}$, since $|\!-\! p\kern1pt| = |p\,|$. \\
Regarding the second term above, since $1 \leq s \leq p$, $|p-2s| \leq |p|$, so it is also an element of $\la{Y^{s_1} \y{2}{s_2} \ldots \y{i-1}{s_{i-1}} \y{i}{s_i} e_i \,\big|\, |s_i| \leq |p|\,}$.
Moreover, we have $0 \leq p-s \leq p-1 < p$, so by induction on $p$,
\begin{align*}
  \y{i}{-s} e_i \y{i}{p-s} e_i 
  &\in \la{\y{i}{-s} Y^{m_1} \y{2}{m_2} \ldots \y{i-1}{m_{i-1}} e_i} \\
&\subseteq \la{Y^{s_1} \y{2}{s_2} \ldots \y{i-1}{s_{i-1}} \y{i}{s_i} e_i \,\big|\, |s_i| \leq |p|\,}.
\end{align*}
Therefore, for all $p \geq 0$,
\[
X_i \y{i}{p} e_i \in \la{Y^{s_1} \y{2}{s_2} \ldots \y{i-1}{s_{i-1}} \y{i}{s_i} e_i \,\big|\, |s_i| \leq |p|\,}
\]
and
\[
e_i \y{i}{p} e_i \in \la{Y^{s_1} \y{2}{s_2} \ldots \y{i-1}{s_{i-1}} e_i}.
\]

Let us denote $\dagger: \bnk(q^{-1},\l^{-1},A_{-i},-q_{k-i}q_0^{-1}) \rightarrow \bnk(q,\l,A_i,q_i)$ to be the isomorphism of $R$-algebras defined by 
\[
  Y \mapsto \iy, \quad X_i \mapsto \ix, \quad e_i \mapsto e_i.
\]
Note that $\dagger$ maps $Y_i'$ to its inverse. \\
We have shown above that, for all $p \geq 0$, $e_i \y{i}{p} e_i \in \la{Y^{s_1} \y{2}{s_2} \ldots \y{i-1}{s_{i-1}} e_i}$, as an element of $\bnk(q^{-1},\l^{-1},A_{-i},-q_{k-i}q_0^{-1})$. Therefore, using $\dagger$,
\begin{equation} \label{eqn:eypeneg}
e_i \y{i}{-p} e_i \in \la{Y^{s_1} \y{2}{s_2} \ldots \y{i-1}{s_{i-1}} e_i},
\end{equation}
as an element of $\bnk(q,\l,A_i,q_i)$, for all $p \geq 0$. \\
Furthermore, our previous work also shows that, as an element of $\bnk(q^{-1},\l^{-1},A_{-i},-q_{k-i}q_0^{-1})$,
\[
X_i \y{i}{p} e_i \in \la{Y^{s_1} \y{2}{s_2} \ldots \y{i-1}{s_{i-1}} \y{i}{s_i} e_i \,\big|\, |s_i| \leq |p|\,}.
\]
Thus, applying $\dagger$,
\[
\ix \y{i}{-p} e_i \in \la{Y^{s_1} \y{2}{s_2} \ldots \y{i-1}{s_{i-1}} \y{i}{s_i} e_i \,\big|\, |s_i| \leq |p|\,},
\]
as an element of $\bnk(q,\l,A_i,q_i)$.

\noi However, by relation (\ref{eqn:ix}), $\ix \y{i}{-p} e_i = X_i \y{i}{-p} e_i - \d \y{i}{-p} e_i + \d e_i \y{i}{-p} e_i$. By (\ref{eqn:eypeneg}), the last two terms are clearly in $\la{Y^{s_1} \y{2}{s_2} \ldots \y{i-1}{s_{i-1}} \y{i}{s_i} e_i \,\big|\, |s_i| \leq |p|\,}$.
Hence, as an element of $\bnk(q,\l,A_i,q_i)$,
\[
X_i \y{i}{-p} e_i \in \la{Y^{s_1} \y{2}{s_2} \ldots \y{i-1}{s_{i-1}} \y{i}{s_i} e_i \,\big|\, |s_i| \leq |p|\,}.
\]
This concludes the proof of (I) and (II) for all integers $p$. Applying (\ref{eqn:star}) to part (II) immediately gives part (III) of the Lemma.
\end{proof}

\vskip 1cm
\begin{center} \textbf{Notational conventions}
\end{center}
\vskip 0.5cm
We now take the opportunity to fix some ``standard'' notation to be used throughout. \\
If $R$ is a ring as in Definition \ref{defn:bnk}, we may use $\bnk(R)$ or $\bnk$ for short to denote the algebra $\bnk(q,\l,A_i,q_i)$.
If $J$ is a subset of an $R$-module, $\la{J}$ is used to denote the $R$-span of $J$.
Finally, for a subset $S \subseteq R$, we denote $\la{S}_R$ to be the \emph{ideal} generated by $S$ in $R$ and only omit the subscript $R$ if it does not create any ambiguity in the current context.

\chapter{Spanning sets of {$\mathscr{B}_{n}^{k}$}} \label{chap:spanning}

In this chapter, we produce a spanning set of $\bnk(R)$ for any ring $R$, as in Definition \ref{defn:bnk}, of cardinality $k^n(2n-1)!! = k^n (2n-1) \cdot (2n-3) \cdots 3 \cdot 1$. Hence this shows the rank of $\bnk$ is at most $k^n(2n-1)!!$. The spanning set we obtain involves picking a basis of the Ariki-Koike algebras, which we define below. We note here that our spanning sets differs from that obtained by Goodman and Hauschild Mosley in \cite{GH107}.

\begin{defn}
For any unital commutative ring $R$ and $q', q_0, \o, q_{k-1} \in R$, let $\mathfrak{h}_{n,k}(R)$ denote the unital associative $R$-algebra with generators $T_0^{\pm 1}$, $T_1^{\pm 1}$, \o, $T_{n-1}^{\pm 1}$ and relations
\[
\begin{array}{rcll}
  T_0 T_1 T_0 T_1 &=& T_1 T_0 T_1 T_0& \\
  T_i T_{i \pm 1} T_i \= T_{i \pm 1} T_i T_{i \pm 1}&\quad \text{for } i = 1,\o,n-2\\
  T_i T_j \= T_j T_i&\quad \text{for } |i-j|\geq 2 \\
T_0^k \= {\displaystyle \sum_{i=0}^{k-1} q_i T_0^i} \\
   T_i^2 \= (q'-1)T_i + q'&\quad \text{for } i = 1,\o,n-2.
\end{array}
\]
\end{defn}
Here $\mathfrak{h}_{n,k}$ denotes the Ariki-Koike algebras, as introduced independently by Ariki and Koike in \cite{AK94} and Brou\'e and Malle in \cite{BM93}. They are sometimes referred to as the `cyclotomic Hecke algebras of type $G(k,1,n)$' and may be thought of as the Iwahori-Hecke algebras corresponding to the complex reflection group $(\mathbb{Z}/k\mathbb{Z}) \wr \mathfrak{S}_n$, the wreath product of the cyclic group $\mathbb{Z}/k\mathbb{Z}$ of order $k$ with the symmetric group $\mathfrak{S}_n$ of degree $n$. Indeed, by considering the case when $q' = 1$, $q_0 = 1$ and $q_i = 0$, one recovers the group algebra of $(\mathbb{Z}/k\mathbb{Z}) \wr \mathfrak{S}_n$. Also, it is isomorphic to the Iwahori-Hecke algebra of type $A_{n-1}$ or $B_n$, when $k = 1$ or $2$, respectively. These algebras feature extensively in the literature. For example, Ariki and Koike \cite{AK94} prove that it is $R$-free of rank $k^n n!$, the cardinality of $(\mathbb{Z}/k\mathbb{Z}) \wr \mathfrak{S}_n$. In addition, they classify its irreducible representations, and give explicit matrix representations in the generic semisimple setting. Also, Graham and Lehrer \cite{GL96} and Dipper, James and Mathas \cite{DJM98} prove that the algebra is cellular. 

Now suppose $R$ is a ring as in the definition of $\bnk$ and let $q' := q^2$. Then, from the given presentations of the algebras, it is straightforward to show that $\ak{n}(R)$ is a quotient of $\bnk(R)$ under the following projection
\begin{eqnarray*}
  \hspace{2cm}\pi_n: \bnk &\rightarrow& \ak{n} \\
  Y &\mapsto& T_0, \\
  X_i &\mapsto& q^{-1} T_i, \text{ \quad for }1 \leq i \leq n-1  \\
  e_i &\mapsto& 0.
\end{eqnarray*}
Indeed, $\bnk/I \cong \mathfrak{h}_{n,k}$ as $R$-algebras, where $I$ is the two-sided ideal generated by the $e_i$'s in $\bnk(R)$. (Remark: due to relation (\ref{eqn:eee}), it is straightforward to see that the ideal $I$ is actually equal to the two-sided ideal generated by just a single fixed $e_j$).
Our main aim in this chapter is to obtain a spanning set of $\bnk$, for any choice of basis $\X_{n,k}$ for $\mathfrak{h}_{n,k}$. 
For any basis $\X_{n,k}$ of $\mathfrak{h}_{n,k}$, let $\overline{\X}_{n,k}$ be any subset of $\bnk$ mapping onto $\X_{n,k}$ of the same cardinality. 
Also, for any $l \leq n$, there is a natural map $\blk \rightarrow \bnk$. Let $\bl{l}$ denote the image of $\blk$ under this map; it is the subalgebra of $\bnk$ generated by $Y, X_1, \o, X_{l-1}, e_1, \o, e_{l-1}$. Observe that \emph{a priori} it is not clear that this map is injective; i.e., that $\bl{l}$ is isomorphic to $\blk$. In fact, over a specific class of ground rings, this will follow as a consequence of freeness of $\bnk$ which is established in Chapter \ref{chap:basis}.
Finally, let $\widetilde{\X}_{l,k}$ be the image of $\overline{\X}_{l,k}$ in $\bnk$.

\begin{thm} \label{thm:span}
The set of elements of the following form spans $\bnk$.
\vspace{-0.1cm}
\begin{eqnarray*}
\y{i_1}{s_1} \y{i_2}{s_2} \, \ldots \, \y{i_f}{s_f} \!\!\!\!\!&&\!\!(X_{i_1} \ldots X_{j_{\scriptscriptstyle 1}-1} e_{j_1} \ldots e_{n-2} e_{n-1}) \\
 \ldots \!\!\!\!\!&&\!\! (X_{i_f} \ldots X_{j_{\scriptscriptstyle f}-1} e_{j_f} \ldots e_{n-2f} e_{n-2f+1})\,\, \chi^{(n-2f)} \\
\!\!\!\!\!&&\!\! (e_{n-2f+1} e_{n-2f} \ldots e_{h_{\scriptscriptstyle f}} X_{h_{\scriptscriptstyle f} - 1} \ldots X_{g_f}) \ldots \\
\!\!\!\!\!&&\!\! (e_{n-1} e_{n-2} \ldots e_{h_{\scriptscriptstyle 1}} X_{h_{\scriptscriptstyle 1} - 1} \ldots X_{g_1}) \,\,\, \y{g_f}{t_f} \y{g_2}{t_2} \, \ldots \, \y{g_1}{t_1},
\end{eqnarray*}
where $f = 1,2,\ldots, \lfloor \frac{n}{2} \rfloor$, $i_1 > i_2 > \ldots > i_f$, $g_1 > g_2 > \ldots > g_f$ and, for each $m = 1,2,\ldots f$, we require $1 \leq i_m < j_m \leq n-2m+1$, $s_1, \ldots, s_f, t_1, \ldots, t_f \in \left\{ \floor{\frac{k}{2}} - (k-1), \o, \floor{\frac{k}{2}} \right\}$ and $\chi^{(n-2f)}$ is an element of $\widetilde{\X}_{n-2f,\,k}$. 
\end{thm}

\noi In a later chapter, we will show that, under additional assumptions on the ground ring, picking an appropriate `lifting' of a certain basis of $\ak{n}$ leads to a cellular basis of $\bnk$.

Suppose $l \geq 1$. Let $i$ and $j$ be such that $i \leq j \leq l$ and $p$ be any integer. Define
\[
\a{ijl}{p} := \y{i}{p}  X_i \o X_{j-1} e_j \o e_l.
\]
Then Theorem \ref{thm:span} says that the algebra $B_n^k$ is spanned by the set of elements
\[
\alpha_{i_1j_1,n-1}^{s_1}\ldots\alpha_{i_fj_f,n-2f+1}^{s_f} \chi^{(n-2f)}(\alpha_{g_fh_f,n-2f+1}^{t_f})^*\ldots (\alpha_{g_1h_1,n-1}^{t_1})^*,
\]
with conditions specified as above.
Diagrammatically (in the Kauffman tangle algebra on $n$ strands), the product $\a{ijl}{0}$ may be visualised as a `tangle diagram' with $n$ points on the top and bottom row such that the $i^\mathrm{th}$ and $(j+1)^\mathrm{th}$ are joined by a horizontal strand in the top row. The rest of the diagram consists of vertical strands, which cross over this horizontal strand but not each other, and a horizontal strand joining the $l^\mathrm{th}$ and $(l+1)^\mathrm{th}$ points in the bottom row. We illustrate this roughly in Figure \ref{fig:onealpha} below.
 \begin{figure}[h!]
 \begin{center}
   \includegraphics{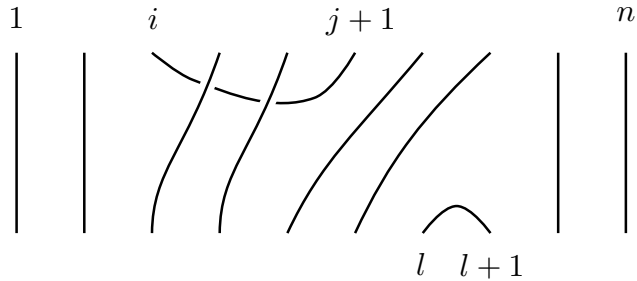}
   \caption{A diagrammatic intepretation of $\a{ijl}{0} = X_i \o X_{j-1} e_j \o e_l$ as a tangle on $n$ strands.} \label{fig:onealpha}
\end{center}
\end{figure}
\\ Using this picture visualisation, one may then use a straightforward calculation to show that the spanning set of Theorem \ref{thm:span} has cardinality $k^n (2n-1)!!$.

The following lemma essentially states that left multiplication of an $\a{mjl}{p}$ chain by a generator of $\bl{l+1}$ yields another $\alpha$ chain multiplied by `residue' terms in the smaller $\bl{l-1}$ subalgebra. Specifically, it helps us to prove that the $R$-submodule spanned by $\{\a{ijl}{p} \bl{l-1}\}$ is a left ideal of $\bl{l+1}$, in particular when $p$ is restricted to be within a certain range of $k$ consecutive integers.
\begin{lemma} \label{lemma:leftideal}
  For $\gamma \in \{X,e\}$, $m \leq l$ and $p \in \mathbb{Z}$,
  \[
  \gamma_m \alpha_{ijl}^p \in \la{\a{i'j'l}{p'} \bl{l-1} \,\Big|\, i' \geq \mathrm{min}(i,m), |p'| \leq |p| \text{ and } p'p \geq 0 \text{ unless } i'=m}.
  \]
  In fact, the only case in which $p'p < 0$ occurs is the case $X_i \alpha_{ijl}^p$, where $i \leq j \leq l$ and $p \in \mathbb{Z}$.
\end{lemma}

\begin{proof}
  Let $p$ be any integer and fix $m$, $i$, $j$ and $l$. \\
Henceforth, let $T := \la{\a{i'j'l}{p'} \bl{l-1} \,\big|\, i' \geq \mathrm{min}(i,m), |p'| \leq |p| \text{ and } p'p \geq 0 \text{ unless } i'=m}$.
In all the following calculations, it is straightforward in each case to check that the resulting elements satisfy the minimality condition required to be a member of $T$.

\noi  \textbf{The action of $e_m$}. \\
  The action of $e_m$ on $\a{ijl}{p}$ falls into the following four cases:
\begin{enumerate}
  \item[(1)] $e_m \cdot \y{m}{p} X_m \ldots e_j \ldots e_l$, where $m < j \leq l$,\\
  \item[(2)] $e_m \cdot \y{m}{p} e_m \ldots e_l$, where $m = j \leq l$,\\
  \item[(3)] $e_m \cdot \a{m+1,j,l}{p}$, where $m+1 \leq j \leq l$, \\
  \item[(4a)] $e_m \cdot \a{ijl}{p}$, where $m < i-1$ and $i \leq j \leq l$, \\
  \item[(4b)] $e_m \cdot \a{ijl}{p}$, where $m > i$ and $i \leq j \leq l$.
\end{enumerate}
\vskip .3cm
\textbf{(1).} By Lemma \ref{lemma:eype} (III),
\begin{align*}
  e_m \y{m}{p} X_m \ldots e_j \ldots e_l &\stackrel{\hphantom{(1)}}{\in} \la{e_m Y^{s_1} \ldots \y{m-1}{s_{m-1}}\y{m}{s_m} X_{m+1} \ldots e_j \ldots e_l \,\big|\, |s_m| \leq |p|\,} \\
&\stackrel{\hphantom{(1)}}{\in} \la{e_m  X_{m+1} \ldots e_j \ldots e_l \left( Y^{s_1} \ldots \y{m-1}{s_{m-1}}\y{m}{s_m} \right) \,\big|\, |s_m| \leq |p|\,} \\
&\stackrel{(\ref{eqn:xxe})}{\in} \la{e_m \ldots e_l \left(X_{j-2}^{-1} \ldots X_{m+1}^{-1} X_m^{-1} \y{m}{s_m} \y{m-1}{s_{m-1}} \ldots Y^{s_1}\right) \,\big|\, |s_m| \leq |p|\,}.
\end{align*}
Here $m\leq j-1\leq l-1$, so the term in the brackets above $X_{j-2}^{-1} \ldots X_{m+1}^{-1} X_m^{-1} \y{m}{s_m} \y{m-1}{s_{m-1}} \ldots Y^{s_1}$ is in $\bl{l-1}$. 
Hence, $e_m \y{m}{p} X_m \ldots e_j \ldots e_l \in \la{\a{mml}{0}\bl{l-1}} \subseteq T$.

\noi \textbf{(2).} By Lemma \ref{lemma:eype} (I),
\begin{align*}
  e_m \y{m}{p} e_m \ldots e_j \ldots e_l &\in \la{Y^{s_1} \y{2}{s_2} \ldots \y{m-1}{s_{m-1}} e_m \o e_l} \\
&\in \la{\a{mml}{0} \bl{l-1}} \subseteq T,
\end{align*}
since $Y^{s_1} \y{2}{s_2} \ldots \y{m-1}{s_{m-1}} \in \bl{l-1}$, as $m \leq l$ in this case. 

\noi \textbf{(3).} By equations (\ref{eqn:m9}), (\ref{eqn:prop1b}) and (\ref{eqn:xxe}),
\begin{align*}
  e_m \a{m+1,j,l}{p} &= e_m \y{m}{-p} X_{m+1} \o e_j \o e_l \\
&= e_m X_{m+1} \o e_j \o e_l \y{m}{-p} \\
\- e_m \o e_l \left(X_{j-2}^{-1} \ldots X_{m+1}^{-1} X_m^{-1} \y{m}{-p} \right) \in T.
\end{align*}

\noi \textbf{(4).} We want to prove that $e_m \a{ijl}{p} \in T$, when $i \neq m$, $m+1$ and $i \leq j \leq l$.
This is separated into the following two cases. 

(a) If $m \leq i-2 \leq l-2$, then $e_m \in \bl{l-1}$ commutes past $\a{ijl}{p}$. Hence $e_m \a{ijl}{p} = \a{ijl}{p} e_m \in T$.

(b) On the other hand, if $m \geq i+1$ then: 

When $m < j$, 
\begin{align*}
  e_m \a{ijl}{p} &\p \y{i}{p} X_i \o X_{m-2} e_m X_{m-1} X_m \left( X_{m+1} \o e_j \o e_l \right) \\
&\rel{eqn:xxe} \y{i}{p} X_i \o X_{m-2} e_m e_{m-1} \left( X_{m+1} \o e_j \o e_l \right) \\
&\p e_m X_{m+1} \o e_j \o e_l \left(\y{i}{p} X_i \o X_{m-2} e_{m-1} \right) \\
&\rel{eqn:xxe} e_m \o e_l \left(X_{j-2}^{-1} \ldots X_{m+1}^{-1} X_m^{-1}\right) \left(\y{i}{p} X_i \o X_{m-2} e_{m-1} \right),
\end{align*}
which is an element of $T$, since $m \leq l-1$ and $j-2 \leq l-2$.

When $m=j$,
\begin{align*}
  e_m \a{ijl}{p} &\p \y{i}{p} X_i \o X_{m-2} e_m X_{m-1} e_m \o e_l \\
&\rel{eqn:exe} \l^{-1} \y{i}{p} X_i \o X_{m-2} e_m \o e_l \\
&\p e_m \o e_l (\l^{-1} \y{i}{p} X_i \o X_{m-2}),
\end{align*}
which is an element of $T$, since $m-2 = j-2 \leq l-2$.

When $m>j$,
\begin{align*}
  e_m \a{ijl}{p} &\p \y{i}{p} X_i \o \gamma_{m-2} e_m e_{m-1} e_m e_{m+1} \o e_l, \quad \text{where $\gamma$ could be $X$ or $e$,} \\
&\rel{eqn:eee} \y{i}{p} X_i \o \gamma_{m-2} e_m e_{m+1} \o e_l \\
&\p e_m \o e_l \left( \y{i}{p} X_i \o \gamma_{m-2} \right) \in T.
\end{align*}

\noi We have now proved that $e_m \a{ijl}{p} \in T$, for all $m \leq l$, $i \leq j \leq l$, and $p \in \mathbb{Z}$.
\\
\\
\textbf{The action of $X_m$.}
\\
The action of $X_m$ on $\a{ijl}{p}$ falls into the following four cases:
\begin{enumerate}
\item[(A)] $X_m \cdot \a{m+1,j,l}{p}$, where $m+1 \leq j \leq l$, \\
\item[(B)] $X_m \cdot \y{m}{p}e_m \o e_l$, where $m = j \leq l$, \\
\item[(C)] $X_m \cdot \y{m}{p}X_m \o e_j \o e_l$, where $m < j \leq l-1$,  \\
\item[(D1)] $X_m \cdot \a{ijl}{p}$, where $m < i-1$ and $i \leq j \leq l$, \\
\item[(D2)] $X_m \cdot \a{ijl}{p}$, where $m > i$ and $i \leq j \leq l$.
\end{enumerate}
\vskip .3cm
\noi \textbf{(A).} When $p$ is a non-negative integer, using equations (\ref{eqn:m5}), (\ref{eqn:m9}) and (\ref{eqn:prop1b}) gives 
\begin{align*}
X_m & \a{m+1,j,l}{p} \\
\- \y{m}{p} X_m \gamma_{m+1} \o e_j \o e_l + \d \pum{s} \y{m}{p-s} \y{m+1}{s} \gamma_{m+1} \o e_j \o e_l \\
&\phantom{=} \qquad \qquad \qquad \qquad \quad \,\,\,\,\, - \d \pum{s} \y{m}{p-s} e_m \y{m+1}{s}  \gamma_{m+1} \o e_j \o e_l \\
\- \y{m}{p} X_m \o e_j \o e_l + \d \pum{s} \y{m+1}{s} \gamma_{m+1} \o e_j \o e_l \left( \y{m}{p-s} \right) \\
&\phantom{=} \qquad \qquad \qquad \quad \,\,\,\,\, - \d \pum{s} \y{m}{p-s} e_m \o e_l \left(X_{j-2}^{-1} \ldots X_{m+1}^{-1} X_m^{-1} \y{m}{-s}\right),
\end{align*}
where $\gamma_{m+1}$ could be either $X_{m+1}$ or $e_{m+1}$.
Observe that if $1 \leq s \leq p$ then $0 \leq p-s \leq p-1$, hence $|s|, |p-s| \leq |p|$. Also, in this case, $m \leq l-1$ and $j-2 \leq l-2$ so the expressions in the brackets above are indeed elements of $\bl{l-1}$. Hence, for all $m+1 \leq j \leq l$, we have $X_m \cdot \a{m+1,j,l}{p} \in T$.

\noi Also, by equations (\ref{eqn:m6}), (\ref{eqn:m9}) and (\ref{eqn:prop1b}), 
\begin{align*}
X_m & \a{m+1,j,l}{-p} \\
\- \y{m}{-p} X_m \gamma_{m+1} \o e_j \o e_l - \d \pum{s} \y{m}{-s} \y{m+1}{s-p} \gamma_{m+1} \o e_j \o e_l \\
&\phantom{=} \qquad \qquad \qquad \qquad \quad \,\,\,\,\, + \d \pum{s} \y{m}{-s} e_m \y{m+1}{s-p} \gamma_{m+1} \o e_j \o e_l \\
\- \y{m}{-p} X_m \o e_j \o e_l - \d \pum{s} \y{m+1}{s-p} \gamma_{m+1} \o e_j \o e_l \ll{\y{m}{-s}} \\
&\phantom{=} \qquad \qquad \qquad \quad \,\,\,\,\, + \d \pum{s} \y{m}{-s} e_m \o e_l \left(X_{j-2}^{-1} \ldots X_{m+1}^{-1} X_m^{-1} \y{m}{p-s}\right).
\end{align*}
Observe that when $1 \leq s \leq p$, we have that $1-p \leq s-p \leq 0$ so certainly $|\!-\! s\kern1.1pt| = |s|$ and $|s-p| \leq |p|$. Again, since $m \leq l-1$ and $j-2 \leq l-2$, the expressions in the brackets above are indeed elements of $\bl{l-1}$. Hence, for all $m+1 \leq j \leq l$, $X_m \cdot \a{m+1,j,l}{-p} \in T$.

\noi \textbf{(B).} By Lemma \ref{lemma:eype}(II) and equation (\ref{eqn:prop1b}),
\begin{align*}
  X_m \y{m}{p}e_m \o e_l &\in \la{Y^{s_1} \o \y{m-1}{s_{m-1}} \y{m}{s_m} e_m \o e_l \,\big|\, |s_i| \leq |p|\,} \\
&\in \la{\a{mml}{s_m} \left(Y^{s_1} \o \y{m-1}{s_{m-1}} \right) \,\big|\, |s_m| \leq |p|\,} \subseteq T, \quad \text{since $m \leq l$}.
\end{align*}

\noi \textbf{(C).} 
When $p$ is a non-negative integer, using equations (\ref{eqn:m11}), (\ref{eqn:m9}) and (\ref{eqn:prop1b}),
\begin{align*}
X_m & \y{m}{p}X_m \o e_j \o e_l \\
&\p \y{m+1}{p} X_{m+1} \ldots e_j \o e_l - \d \sum_{s=1}^{p-1} X_m \y{m}{s} \y{m+1}{p-s} X_{m+1} \ldots e_j \o e_l \\
&\phantom{\p} \qquad \qquad \qquad \qquad \quad + \d \sum_{s=1}^{p-1} X_m \y{m}{s} e_m \y{m+1}{p-s} X_{m+1} \ldots e_j \o e_l \\
&\p  \y{m+1}{p} X_{m+1} \ldots e_j \o e_l - \d \sum_{s=1}^{p-1} X_m \y{m+1}{p-s} X_{m+1} \ldots e_j \o e_l \left( \y{m}{s} \right) \\
&\phantom{\p} \qquad \qquad \qquad \qquad \quad + \d \sum_{s=1}^{p-1} X_m \y{m}{s} e_m X_{m+1} \ldots e_j \o e_l \left( \y{m}{s-p} \right) \\
&\rel{eqn:xxe} \y{m+1}{p} X_{m+1} \ldots e_j \o e_l - \d \sum_{s=1}^{p-1} X_m \y{m+1}{p-s} X_{m+1} \ldots e_j \o e_l \left( \y{m}{s} \right) \\
&\phantom{\p} \qquad \qquad \qquad \qquad \quad + \d \sum_{s=1}^{p-1} X_m \y{m}{s} e_m \o e_l \left( X_{j-2}^{-1} \o X_{m+1}^{-1} X_m^{-1} \y{m}{s-p} \right).
\end{align*}

\noi The first term in the above equation is $\a{m+1,j,l}{p} \in T$. In the second summation term, we have elements of the form $X_m \y{m+1}{u} X_{m+1} \ldots e_j \o e_l \bl{l-1}$, where $1 \leq u \leq p-1$. By case (A) above and since $|u| \leq |p|$, we know therefore the second term is in $T$. Moreover, by case (B), $X_m \y{m}{s} e_m \o e_l \in T$ for all $1\leq s \leq p-1$, so the third term is also in $T$. \\
Hence, for all $m \leq l-1$, we have $X_m \y{m}{p}X_m \o e_j \o e_l \in T$.

Similarly, using equations (\ref{eqn:m13}),(\ref{eqn:m9}) and (\ref{eqn:prop1b}),
\begin{align*}
&  X_m \y{m}{-p}X_m \o e_j \o e_l \\
\- \y{m+1}{-p} X_{m+1} \ldots e_j \o e_l + \d \sum_{s=0}^{p} X_m \y{m}{s-p} X_{m+1} \ldots e_j \o e_l \left( \y{m}{-s} \right) \\
&\phantom{\p} \qquad \qquad \qquad \qquad \quad + \d \sum_{s=0}^{p} X_m \y{m}{-s} e_m \o e_l \left( X_{j-2}^{-1} \o X_{m+1}^{-1} X_m^{-1} \y{m}{p-s} \right).
\end{align*}

\noi The first term in the above equation is $\a{m+1,j,l}{-p} \in T$. The second summation term involves elements of the form $X_m \y{m+1}{u} X_{m+1} \ldots e_j \o e_l \bl{l-1}$, where $-p \leq u \leq 0$. By case (A) above and since $|u| \leq |p|$, we know therefore the 2nd term is in $T$. Moreover, by case (B), $X_m \y{m}{-s} e_m \o e_l \in T$ for all $0 \leq s \leq p$, so the third term is also in $T$. \\
Hence, for all $m \leq l-1$, $X_m \y{m}{-p}X_m \o e_j \o e_l \in T$. We have now proved that, whether
$p$ is positive or negative, $X_m \a{mjl}{p} \in T$.

\noi \textbf{(D).} We want to prove that $X_m \a{ijl}{p} \in T$, when $i \neq m$, $m+1$, $i \leq j \leq l$ and $p$ is any integer. This is separated into the following two cases.

(D1). If $m \leq i-2 \leq l-2$, then $X_m \in \bl{l-1}$ commutes past $\a{ijl}{p}$. Hence $X_m \a{ijl}{p} = \a{ijl}{p} X_m \in T$.

(D2). On the other hand, if $m \geq i+1$ then again we have the following three cases to consider:

When $m <j \leq l$,
\begin{align*}
  X_m \a{ijl}{p} &\p \y{i}{p} X_i \o X_{m-2} X_m X_{m-1} X_m \left( X_{m+1} \o e_j \o e_l \right) \\
&\rel{eqn:braid1} \y{i}{p} X_i \o X_{m-2} X_{m-1} X_m X_{m-1} \left( X_{m+1} \o e_j \o e_l \right) \\
&\rel{eqn:braid2} \y{i}{p} X_i \o X_{m-2} X_{m-1} X_m \o e_j \o e_l (X_{m-1}).
\end{align*}
This is an element of $T$ as $m < j \leq l$ in this case so $X_{m-1} \in \bl{l-1}$.

When $m = j \leq l$,
\begin{align*}
  X_m \a{ijl}{p} &\p \y{i}{p} X_i \o X_{m-2} X_m X_{m-1} e_m \o e_l \\
&\rel{eqn:xxe}\y{i}{p} X_i \o X_{m-2} e_{m-1} e_m \o e_l = \a{i,j-1,l}{p} \in T.
\end{align*}

When $m>j$,
\begin{align*}
  X_m \a{ijl}{p} &\p \y{i}{p} X_i \o \gamma_{m-2} X_m e_{m-1} e_m e_{m+1} \o e_l, \quad \text{where $\gamma$ could be $X$ or $e$,} \\
&\rel{eqn:xxe} \y{i}{p} X_i \o \gamma_{m-2} X_{m-1}^{-1} e_m e_{m+1} \o e_l \\
&\rel{eqn:ix} \y{i}{p} X_i \o \gamma_{m-2} X_{m-1} e_m e_{m+1} \o e_l - \d \y{i}{p} X_i \o \gamma_{m-2} e_m e_{m+1} \o e_l \\
&\qquad \qquad\qquad\qquad\qquad\qquad\qquad\, + \d \y{i}{p} X_i \o \gamma_{m-2} e_{m-1} e_m e_{m+1} \o e_l \\
&\rel{eqn:braid2} \y{i}{p} X_i \o \gamma_{m-2} X_{m-1} e_m e_{m+1} \o e_l - \d e_m e_{m+1} \o e_l \left( \y{i}{p} X_i \o \gamma_{m-2} \right) \\
&\qquad \qquad\qquad\qquad\qquad\qquad\qquad\, + \d \y{i}{p} X_i \o \gamma_{m-2} e_{m-1} e_m e_{m+1} \o e_l.
\end{align*}
Observe that $\y{i}{p} X_i \o \gamma_{m-2} \in \bl{l-1}$, as $m\leq l$ in this case. 

\noi Furthermore, if $m-2 \geq j$, then $\gamma_{m-2} = e_{m-2}$ and 
\begin{align*}
X_m \a{ijl}{p} \- \y{i}{p} X_i \o e_{m-2} e_{m-1} e_m \o e_l \left( X_{m-2}^{-1} \right) \\ 
&- \d e_m e_{m+1} \o e_l \left( \y{i}{p} X_i \o e_{m-2} \right) + \d \y{i}{p} X_i \o e_{m-2} e_{m-1} e_m \o e_l.
\end{align*}
Otherwise, if $m-1 = j$, then $\gamma_{m-2} = X_{m-2}$ and
\begin{align*}
X_m \a{ijl}{p} \- \y{i}{p} X_i \o X_{m-2} e_{m-1} e_m \o e_l \\ 
&- \d e_m e_{m+1} \o e_l \left( \y{i}{p} X_i \o X_{m-2} \right) + \d \y{i}{p} X_i \o X_{m-2} e_{m-1} e_m \o e_l.
\end{align*}
We have now proved that for all $m \leq l$ and $i \leq j \leq l$ and $p \in \mathbb{Z}$, $X_m \a{ijl}{p} \in T$.
\end{proof}

The following lemma says, for a fixed $l$, the $R$-span of all $\a{ijl}{p} \bl{l-1}$ is a \emph{left ideal} of $\bl{l+1}$, when $p$ lies in a range of $k$ consecutive integers.

\begin{lemma} \label{lemma:angels}
  Fix some $l$. Suppose $0 \leq K < k$ and let \[P = \{-K, -K+1, \o, k-K-1\}.\] The $R$-submodule 
\[
L := \la{\a{ijl}{p} \bl{l-1} \,\big| \, p \in P}
\]
is a left ideal of $\bl{l+1}$.
\end{lemma}

\begin{proof}
We want to prove that $L$ is invariant under left multiplication by the generators of $\bl{l+1}$, namely $Y, X_1, \o, X_l, e_1, \o, e_l$.

If $i > 1$, $Y$ commutes with $\a{ijl}{p}$. Otherwise, when $i = 1$, $Y \a{ijl}{p} = \a{ijl}{p+1}$, so clearly, by the $k^{\text{th}}$ order relation on $Y$, $L$ is invariant under left multiplication by $Y^{\pm 1}$. We will show by induction on $m \leq l$ that $L$ is invariant under $X_m$ and $e_m$.

Suppose $L$ is invariant under $X_{m'}$ and $e_{m'}$ for $m'<m$. Note that when $m=1$, this assumption is vacuous. 
Then in particular, $L$ is invariant under $X_{m'}^{-1} = X_{m'} - \d + \d e_{m'}$  for all $m'<m$. Moreover, this implies $L$ is invariant under $\y{m}{\pm 1} = (X_{m-1} \o X_1 Y X_1 \o X_{m-1})^{\pm 1}$. \\
Thus for all $p' \in \mathbb{Z}$,
\begin{equation} \label{eqn:angel}
\a{m,j',l}{p'} = \y{m}{p'} \a{m,j',l}{0} \in L.
\end{equation}
For $\gamma_m \in \{X_m,e_m\}$ and $p \in P$, Lemma \ref{lemma:leftideal} implies that
\begin{align*}
  \gamma_m \a{ijl}{p} &\in  \la{\a{i'j'l}{p'} \bl{l-1} \,\big|\, i' \geq \mathrm{min}(m,i), |p'| \leq |p| \text{ and } p'p \geq 0 \text{ unless } i'=m}\\
&\subseteq \la{\a{i'j'l}{p'} \bl{l-1} \,\big|\, |p'| \leq |p| \text{ and } p'p \geq 0}+ \la{\a{mj'l}{p} \bl{l-1}}.
\end{align*}
The first set lies in $L$, as if $|p'| \leq |p|$ and $p'p \geq 0$, then $p \in P$ implies $p' \in P$. By (\ref{eqn:angel}) above, $\la{\a{mj'l}{p} \bl{l-1}} \subseteq L$. Thus $\gamma_m \a{ijl}{p} \in L$, whence $\gamma_m L \subseteq L$ and $L$ is a left ideal of $\bl{l+1}$.
\end{proof}

Now we fix $K := \left\lfloor \frac{k-1}{2} \right\rfloor$. When $k$ is odd, $K = \frac{k-1}{2}$ and when $k$ is even, $K = \frac{k-2}{2}$.
The range $P$ in Lemma \ref{lemma:angels} becomes
\[P = \left\{ -\left\lfloor \frac{k-1}{2} \right\rfloor, \o, k - \left\lfloor \frac{k-1}{2} \right\rfloor - 1\right\} = \left\{ \floor{\frac{k}{2}} - (k-1), \o, \floor{\frac{k}{2}} \right\}.\]
For $k$ odd, \[P = \{ -K, \o, K \}\] and for $k$ even, \[P = \{ -K, \o, K+1 \}.\]

We are now almost ready to prove Theorem \ref{thm:span}. A standard way to show that a set which contains the identity element spans the entire algebra is to show it spans a left ideal of the algebra or, equivalently, show that its span is invariant under left multiplication by the generators of the algebra. With the previous lemma in mind, we observe that `pushing' a generator through each $\alpha$ chain may distort the `ordering' of the $\alpha$ chains (the $i_1 > i_2 > \ldots > i_f$ requirement in the statement of Theorem \ref{thm:span}). Motivated by this, we first prove the following Lemma. 

\begin{lemma} \label{lemma:beer}
  If $i \leq g$ and $p,r \in P$,
\[
\a{i,j,l}{p} \a{g,h,l-2}{r} \in \la{\a{i',j',l}{p'} \a{g',h',l-2}{r'} \bl{l-3} \,|\, i' > i \text{ and } p',r' \in P}.
\]
\end{lemma}
\begin{proof}
  Observe that, by Lemma \ref{lemma:angels}, $L = \la{\a{g'h'l-2}{r'} \bl{l-3} | r' \in P}$ is a left ideal of $\bl{l-1}$ therefore it suffices to prove that, for all $i \leq g$ and $p,r \in P$,
\begin{align*}
\a{ijl}{p} \a{g,h,l-2}{r}  \- \left( \y{i}{p} X_i \o e_j \o e_l \right) \left(\y{g}{r} X_g \o e_h \o e_l \right) \\
&\in \la{\a{i',j',l}{p'} \bl{l-1} \a{g',h',l-2}{r'} \bl{l-3} \,|\, i' > i \text{ and } p',r' \in P}.
\end{align*}

\noi Let us denote $\la{\a{i',j',l}{p'} \bl{l-1} \a{g',h',l-2}{r'} \bl{l-3} \,\big|\, i' > i \text{ and } p',r' \in P}$ by $S$.

\noi If $g \geq j$, then 
\begin{eqnarray*}
  e_j \o e_l \ll{\y{g}{r} X_g \o e_h \o e_{l-2}} &\rel{eqn:ep2}& \y{g+2}{r} X_{g+2} \o e_{h+2} \o e_l e_j \o e_{l-1} e_l \\
&\rel{eqn:eee}& \y{g+2}{r} X_{g+2} \o e_{h+2} \o e_l e_j \o e_{l-2}. \end{eqnarray*}
Thus, using the commuting relations (\ref{eqn:braid2}), (\ref{eqn:xecomm}) and equation (\ref{eqn:prop1b}),
\begin{align*}
  \a{ijl}{p} \a{g,h,l-2}{r} \- \y{g+2}{r} X_{g+2} \o e_{h+2} \o e_l \y{i}{p} X_i \o X_{j-1} e_j \o e_{l-2} \\
\- \a{g+2,h+2,l}{r} \a{i,j,l-2}{p}.
\end{align*}
Note that $g+2 > j \geq i$ in this case.
Hence, when $g \geq j$, we have $\a{ijl}{p} \a{g,h,l-2}{r} \in  S$. 

\noi Now suppose on the contrary $g < j$. 
When $r$ is non-negative, we have the following:
\begin{eqnarray*}
  \a{ijl}{p} \a{g,h,l-2}{r} \= \left( \y{i}{p} X_i \o X_g \o X_{j-1} e_j \o e_l \right) \left(\y{g}{r} X_g \o e_h \o e_l \right) \\
&\rel{eqn:prop1b}& \y{i}{p} X_i \o X_g \y{g}{r} \a{g+1,j,l}{0} \a{g,h,l-2}{0} \\
&\rel{eqn:magic}& \y{i}{p} X_i \o X_{g-1} \ll{\y{g+1}{r} X_g} \a{g+1,j,l}{0} \a{g,h,l-2}{0} \\
&& - \,\, \d \sum_{s=1}^{r} \y{i}{p} X_i \o X_{g-1} \ll{\y{g+1}{s} \y{g}{r-s}} \a{g+1,j,l}{0} \a{g,h,l-2}{0} \\
&& + \,\, \d \sum_{s=1}^{r} \y{i}{p} X_i \o X_{g-1} \ll{\y{g+1}{s} e_g \y{g}{r-s}} \a{g+1,j,l}{0} \a{g,h,l-2}{0} \\
&\rel{eqn:prop1b}& \y{g+1}{r} \a{ijl}{p} \a{g,h,l-2}{0} \\
&& - \,\, \d \sum_{s=1}^{r} \a{g+1,j,l}{s} \y{i}{p} X_i \o X_{g-1} \a{g,h,l-2}{r-s} \\
&& + \,\, \d \sum_{s=1}^{r} \y{i}{p} X_i \o X_{g-1} \y{g+1}{s} e_g \a{g+1,j,l}{0} \a{g,h,l-2}{r-s}.
\end{eqnarray*}

\noi Observe that if $r \in P$ is non-negative, then because $1 \leq s \leq r$, it is clear that $s \in P$ and $r-s \leq r-1 \leq K$, hence $r-s \in P$ and $|r-s| \leq K$.

\noi On the other hand,
\begin{eqnarray*}
  \a{ijl}{p} \a{g,h,l-2}{-r} &\rel{eqn:m2}& \y{i}{p} X_i \o X_{g-1} \ll{\y{g+1}{-r} X_g} \a{g+1,j,l}{0} \a{g,h,l-2}{0} \\
&& + \,\, \d \sum_{s=1}^{r} \y{i}{p} X_i \o X_{g-1} \ll{\y{g+1}{s-r} \y{g}{-s}} \a{g+1,j,l}{0} \a{g,h,l-2}{0} \\
&& - \,\,  \d \sum_{s=1}^{r} \y{i}{p} X_i \o X_{g-1} \ll{\y{g+1}{s-r} e_g \y{g}{-s}} \a{g+1,j,l}{0} \a{g,h,l-2}{0} \\
&\rel{eqn:prop1b}& \y{g+1}{-r} \a{ijl}{p} \a{g,h,l-2}{0} \\
&& + \,\, \d \sum_{s=1}^{r} \a{g+1,j,l}{s-r} \y{i}{p} X_i \o X_{g-1} \a{g,h,l-2}{-s} \\
&& - \,\, \d \sum_{s=1}^{r} \y{i}{p} X_i \o X_{g-1} \y{g+1}{s-r} e_g \a{g+1,j,l}{0} \a{g,h,l-2}{-s}.
\end{eqnarray*}

If $-r \in P$, this means $-r \in \{-K, \o, -1\}$. So $1 \leq s \leq r$ implies $-s \in P$. Moreover, $|\!-\! s\kern1.1pt| \leq K$ and $s-r \in P$. To summarise, whether $r$ is positive \emph{or} negative,
\begin{align}
&  \a{ijl}{p} \a{g,h,l-2}{r} \in \y{g+1}{r} \a{ijl}{p} \a{g,h,l-2}{0} + \la{\a{g+1,j,l}{s} \y{i}{p} X_i \o X_{g-1} \a{g,h,l-2}{r-s} \,\big|\, s \in P, |r-s| \leq K} \nonumber \\
&\qquad \qquad \qquad \qquad \qquad \qquad \kern-.1em + \la{\y{i}{p} X_i \o X_{g-1} \y{g+1}{s} e_g \a{g+1,j,l}{0}\a{g,h,l-2}{r-s} \,\big|\, s \in P, |r-s| \leq K}. \label{eqn:twochains}
\end{align}
We now deal with each term of (\ref{eqn:twochains}) separately. \\
The first term is $\y{g+1}{r} \a{ijl}{p} \a{g,h,l-2}{0} = \y{g+1}{r} \y{i}{p} \a{ijl}{0} \a{g,h,l-2}{0}$.

If $h \leq j-2$ (so $i \leq g\leq h \leq j-2$), then 
\begin{align*}
  \a{ijl}{0} \a{g,h,l-2}{0} &\h \ll{X_i \o X_g \o X_h X_{h+1} \o X_{j-1} e_j \o e_l} \ll{X_g \o X_{h-1} e_h \o e_{l-2}} \\
&\rel{eqn:XXp1} \ll{X_{g+1} \o X_h e_{h+1} \o e_{j-2}} \ll{X_i \o X_{j-2} X_{j-1} e_j \o e_l} \ll{e_{j-2} \o e_{l-2}} \\
&\kern.18em\rel{eqn:xxe} \ll{X_{g+1} \o X_h e_{h+1} \o e_{j-2}} \ll{X_i \o X_{j-3} e_{j-1} e_{j-2} e_j \o e_l} \ll{e_{j-1} \o e_{l-2}} \\
&\h \ll{X_{g+1} \o X_h e_{h+1} \o e_{j-2} e_{j-1} \o e_l} \ll{X_i \o X_{j-3} e_{j-2} \o e_l} \\
&\h \a{g+1,h+1,l}{0} \a{i,j-2,l-2}{0}.
\end{align*}
As $i < g+1$, $\y{i}{p}$ commutes with $\a{g+1,h+1,l}{0}$, by equation (\ref{eqn:prop1b}).
Thus we have shown that $\y{g+1}{r} \a{ijl}{p} \a{g,h,l-2}{0} =  \a{g+1,h+1,l}{r} \a{i,j-2,l-2}{p}$ is an element of $S$, when $h \leq j-2$.

Now suppose $h \geq j-1$. Then, using equation (\ref{eqn:XXp1}) for $h \geq j$, we have the following:
\begin{eqnarray*}
  \a{ijl}{0} \a{g,h,l-2}{0} \= \ll{X_{g+1} \o X_{j-1}} \ll{X_i \o X_g \o X_{j-1}e_j \o e_h \o e_l} \ll{X_{j-1} \o e_h \o e_{l-2}} \\
&\rel{eqn:ix}& \ll{X_{g+1} \o X_{j-1}} \ll{X_i \o X_g \o X_{j-1}^{-1} e_j \o e_h \o e_l} \ll{X_{j-1} \o e_h \o e_{l-2}} \\
&& + \, \d \ll{X_{g+1} \o X_{j-1}} \ll{X_i \o X_g \o X_{j-2} e_j \o e_h \o e_l} \ll{X_{j-1} \o e_h \o e_{l-2}} \\
&& - \, \d \ll{X_{g+1} \o X_{j-1}} \ll{X_i \o X_g \o X_{j-2} e_{j-1} e_j \o e_h \o e_l} \ll{X_{j-1} \o e_h \o e_{l-2}} \\
&\stack{(\ref{eqn:xxe}),(\ref{eqn:ep2})}& \ll{X_{g+1} \o X_j} \ll{X_i \o X_g \o X_{j-2} e_{j-1} e_j \o e_h \o e_l} \ll{X_{j-1} \o e_h \o e_{l-2}} \\
&& + \, \d \ll{X_{g+1} \o X_{j-1} e_j \o e_l} \ll{X_i \o X_{j-2} X_{j-1} \o e_h \o e_{l-2}} \\
&& - \, \d \ll{X_{g+1} \o X_{j-1}} \ll{X_{j+1} \o e_{h+2} \o e_l} \ll{X_i \o X_{j-2} e_{j-1} \o e_h \o e_l} \\
&\stack{(\ref{eqn:ep2}),(\ref{eqn:eee})}& \ll{X_{g+1} \o X_j X_{j+1} \o e_l} \ll{X_i \o X_{j-2} e_{j-1} \o e_h \o e_l} \\
&& + \, \d \ll{X_{g+1} \o X_{j-1} e_j \o e_l} \ll{X_i \o X_{j-2} X_{j-1} \o e_h \o e_{l-2}} \\
&& - \, \d {X_{j+1} \o e_{h+2} \o e_l} \ll{X_{g+1} \o X_{j-1}} \ll{X_i \o X_{j-2} e_{j-1} \o e_h \o e_{l-2}} \\
&\rel{eqn:xxe}& \a{g+1,h+2,l}{0} \a{i,j-1,l-2}{0} + \d \a{g+1,j,l}{0} \a{i,h,l-2}{0} - \d \a{j+1,h+2,l}{0} \a{i,g,l-2}{0}.
\end{eqnarray*}
Therefore, when $h \geq j-1$,
\begin{align*}
\y{g+1}{r} &\a{ijl}{p} \a{g,h,l-2}{0} \\
&= \a{g+1,h+2,l}{r} \a{i,j-1,l-2}{p} + \d \, \a{g+1,j,l}{r} \a{i,h,l-2}{p} - \d \,\y{g+1}{r} \a{j+1,h+2,l}{0} \a{i,g,l-2}{p}.
\end{align*}
However, since $g < j$,  $\y{g+1}{r}$ commutes with $\a{j+1,h+2,l}{0}$. Moreover, as $g \leq l-2$, $\y{g+1}{r} \in \bl{l-1}$. Thus $\y{g+1}{r} \a{ijl}{p} \a{g,h,l-2}{0} \in S$.
So far we have proved that the first term of (\ref{eqn:twochains}) is a member of $S$, for all possibilities $i,j,g,h$ where $i \leq g < j$.

We now need to show $\la{\a{g+1,j,l}{s} \y{i}{p} X_i \o X_{g-1} \a{g,h,l-2}{r-s} \,\big|\, s \in P, |r-s| \leq K} \subseteq S$. But this follows immediately from the definition of $S$, as $r-s \in P$ and $\y{i}{p} X_i \o X_{g-1} \in \bl{l-1}$, since $i \leq g-1 \leq l-3$.

Finally, we now prove $\la{\y{i}{p} X_i \o X_{g-1} \y{g+1}{s} e_g \a{g+1,j,l}{0}\a{g,h,l-2}{r-s} \,\big|\, s \in P, |r-s| \leq K} \subseteq S$.  

\noi Let $\sigma := \y{i}{p} X_i \o X_{g-1} \y{g+1}{s} e_g \a{g+1,j,l}{0}\a{g,h,l-2}{r-s}$, where $s \in P$ and $|r-s| \leq K$.

\noi By equation (\ref{eqn:xxe}), $e_g \a{g+1,j,l}{0} = e_g \o e_l \ll{X_{j-2}^{-1} \o X_{g+1}^{-1} X_g^{-1}}$. Then
\[
\sigma = \y{i}{p} X_i X_{i+1} \o X_{g-1} \y{g+1}{s} \a{ggl}{0} \ll{X_{j-2}^{-1} \o X_{g+1}^{-1} X_g^{-1}} \a{g,h,l-2}{r-s}.
\] 
Lemma \ref{lemma:leftideal} implies that 
\[
X_g^{-1} \a{g,h,l-2}{r'} \in \la{\a{g'h'l-2}{p'} \bl{l-3} \,|\, g' \geq g, |p'| \leq |r'| \text{ and } p'r' \geq 0 \text{ unless } g' = g},
\]
where $r' = r-s$.
Observe that $p'$ and $r'$ could be of different signs, but as $|r-s| \leq K$, we know that $|p'| \leq K$.
This allows us to apply Lemma \ref{lemma:leftideal} repeatedly to get
\[
X_{j-2}^{-1} \o X_{g+1}^{-1} X_g^{-1} \a{g,h,l-2}{r-s} \in \la{\a{g'h'l-2}{p'} \bl{l-3} \,|\, g' \geq g, |p'| \leq K}.
\]
Therefore
\begin{align*}
  \sigma &\stackrel{\phantom{(19)}}{\in} \la{\y{i}{p} X_i X_{i+1} \o X_{g-1} \y{g+1}{s} \a{ggl}{0} \a{g'h'l-2}{p'} \bl{l-3} \,|\, g' \geq g, |p'| \leq K} \\
&\rel{eqn:ep2} \la{\y{i}{p} X_i X_{i+1} \o X_{g-1} \y{g+1}{s} \a{g'+2,h'+2,l}{p'} \a{ggl}{0} \bl{l-3} \,|\, g' \geq g, |p'| \leq K} \\
&\kern.18em\rel{eqn:eee} \la{\y{i}{p} X_i X_{i+1} \o X_{g-1} \y{g+1}{s} \a{g'+2,h'+2,l}{p'} \a{g,g,l-2}{0} \bl{l-3} \,|\, g' \geq g, |p'| \leq K} \\
&\stackrel{\phantom{(19)}}{\in} \la{\a{g'+2,h'+2,l}{p'} \y{i}{p} X_i X_{i+1} \o X_{g-1} \y{g+1}{s} \a{g,g,l-2}{0} \bl{l-3} \,|\, g' \geq g, |p'| \leq K}.
\end{align*}
Now $i \leq g \leq l-2$, so $\y{i}{p} X_i X_{i+1} \o X_{g-1} \y{g+1}{s} \in \bl{l-1}$. 
Note that $g'+2 \geq g+2 > i$ and $p' \in P$. Thus $\sigma \in S$.

Therefore we have now proven that each term arising in (\ref{eqn:twochains}) is in S, for all $i,j,g,h$ where $i \leq g < j$.
This concludes the proof of Lemma \ref{lemma:beer}.
\end{proof}

Henceforth, we implicitly require $p \in P$ whenever we write $\alpha_{ijl}^p$. For all $m \geq 0$ and $l \geq 2m$, let us define the following subsets of $\bnk$. Note that these are \emph{not} $R$-submodules.
\begin{eqnarray*}
  V_{l,m}^g &:=& \left\{ \a{i_1,j_1,l-1}{s_1} \a{i_2,j_2,l-3}{s_2} \o \a{i_m,j_m,l-2m+1}{s_m} \right\} \quad \text{and} \\
  V_{l,m} &:=& \left\{ \a{i_1,j_1,l-1}{s_1} \a{i_2,j_2,l-3}{s_2} \o \a{i_m,j_m,l-2m+1}{s_m} \,\big|\, i_1 > i_2 > \ldots > i_m \right\}.
\end{eqnarray*}
If $m > l/2$, we let $V_{l,m}^g = V_{l,m}$ be the empty set.

\noi 
If $U_1$ and $U_2$ are subsets of $\bnk$, let $U_1 U_2 := \la{ u_1u_2 | u_1 \in U_1, u_2 \in U_2}$; i.e., the $R$-span of the set $\{u_1u_2 | u_1 \in U_1, u_2 \in U_2\}$.

\begin{lemma} \label{lemma:Vgleftideal}
For all $l$ and $m$, $V_{l,m}^g\bl{l-2m}$ is a left ideal of $\bl{l}$.
\end{lemma}

\begin{proof}
We prove the statement by induction on $m$. When $m=0$, we have $V_{l,m}^g = \{1\}$ so $V_{l,m}^g\bl{l-2m}=\bl{l}$ and the statement then follows trivially.

Suppose that $m\geq1$ and assume the statement is true for $m-1$. Note that $l\geq2m\geq2$. By the definition of $V_{l,m}^g$, we have
\begin{align*}
    \bl{l}V_{l,m}^g\bl{l-2m}\-\langle\bl{l}\alpha_{ij,l-1}^pV_{l-2,m-1}^g\bl{l-2m}\rangle\\
        &\subseteq \langle\alpha_{i,j,l-1}^p\bl{l-2}V_{l-2,m-1}^g\bl{l-2m}\rangle\text{,\quad by Lemma \ref{lemma:angels},}\\
        &\subseteq \langle\alpha_{ij,l-1}^pV_{l-2,m-1}^g\bl{l-2m}\rangle\text{,\quad by induction,}\\
        \-V_{l,m}^g\bl{l-2m},
\end{align*}
as required.
\end{proof}

\begin{lemma}\label{lemma:VgV}
For all $l$ and $m$, we have $V_{l,m}^g\bl{l-2m} = V_{l,m}\bl{l-2m}$.
\end{lemma}

\begin{proof}
By definition, $V_{l,m} \subseteq V_{l,m}^g$ hence $V_{l,m}^g\bl{l-2m}\supseteq V_{l,m}\bl{l-2m}$. It now remains to prove the reverse inclusion. We again proceed by induction on $m$. In the case $m=0$, the statement merely says $\bl{l} = \bl{l}$. And when $m=1$, the statement is clearly satisfied as $V_{l,1}^g = V_{l,1}$.

Suppose $m\geq2$ and that the statement is true for $m'<m$. By definition of $V_{l,m}^g$, we have
\[
V_{l,m}^g\bl{l-2m}=\langle\alpha_{ij,l-1}^pV_{l-2,m-1}^g\bl{l-2m}\rangle.
\]
It therefore suffices to show
\begin{equation}\label{eqn:order}
    \alpha_{i,j,l-1}^pV_{l-2,m-1}^g\bl{l-2m}\subseteq V_{l,m}\bl{l-2m},
\end{equation}
for $1\leq i<l$. We will prove (\ref{eqn:order}) by descending induction on $i$. Suppose (\ref{eqn:order}) holds for all $i'$ such that $i<i'<l$. Observe that when $i=l-1$, the inductive hypothesis is vacuous. By induction on $m$, $V_{l-2,m-1}^g\bl{l-2m}=V_{l-2,m-1}\bl{l-2m}$, thus the LHS of (\ref{eqn:order}) is spanned by the set of elements of the form
\[
\alpha_{i,j,l-1}^p\alpha_{i_2j_2l-3}^{s_2}\ldots\alpha_{i_mj_ml-2m+1}^{s_m}\bl{l-2m},
\]
where $i_2>i_3>\ldots>i_m$. If $i>i_2$, then we already have $i>i_2>\ldots>i_m$, so this is a subset of $V_{l,m}\bl{l-2m}$ by definition. On the 
other hand, if $i\leq i_2$ then
\begin{align*}
    \alpha_{i,j,l-1}^p\alpha_{i_2j_2l-3}^{s_2}\ldots\alpha_{i_mj_ml-2m+1}^{s_m} \bl{l-2m} \hspace{-40mm}\\
        &\subseteq \alpha_{i,j,l-1}^p\alpha_{i_2j_2l-3}^{s_2}V_{l-4,m-2}^g\bl{l-2m}\\
        &\subseteq \langle\alpha_{i'j'l-1}^{p'}\alpha_{g'h'l-3}^{r'}\bl{l-4}V_{l-4,m-2}^g\bl{l-2m}\mid i'>i\rangle\text{, \quad by Lemma \ref{lemma:beer}},\\
        &\subseteq \langle\alpha_{i'j'l-1}^{p'}\alpha_{g'h'l-3}^{r'}V_{l-4,m-2}^g\bl{l-2m}\mid i'>i\rangle\text{, \quad by Lemma \ref{lemma:Vgleftideal}}, \\
        &\subseteq \langle\alpha_{i'j'l-1}^{p'}V_{l-2,m-1}^g\bl{l-2m}\mid i'>i\rangle\\
        &\subseteq V_{l,m}\bl{l-2m}\text{, \quad by induction on }i.
\end{align*}
Thus (1) holds. Hence $V_{l,m}^g\bl{l-2m} = V_{l,m}\bl{l-2m}$.
\end{proof}

\noi Recall
\[
\blk/\blk e_{l-1}\blk\cong \ak{l}
\]
and $\pi_l:\blk \twoheadrightarrow \ak{l}$ is the corresponding projection. Recall $\overline{\X}_{l,k}$ was an arbitrary subset of $\blk$ mapping onto a basis $\X_{l,k}$ of $\ak{l}$ and $|\overline{\X}_{l,k}| = |\X_{l,k}|$. We can define an $R$-module homomorphism $\phi_l:\ak{l} \rightarrow \blk$ by sending each element of $\X_{l,k}$ to the corresponding element of $\overline{\X}_{l,k}$.
Thus $\pi_l\phi_l={\rm id}_{\ak{l}}$.
Note that when $l=0$ or $1$, we have an isomorphism $\pi_l:\blk\rightarrow \ak{l}$, with inverse $\phi_l$. And, for $l \geq 2$,
\[
    \blk=\phi_l(\ak{l})+\blk e_{l-1}\blk.
\]
Thus
\begin{equation}\label{Bsplit}
    \bl{l}=\At{l}+\bl{l}e_{l-1}\bl{l}, \quad\text{ for }l\geq2,
\end{equation}
where $\At{l}$ is the image of $\phi_l(\ak{l})$ in $\bnk$.\label{defn:phiak} 
Also, \begin{equation}\label{Bsplit01}
    \At{0}=\bl{0}\qquad \text{and}\qquad \At{1}=\bl{1}.
\end{equation}

\begin{lemma}\label{Iprop}
Let $I_{l,m}=V_{l,m}\bl{l-2m}V_{l,m}^*$.
\begin{enumerate}
\item[(a)] $I_{l,m}$ is a two-sided ideal of $\bl{l}$.
\item[(b)] For $l\geq2m+2$, we have
\[
    V_{l,m}\bl{l-2m}e_{l-2m-1}\bl{l-2m}V_{l,m}^*\subseteq I_{l,m+1}.
\]
\item[(c)] For any fixed $M$, $I_{l,M} = \sum_{m \geq M} V_{l,m} \At{l-2m} V_{l,m}^*$ and is spanned by elements of the form
\[
    \alpha_{i_1j_1l-1}^{s_1}\ldots\alpha_{i_mj_ml-2m+1}^{s_m} \chi(\alpha_{g_mh_ml-2m+1}^{t_m})^*\ldots(\alpha_{g_1h_1l-1}^{t_1})^*,
\]
for $m\geq M$, $i_1>i_2>\ldots>i_m$, $g_m<g_{m-1}<\ldots<g_1$, and $\chi$ is an element of $\widetilde{\X}_{l-2m,k}$.
\end{enumerate}
\end{lemma}

\begin{proof}
\begin{enumerate}
\item[(a)]
    By Lemma \ref{lemma:VgV}, we have $I_{l,m}=V_{l,m}^g\bl{l-2m}V_{l,m}^*$. Therefore $I_{l,m}$ is a left ideal of $\bl{l}$ by Lemma \ref{lemma:Vgleftideal}. Because $^*$ preserves the subalgebra $\bl{l-2m}$, thus $I_{l,m}^*=I_{l,m}$ so $I_{l,m}$ is also a right ideal.
\item[(b)]
Suppose $l\geq2m+2$. We have
\[
e_{l-1}e_{l-3}\ldots e_{l-2m-1} = \alpha_{l-1,l-1,l-1}^0\alpha_{l-3,l-3,l-3}^0\ldots\alpha_{l-2m-1,l-2m-1,l-2m-1}^0\in V_{l,m+1}.
\]
Thus
\begin{eqnarray*}
    e_{l-1}e_{l-3}\ldots e_{l-2m-1} &\rel{eqn:eye}& A_0^{-m-1}(e_{l-1}e_{l-3}\ldots e_{l-2m-1})(e_{l-1}e_{l-3}\ldots e_{l-2m-1})^*\\
    &\in& V_{l,m+1}V_{l,m+1}^*\subseteq I_{l,m+1}.
\end{eqnarray*}
    Since $\alpha_{ijl}^p=A_0^{-1}\alpha_{ijl}^pe_l$, by relation (\ref{eqn:eye}), and $e_j$ commutes with $\bl{j-1}$, by equations (\ref{eqn:xecomm}), (\ref{eqn:eecomm}) and (\ref{eqn:prop1b}), $V_{l,m} \bl{l-2m} = V_{l,m}\bl{l-2m} e_{l-1}e_{l-3}\ldots e_{l-2m+1}$. This implies
    \begin{align*}
        V_{l,m}\bl{l-2m}e_{l-2m-1}\bl{l-2m}V_{l,m}^* &= V_{l,m}\bl{l-2m}e_{l-1}e_{l-3}\ldots e_{l-2m+1}e_{l-2m-1}\bl{l-2m}V_{l,m}^*\\
            &\subseteq V_{l,m}\bl{l-2m}I_{l,m+1}\bl{l-2m}V_{l,m}^*\\
            &\subseteq I_{l,m+1},
    \end{align*}
    using that $I_{l,m+1}$ is a two sided ideal in $\bl{l}$ as shown in part (a) above.
\item[(c)]
    If $m \geq M$, then $l-2m \leq l-2M$ so the given elements are clearly contained in $I_{l,M}$. For a fixed $m$, they span the set
    \[
        V_{l,m}\At{l-2m}V_{l,m}^*.
    \]
    It therefore suffices to prove that
    \begin{equation}\label{spanind}
        I_{l,M}\subseteq\sum_{m\geq M}V_{l,m}\At{l-2m}V_{l,m}^*.
    \end{equation}
    We prove this statement by induction on $l-2M$. If $l-2M<2$ then  
\[
I_{l,M}=V_{l,M}\bl{l-2M}V_{l,M}^*=V_{l,M}\At{l-2M}V_{l,M}^* \text{,\quad by }(\ref{Bsplit01}).
\]
Now suppose $l-2M\geq2$ and assume $I_{l,M+1}\subseteq\sum_{m\geq M+1}V_{l,m}\At{l-2m}V_{l,m}^*$. Then using (\ref{Bsplit}) and part (b) of this Lemma we have that
\begin{eqnarray*}
I_{l,M}
\= V_{l,M}\bl{l-2M}V_{l,M}^* \\
&\rel{Bsplit}& V_{l,M}\At{l-2M}V_{l,M}^*+V_{l,M}\bl{l-2M}e_{l-2M-1}\bl{l-2M}V_{l,M}^*\\
&\subseteq& V_{l,M}\At{l-2M}V_{l,M}^*+I_{l,M+1}\\
&\sstack{\text{ind. hypo.}}&V_{l,M}\At{l-2M}V_{l,M}^*+\sum_{m\geq M+1}V_{l,m}\At{l-2m}V_{l,m}^*\\
&=&\sum_{m\geq M}V_{l,m}\At{l-2m}V_{l,m}^*,
\end{eqnarray*}
proving part (c).
\end{enumerate}
\end{proof}
In particular, $I_{n,0}=B_n^k$ by definition, so when $l = n$ and $m=0$, statement (c) of the previous Lemma implies that $\bnk$ is spanned by $\{u\chi^{(n-2m)}v^*|u,v\in V_{n,m}, \chi^{(n-2m)}\in\widetilde{\X}_{n-2m,k}\}$, hence proving Theorem \ref{thm:span}.

%%%%%%%%%%%%%%%%%%%%%%%%%%%%%%%%%%%%%%%%%%%%%%%%%%%%%%%%%%

\chapter{The Admissibility Conditions} \label{chap:admissible}

In the previous chapter, we obtained a spanning set of $\bnk$ over an arbitrary ring $R$ and hence we can conclude that the rank of $\bnk$ is \emph{at most} $k^n (2n-1)!!$. Before we can prove the linear independence of our spanning set, we must first focus our attention on the representation theory of the algebra $\bt(R)$. It is here that the notion of admissibility, as first introduced by H\"aring-Oldenburg \cite{HO01}, arises. Essentially, it is a set of conditions on the parameters $A_0,\o, A_{k-1}, q_0, \o, q_{k-1}, q, \l$ in our ground ring $R$ which ensure the algebra $\bt(R)$ is $R$-free of the expected rank, namely $3k^2$.
 It turns out that, if $R$ is a ring with admissible parameters $A_0, \o, A_{k-1},q_0, \o, q_{k-1}, q, \l$ (see Definition \ref{defn:adm}) then our spanning set for \emph{general} $n$ is actually a basis. 

In Section \ref{section:V}, we establish these ``admissibility'' conditions (see Definition \ref{defn:adm}) and construct a $\bt$-module $V$ of rank $k$ (see Lemma \ref{lem:V}). These results are contained in \cite{WY06}, in which Wilcox and the author are able to use $V$ to then construct the regular representation of $\bt$ and provide an explicit basis of the algebra under a slightly stronger notion of admissibility.
These results are stated but incompletely proved in H\"aring-Oldenburg~\cite{HO01}; specifically, additional arguments are needed to prove Lemma~25 of \cite{HO01}. We take a slightly different approach and the arguments we offer in \cite{WY06} correct this problem. 
Goodman and Hauschild Mosley also study the representations of the $n=2$ algebra in detail in \cite{GH207}. However, they construct the module $V$ assuming $R$ is a field in which $\delta\neq0$ and the roots of the
$k^\mathrm{th}$ order relation and their inverses are all distinct. 
 It is worth noting here as an aside that, strongly influenced by the methods and proofs for the cyclotomic Nazarov-Wenzl algebra in Ariki et al. \!\!\cite{AMR06}, they find that the existence of this (irreducible) $k$-dimensional $\bt$-module leads them to formulate a second (different) type of admissibility condition. For more details, we refer the reader to Goodman and Hauschild Mosley \cite{GH207} and Ariki et al. \!\!\cite{AMR06}.

It is non-trivial to show that there are any non-trivial admissible rings; in other words, that the conditions we impose are consistent with each other. In Section \ref{section:V}, we also construct a  generic (``universal'') ground ring $R_0$ with admissible parameters which we will show satisfies the requirements of an admissible ring as defined in Goodman and Hauschild Mosley \cite{GH107,GH207} and allows us to follow a path similar to their trace arguments to establish linear independence of our spanning set in Chapter \ref{chap:basis}.
It is important to clarify the different notions of admissibility used in the literature. We discuss in Section \ref{section:comparison} the relationship between our admissibility conditions and those used by H\"aring-Oldenburg and Goodman and Hauschild Mosley.

For this chapter, we simplify our notation by omitting the index $1$ of $X_1$ and $e_1$. Specifically, $\bt(R)$ is the unital associative $R$-algebra generated by $Y^{\pm 1}, X^{\pm 1}$ and~$e$ subject to the following relations: 
\begin{eqnarray}
Y^k \= \sum_{i=0}^{k-1} q_i Y^i \\
X - X^{-1} &=& \d(1-e) \label{eqn:xinv} \\
XYXY &=& YXYX \label{eqn:xyxy} \\
Xe &=& \l e \,\,\,=\,\,\, eX \label{eqn:xele}\\
YXYe &=& \lambda^{-1}e \,\,\,=\,\,\, eYXY \label{eqn:eyxy2} \\
eY^me &=& A_m e, \,\,\,\,\,\, \qquad \text{for }\, 0 \leq m \leq 
k-1. \label{eqn:eye2}
\end{eqnarray}
Recall, in Lemma \ref{lemma:eype}, we showed
\[
X Y^p e = \l Y^{-p} e - \d \pum{s} Y^{p-2s} e + \d \pum{s} A_{p-s} Y^{-s} e, \quad \text{for all } p \geq 0.
\]
Using this and the $k^\mathrm{th}$ order relation on $Y$, it is straightforward to show the left ideal of $\bt$ generated by $e$ is the span of $\{Y^i e \mid 0 \leq i \leq k-1\}$. 
 As a consequence of the results in Goodman and Hauschild \cite{GH06}, the set $\{Y^i e \mid i \in \mathbb{Z}\}$ is linearly independent in the affine BMW algebra and so it seems natural to expect that the set $\{Y^i e \mid 0 \leq i \leq k-1\}$ be linearly independent in the cyclotomic BMW algebra. For this purpose, we need to impose additional restrictions on our parameters $A_0, \o, A_{k-1},q_0, \o, q_{k-1}, q, \l$. 

%%%%%%%%%%%%%%%%%%%%%%%%%%%%%%%%%%%%%%%%%%%%%%%%%%%%%%%%%%%%

\section{Construction of a Generic Ground Ring with Admissible Parameters} \label{section:V}

Our aim now is to construct a $\bt$-module $V$ of rank $k$ which is isomorphic, as $\bt$-modules, to precisely the $R$-span of $\{Y^i e \mid 0 \leq i \leq k-1\}$, hence proving the $\{Y^i e \mid 0 \leq i \leq k-1\}$ are indeed linearly independent. 

Let $V$ be the free $R$-module of rank $k$ with basis $v_0, v_1, \ldots, v_{k-1}$. \\
Define a linear map $\b{Y}: V \rightarrow V$ by
\begin{eqnarray}
\b{Y}v_i &=& v_{i+1} \text{,\quad for } 0 \leq i \leq k-2, 
\label{eqn:Yv}\\
\b{Y}v_{k-1} &=& \sum_{i=0}^{k-1} q_i v_i.
\end{eqnarray}

\noi Since $q_0$ is invertible, this guarantees that $\b{Y}$ is invertible, with inverse
\begin{eqnarray}
\b{Y}^{-1} v_i &=& v_{i-1}\text{,\quad for } 1 \leq i \leq k-1, \label{eqn:iyl}\\
\b{Y}^{-1} v_0 &=& -q_0^{-1} \kum{i} q_{i+1} v_i. \label{eqn:iy0}
\end{eqnarray}
Also $\b{Y}^iv_0=v_i$ for $i = 0, \ldots, k-1$. Define $v_s=\b{Y}^sv_0$ for all integers $s$. Note that in particular, this means 
\begin{equation} \label{eqn:iys} 
 \b{Y}^{-1} v_s = v_{s-1} \text{,\quad for \emph{all} } s \in \mathbb{Z}.  
\end{equation}
The definition of $\b{Y}v_{k-1}$ gives
\[
  \sum_{l=0}^k q_l\b{Y}^lv_0=0.
\]
For any integer $i$, applying $\b{Y}^i$ to this gives
\begin{equation}\label{kthV}
  \sum_{l=0}^k q_l\b{Y}^l v_i = 0.
\end{equation}
Also $\b{Y}^l$ is invertible for any integer $l$ and, as $\{v_0, v_1, \ldots, v_{k-1}\}$ is a basis for $V$, the set
\begin{equation}\label{eqn:vbasis}
  \{\b{Y}^lv_0, \b{Y}^lv_1, \ldots, \b{Y}^lv_{k-1}\}
    =\{v_l, v_{l+1}, \ldots, v_{l+k-1}\}
\end{equation}
is also a basis for $V$.

\noi Now let us define linear maps $\b{X}, \b{E}: V \rightarrow V$ by  
\begin{eqnarray} 
\b{X}v_0 &=& \l v_0, \label{eqn:Xv0} \\
\b{X}v_1 &=& \l^{-1} \b{Y}^{-1} v_0, \label{eqn:Xv1}\\
\b{X}v_i &=& \b{Y}^{-1} \b{X}v_{i-1} - \d v_{i-2} + \d A_{i-1} \b{Y}^{-1}  v_0 \text{,\quad for } 2 \leq i \leq k-1, \qquad \label{eqn:Xvi} \\
\b{E}v_i &=& A_i v_0, \text{\quad for } 0 \leq i\leq k-1. \label{eqn:Ev}
\end{eqnarray}
Since $\l^{-1} = \l - \d + \d A_0$ in $R$, substituting $i=1$ into 
(\ref{eqn:Xvi}) reproduces (\ref{eqn:Xv1}). Also, note that the image of the map $\b{E}$ is $\im(\b{E})=\la{v_0}$, where remember here $\la{M}$ denotes the $R$-submodule (of $V$) spanned by the set $M$.
Furthermore, let us denote by $\b{W}$ the map $\b{X} - \d + 
\d\b{E}$. Thus (\ref{eqn:Xv1}) and (\ref{eqn:Xvi}) become 
\begin{equation} \label{eqn:Xvii} 
\b{X} v_i = \b{Y}^{-1} \b{W} 
v_{i-1}\text{,\quad for } 1\leq i \leq k-1. 
\end{equation}

 Our aim is to show that $V$ is actually a $\bt$-module, where the action of the 
generators $Y$, $X$ and $e$ are given by the maps $\b{Y}$, 
$\b{X}$ and $\b{E}$, respectively. In order to prove this, we require $\b{Y}^{-1} \b{W}\b{Y}^{-1} = \b{X}$. By equations (\ref{eqn:iyl}) and (\ref{eqn:Xvii}), this relation automatically holds on $v_1, \ldots, v_{k-1}$, so we need only ensure that 
\[
(\b{Y}^{-1}\b{W}\b{Y}^{-1}-\b{X})v_0=0.
\]

\noi For convenience, we write the left hand side relative to the basis $\{v_0,v_{-1},\ldots,v_{1-k}\}$. Let
\[
  (\b{Y}^{-1}\b{W}\b{Y}^{-1}-\b{X})v_0
    = - q_0^{-1}\beta v_0 - \sum_{l=1}^{k-1}q_0^{-1}h_lv_{-l},
\]
where $\beta$ and $h_l$ are elements of $R$. We now calculate $\beta$ and $h_l$, for $1 \leq l \leq k-1$, explicitly.

By definition, since $\l^{-1} = \l - \d + \d A_0$, we have $\b{W}v_0=\l^{-1}v_0$ and
\[
 \b{W}v_i = (\b{X}-\d+\d\b{E})v_i \rel{eqn:Xvii} \b{Y}^{-1}\b{W}v_{i-1}-\d v_i+\d A_iv_0,
\]
for $1\leq i\leq k-1$. 
It is then easy to verify the following by induction on $l$. \\
\textbf{Claim:} \quad For $0 \leq l \leq k-1$, 
\begin{equation} \label{eqn:Wvl}
\b{W}v_l=\l^{-1}v_{-l}+\d\sum_{i=1}^l(A_{l+1-i}v_{1-i}-v_{l-2i+2}).
\end{equation}
Indeed, when $l = 0$, the RHS of equation (\ref{eqn:Wvl}) is simply $\l^{-1} v_0$. Moreover, for $1\leq l\leq k-1$, if the formula holds for $l-1$ then 
\begin{eqnarray*}
 \b{W}v_l
  &=&\b{Y}^{-1}\b{W}v_{l-1}-\d v_l+\d A_lv_0\\
  &\stackrel{\text{ind. hypo.}}{=}&\b{Y}^{-1}\left(\l^{-1}v_{-l+1}+\d\sum_{i=1}^{l-1}(A_{l-i}v_{1-i}-v_{l-2i+1})\right)
      -\d v_l+\d A_lv_0\\
&\rel{eqn:iys}&\l^{-1}v_{-l}+\d\sum_{i=1}^{l-1}(A_{l-i}v_{-i}-v_{l-2i})
      -\d v_l+\d A_lv_0\\
    &=&\l^{-1}v_{-l}+\d\sum_{i=2}^l(A_{l+1-i}v_{1-i}-v_{l-2i+2})
      +\d(A_lv_0-v_l)\\
    &=&\l^{-1}v_{-l}+\d\sum_{i=1}^l(A_{l+1-i}v_{1-i}-v_{l-2i+2}),
\end{eqnarray*}
as required. Thus
\begin{align*}
  -q_0\b{Y}^{-1}\b{W}\b{Y}^{-1}v_0
    &\rel{eqn:iy0} \b{Y}^{-1}\b{W}\sum_{l=1}^kq_lv_{l-1}\\
    &\rel{eqn:Wvl} \b{Y}^{-1}\sum_{l=1}^kq_l
      \left[\l^{-1}v_{-l+1}+\d\sum_{i=1}^{l-1}(A_{l-i}v_{1-i}-v_{l-2i+1})\right]\\
    &\rel{eqn:iys} \sum_{l=1}^kq_l \left[\l^{-1}v_{-l}+\d\sum_{i=1}^{l-1}(A_{l-i}v_{-i}-v_{l-2i})\right]\\
    &\rel{kthV} \l^{-1}\sum_{l=1}^{k-1}q_lv_{-l}-\l^{-1}v_{-k}
      +\d\sum_{l=1}^k\sum_{i=1}^{l-1}q_l(A_{l-i}v_{-i}-v_{l-2i})\\
    &\h \l^{-1}\sum_{l=1}^{k-1}q_lv_{-l}+\l^{-1}q_0^{-1}\sum_{l=1}^kq_lv_{l-k}
      +\d\sum_{i=1}^{k-1}\sum_{l=i+1}^kq_l(A_{l-i}v_{-i}-v_{l-2i})\\
    &\h \l^{-1}\sum_{l=1}^{k-1}q_lv_{-l}+\l^{-1}q_0^{-1}\sum_{l=0}^{k-1}q_{k-l}v_{-l}
      +\d\sum_{i=1}^{k-1}\left(\sum_{l=i+1}^kq_lA_{l-i}\right)v_{-i}\\
    &\hphantom{\stackrel{(52)}{=}} -\, \d\sum_{i=1}^{k-1}\sum_{l=i+1}^kq_lv_{l-2i} \\
    &\h \l^{-1}\sum_{l=1}^{k-1}q_lv_{-l}+\l^{-1}q_0^{-1}\sum_{l=0}^{k-1}q_{k-l}v_{-l}
      +\d\sum_{l=1}^{k-1}\left(\sum_{r=1}^{k-l} q_{r+l}A_r \right)v_{-l}\\
    &\hphantom{\stackrel{(52)}{=}} -\, \d\sum_{i=1}^{k-1}\sum_{l=i+1}^kq_lv_{l-2i}.
\end{align*}

\noi For now we focus on the last term. Let $k = 2z - \epsilon$, where $\epsilon = 0$ if $k$ is even and $1$ if $k$ is odd. That is, let $z := \ceil{k}$. \\
Then 
\begin{align*}
  -\d\sum_{i=1}^{k-1}\sum_{l=i+1}^kq_lv_{l-2i}
    &\h -\d \sum_{i=z}^{k-1}\sum_{l=i+1}^kq_lv_{l-2i}
      -\d \sum_{i=1}^{z-1}\sum_{l=i+1}^kq_lv_{l-2i}\\
    &\rel{kthV} -\d \sum_{i=z}^{k-1}\sum_{l=i+1}^kq_lv_{l-2i}
      + \d \sum_{i=1}^{z-1}\sum_{l=0}^iq_lv_{l-2i}\\
    &\h -\d \sum_{i=z}^{k-1}\sum_{l=2i-k}^{i-1}q_{2i-l}v_{-l}
      +\d \sum_{i=1}^{z-1}\sum_{l=i}^{2i}q_{2i-l}v_{-l}\\
    &\h - \d\sum_{l=\e}^{k-2}\sum_{i=\max(l+1,z)}^{\lfloor\frac{l+k}{2}\rfloor} q_{2i-l}v_{-l}      + \d \sum_{l=1}^{k+\e-2}\sum_{i=\lceil\frac{l}{2}\rceil}^{\min(l,z-1)}q_{2i-l}v_{-l}.
\end{align*}

\noi Substituting this into the above expression for $-q_0\b{Y}^{-1}\b{W}\b{Y}^{-1}v_0$, we obtain the following:
\begin{align*}
  -q_0(\b{Y}^{-1}\b{W}\b{Y}^{-1}-\b{X})v_0\hspace{-25mm}\\
    &\rel{eqn:Xv0} q_0\l v_0+\l^{-1}\sum_{l=1}^{k-1}(q_l+q_0^{-1}q_{k-l})v_{-l}-\l^{-1}q_0^{-1}v_0
      +\d\sum_{l=1}^{k-1}\left(\sum_{r=1}^{k-l}q_{r+l}A_r\right)v_{-l}\\
    &\hphantom{\stackrel{(52)}{=}} \hspace{5mm}{}-\d\sum_{l=\e}^{k-2}\left(\sum_{i=\max(l+1,z)}^{\lfloor\frac{l+k}{2}\rfloor}\hspace{-5mm}q_{2i-l}\right)v_{-l}
      +\d\sum_{l=1}^{k+\e-2}\left(\sum_{i=\lceil\frac{l}{2}\rceil}^{\min(l,z-1)}\hspace{-5mm}q_{2i-l}\right)v_{-l}\\
    &\h \beta v_0+\sum_{l=1}^{k-1}h_lv_{-l}.
\end{align*}

\noi Observe that the second last inner sum above is zero when $l=k-1$, as is the last, provided $\e=0$. We can therefore change the upper index of the outer sum to $k-1$ in both. Now, equating coefficients in the above equation implies that
\begin{equation} \label{eqn:alpha}
  \beta = q_0\l-q_0^{-1}\l^{-1}+(1-\e)\d
\end{equation}
and for $l = 1, \o, k-1$,
\begin{equation} \label{eqn:hl}
  h_l=\l^{-1}(q_l+q_0^{-1}q_{k-l})
    +\d\left[\sum_{r=1}^{k-l}q_{r+l}A_r
    -\sum_{i=\max(l+1,z)}^{\lfloor\frac{l+k}{2}\rfloor}\hspace{-5mm}q_{2i-l}
    +\sum_{i=\lceil\frac{l}{2}\rceil}^{\min(l,z-1)}\hspace{-5mm}q_{2i-l}\right].
\end{equation}

We now continue to study these equations in further detail. Our aim here is to provide a ``generic'' ground ring $R_0$ for which $V$ is a $\bt(R_0)$-module. As discussed above, we require all $h_l = 0$ in $R$ in order to show $V$ is a $\bt(R)$-module. However, in the following calculations, we demonstrate that a particular linear combination of these $h_l$'s is $\d$ multiplied by an element $h_l'$ of $R$. Therefore, it would make sense to impose stronger conditions on the parameters of $R$ by replacing some of the equations $h_l =0$ with $h_l' = 0$. For convenience, we define
\begin{equation} \label{eqn:h0}
h_0 := \l-\l^{-1}+\d(A_0-1).
\end{equation}

Now suppose $1\leq l\leq z-\e$. Then $k-l=2z-\e-l\geq z$, so
\begin{align*}
    h_{k-l}
        \- \l^{-1}(q_{k-l}+q_0^{-1}q_l)
            +\d\left[\sum_{r=1}^lq_{r+k-l}A_r
            -\hspace{-9mm}\sum_{i=\max(k-l+1,z)}^{\lfloor k-\frac{l}{2}\rfloor}\hspace{-8mm}q_{2i-k+l}
            +\hspace{-7mm}\sum_{i=\lceil\frac{k-l}{2}\rceil}^{\min(k-l,z-1)}\hspace{-6mm}q_{2i-k+l}\right]\\
        \- \l^{-1}(q_{k-l}+q_0^{-1}q_l)
            +\d\left[\sum_{r=1}^lq_{r+k-l}A_r
            -\hspace{-4mm}\sum_{i=k-l+1}^{k-\lceil\frac{l}{2}\rceil}\hspace{-3mm}q_{2i-k+l}
            +\hspace{-5mm}\sum_{i=k-\lfloor\frac{l+k}{2}\rfloor}^{z-1}\hspace{-4mm}q_{2i-k+l}\right]\\
        \- \l^{-1}(q_{k-l}+q_0^{-1}q_l)
            +\d\left[\sum_{r=1}^lq_{r+k-l}A_r
            -\hspace{-2mm}\sum_{i=\lceil\frac{l}{2}\rceil}^{l-1}\hspace{-1mm}q_{k-2i+l}
            +\hspace{-4mm}\sum_{i=k-z+1}^{\lfloor\frac{l+k}{2}\rfloor}\hspace{-3mm}q_{k-2i+l}\right],
\end{align*}
using the change of summation $i \mapsto k-i$. Since $k = 2z - \e$, $k-z+1=z+(1-\e)$, which allows us to rewrite the last sum above as
\[
\sum_{i=k-z+1}^{\lfloor\frac{l+k}{2}\rfloor}\hspace{-3mm}q_{k-2i+l}
=\sum_{i=z}^{\lfloor\frac{l+k}{2}\rfloor}\hspace{-1mm}q_{k-2i+l}-(1-\e)q_{k-2z+l}
        =\sum_{i=z}^{\lfloor\frac{l+k}{2}\rfloor}\hspace{-1mm}q_{k-2i+l}-(1-\e)q_l.
\]
Hence
\begin{align}
    h_{k-l} \- \l^{-1}(q_{k-l}+q_0^{-1}q_l) \nonumber\\
        &\phantom{=}\,\, +\, \d\left[\sum_{r=1}^lq_{r+k-l}A_r
            -\hspace{-2mm}\sum_{i=\lceil\frac{l}{2}\rceil}^{l-1}\hspace{-1mm}q_{k-2i+l}
            +\sum_{i=z}^{\lfloor\frac{l+k}{2}\rfloor}q_{k-2i+l}-(1-\e)q_l\right]. \label{hk-l}
\end{align}
In addition, (\ref{eqn:alpha}) - (\ref{eqn:h0}) implies that
\begin{align*}
    h_l-\beta q_0^{-1}q_l+h_0q_l\hspace{-25mm}\\
        \- \l^{-1}q_l+\l^{-1}q_0^{-1}q_{k-l}+\d\left[\sum_{r=1}^{k-l}q_{r+l}A_r
-\sum_{i=\max(l+1,z)}^{\lfloor\frac{l+k}{2}\rfloor}\hspace{-5mm}q_{2i-l}
+\sum_{i=\lceil\frac{l}{2}\rceil}^{\min(l,z-1)}\hspace{-5mm}q_{2i-l}\right]\\
        \f -\l q_l+q_0^{-2}\l^{-1}q_l-(1-\e)\d q_0^{-1}q_l+\l q_l-\l^{-1}q_l+\d(A_0-1)q_l\\
        \- \d\Bigg[\sum_{r=1}^{k-l}q_{r+l}A_r+q_lA_0
-\sum_{i=\max(l+1,z)}^{\lfloor\frac{l+k}{2}\rfloor}\hspace{-5mm}q_{2i-l}
+\sum_{i=\lceil\frac{l}{2}\rceil}^{\min(l,z-1)}\hspace{-5mm}q_{2i-l}\\
\f \hspace{20mm}-q_l-(1-\e)q_0^{-1}q_l\Bigg]+q_0^{-1}\l^{-1}(q_{k-l}+q_0^{-1}q_l).
\end{align*}

\noi Now $l\leq z-\e\leq z$. If $l = z$, then $k$ is even and
\[
    - \hspace{-5mm} \sum_{i=\max(l+1,z)}^{\lfloor\frac{l+k}{2}\rfloor}\hspace{-5mm}q_{2i-l}
            + \hspace{-5mm}\sum_{i=\lceil\frac{l}{2}\rceil}^{\min(l,z-1)}\hspace{-5mm}q_{2i-l}
        = - \hspace{-2mm}\sum_{i=z+1}^{\lfloor\frac{l+k}{2}\rfloor}\hspace{-2mm}q_{2i-l}
            + \sum_{i=\lceil\frac{l}{2}\rceil}^{z-1}q_{2i-l}
        = - \sum_{i=z}^{\lfloor\frac{l+k}{2}\rfloor}q_{2i-l}
            +\sum_{i=\lceil\frac{l}{2}\rceil}^zq_{2i-l}.
\]
On the other hand, if $l<z$, we have
\[
     - \hspace{-5mm} \sum_{i=\max(l+1,z)}^{\lfloor\frac{l+k}{2}\rfloor}\hspace{-5mm}q_{2i-l}
            + \hspace{-5mm}\sum_{i=\lceil\frac{l}{2}\rceil}^{\min(l,z-1)}\hspace{-5mm}q_{2i-l}
        = - \sum_{i=z}^{\lfloor\frac{l+k}{2}\rfloor}q_{2i-l}
            + \sum_{i=\lceil\frac{l}{2}\rceil}^lq_{2i-l}.
\]
Thus, in either case, the above becomes
\begin{align*}
    h_l-\beta q_0^{-1}q_l+h_0q_l\hspace{-15mm}\\
        \- \d\left[\sum_{r=1}^{k-l}q_{r+l}A_r+q_lA_0
            -\sum_{i=z}^{\lfloor\frac{l+k}{2}\rfloor}q_{2i-l}
            +\sum_{i=\lceil\frac{l}{2}\rceil}^lq_{2i-l}-q_l-(1-\e)q_0^{-1}q_l\right]\\
        \f +q_0^{-1}\l^{-1}(q_{k-l}+q_0^{-1}q_l)\\
        \- \d\left[\sum_{r=0}^{k-l}q_{r+l}A_r
            -\sum_{i=z}^{\lfloor\frac{l+k}{2}\rfloor}q_{2i-l}
            +\sum_{i=\lceil\frac{l}{2}\rceil}^{l-1}q_{2i-l}-(1-\e)q_0^{-1}q_l\right]\\
        \f +q_0^{-1}\l^{-1}(q_{k-l}+q_0^{-1}q_l).
\end{align*}
Therefore, using equation (\ref{hk-l}), we obtain
\begin{align*}
    q_0^{-1}h_{k-l}-h_l+\beta q_0^{-1}q_l-h_0q_l \hspace{-40mm}\\
    \- q_0^{-1} \l^{-1}(q_{k-l}+q_0^{-1}q_l) + \d q_0^{-1} \left[\sum_{r=1}^lq_{r+k-l}A_r
            -\hspace{-2mm}\sum_{i=\lceil\frac{l}{2}\rceil}^{l-1}\hspace{-2mm}q_{k-2i+l}
            +\sum_{i=z}^{\lfloor\frac{l+k}{2}\rfloor}q_{k-2i+l}-(1-\e)q_l\right] \\
 \f - \d\left[\sum_{r=0}^{k-l}q_{r+l}A_r
            -\sum_{i=z}^{\lfloor\frac{l+k}{2}\rfloor}q_{2i-l}
            +\sum_{i=\lceil\frac{l}{2}\rceil}^{l-1}q_{2i-l}-(1-\e)q_0^{-1}q_l\right]  - q_0^{-1}\l^{-1}(q_{k-l}+q_0^{-1}q_l).
\end{align*}
Hence we have shown that, for $1\leq l \leq z - \e$,
\begin{equation}\label{hiprime}
    q_0^{-1}h_{k-l}-h_l+\beta q_0^{-1}q_l-h_0q_l=\d h_l',
\end{equation}
where
\begin{eqnarray}
    h_l'
        &:=&\sum_{r=1}^lq_0^{-1}q_{r+k-l}A_r-\sum_{r=0}^{k-l}q_{r+l}A_r\nonumber\\
        &&{}-\sum_{i=\lceil\frac{l}{2}\rceil}^{l-1}(q_0^{-1}q_{k-2i+l}+q_{2i-l})
            +\sum_{i=z}^{\lfloor\frac{l+k}{2}\rfloor}(q_0^{-1}q_{k-2i+l}+q_{2i-l}).\label{hiprime2}
\end{eqnarray}

Before proceeding we first prove a simple lemma which will be used in a later proof to show $\d$ is not a zero divisor in certain rings.
\begin{prop}\label{zerodiv}
Suppose a commutative ring $S$ contains elements $a$ and $b$, such that $a$ is not a zero divisor in $S$ and 
$b+aS$ is not a zero divisor in $S/aS$. Then $a+bS$ is not a zero divisor in $S/bS$.
\end{prop}
\begin{proof}
Suppose $(a+bS)(x+bS)=0$ for some $x+bS\in S/bS$. Then $ax\in bS$, so $ax=by$ for some $y\in S$. Thus, as an element of $S/aS$, $(b+aS)(y+aS)=0$. This implies $y+aS=0$ since $b+aS$ is not a zero divisor in $S/aS$, by assumption. Hence, $y=az$ for some $z\in S$. Furthermore, $ax = by = azb$, so $x=zb$ since $a$ is not a zero divisor in $S$. Therefore $x+bS=0$ and $a+bS$ is not a zero divisor in $S/bS$.
\end{proof}

It is easy to see that $\beta$ always factorises as $\beta_+\beta_-$, where if $k$ is odd,
\[
\beta_+=q_0\l-1 \qquad \text{and}\qquad \beta_-=q_0^{-1}\l^{-1}+1
\]
and when $k$ is even,
\[
    \beta_+=q_0\l-q^{-1} \qquad \text{and}\qquad \beta_-=qq_0^{-1}\l^{-1}+1.
\]
For convenience, we denote $\beta_0 := \beta$. At this point, we wish to remind the reader that for a subset $J \subseteq R$, we write $\la{J}_R$ to mean the \emph{ideal} generated by $J$ in $R$. Sometimes the subscript may be omitted only if it is clear in the current context.

\begin{lemma} \label{lemma:Rsigma}
Let $\Omega := \Z[q^{\pm 1},\l^{\pm 1},q_0^{\pm 1},q_1,\ldots,q_{k-1},A_0^{\pm 1},A_1,\ldots,A_{k-1}].$ \\
For $\sigma\in\{0,+,-\}$, let
\[
    I_\sigma := \langle\beta_\sigma,h_0,h_1,\ldots,h_{z-\e},h_1',h_2',\ldots,h_{z-\e}'\rangle_{_{\Omega}}\subseteq\Omega
\]
and
\[
    R_\sigma := \Omega/I_\sigma.
\]
Then 
\begin{enumerate}
\item[(a)] the image of $\d$ is not a zero divisor in $R_\sigma$, for $\sigma\in\{0,+,-\}$;
\item[(b)] for $\sigma=\pm$,
\[
    R_\sigma[\d^{-1}]\cong\Z[q^{\pm 1},\l^{\pm 1},q_1,\ldots,q_{k-1}][\d^{-1}] [(\delta^{-1} \lambda^{-1} - \delta^{-1} \lambda + 1)^{-1}];
\]
\item[(c)] for $\sigma=\pm$, the ring $R_\sigma$ is an integral domain;
\item[(d)] $I_0=I_+\cap I_-$.
\end{enumerate}
\end{lemma}

\begin{proof}
\textbf{(a)} Since $\d= q - q^{-1} = q^{-1}(q-1)(q+1)$, to prove (a) it suffices to show that $q+\tau$ is not a zero divisor, for $\tau=\pm 1$. For $1\leq l\leq k-1$, let 
\begin{equation}\label{Bl}
    B_l :=\sum_{r=1}^{k-l}q_{r+l}A_r
        \,-\hspace{-5mm} \sum_{i=\max(l+1,z)}^{\lfloor\frac{l+k}{2}\rfloor}\hspace{-5mm}q_{2i-l}
        \,+\hspace{-3mm}\sum_{i=\lceil\frac{l}{2}\rceil}^{\min(l,z-1)}\hspace{-5mm}q_{2i-l}\in\Omega.
\end{equation}
Then (\ref{eqn:hl}) says that
\begin{equation} \label{eqn:hBl}
    h_l=\l^{-1}(q_l+q_0^{-1}q_{k-l})+\d B_l, \quad \text{for } 1\leq l\leq k-1.
\end{equation}
Over $\Z[q^{\pm 1},\l^{\pm 1},q_0^{\pm 1},q_1,\ldots,q_{k-1}]$, the $B_l$ are related to 
the $A_l$ by an affine linear transformation; specifically, the column vector $(B_l)$, where $l = 1, \o,k-1$, is equal to a matrix $(s_{lr})_{l,r = 1}^{k-1}$ multiplied by the column vector $(A_l)$ plus a column vector of $q_i$'s. Moreover, $s_{lr} = 0$, unless $l + r \leq k$ and $s_{lr} = q_k = -1$, when $l+r = k$. Thus $(s_{lr})$ is triangular, with diagonal entries $q_k=-1$, 
so it is invertible. Therefore we may identify $\Omega$ with the polynomial ring
\begin{equation}\label{eqn:Lambdaalt}
\Omega=\Z[q^{\pm 1},\l^{\pm 1},q_0^{\pm 1},q_1,\ldots,q_{k-1},A_0^{\pm 1},B_1,B_2,\ldots,B_{k-1}].
\end{equation}
Now, when $1 \leq l \leq z-1$, (\ref{hiprime2}) and (\ref{Bl}) implies that
\begin{align*}
  h_l' \- q_0^{-1} B_{k-l} + q_0^{-1} \sum_{i=\max(k-l+1,z)}^{\lfloor\frac{2k-l}{2}\rfloor}\hspace{-5mm}q_{2i-k+l}
         - q_0^{-1} \sum_{i=\lceil\frac{k-l}{2}\rceil}^{\min(k-l,z-1)}\hspace{-5mm}q_{2i-k+l} \\
\f -\sum_{r=0}^{k-l}q_{r+l}A_r -\sum_{i=\lceil\frac{l}{2}\rceil}^{l-1}(q_0^{-1}q_{k-2i+l}+q_{2i-l})
            +\sum_{i=z}^{\lfloor\frac{l+k}{2}\rfloor}(q_0^{-1}q_{k-2i+l}+q_{2i-l}) \\
&\rel{Bl}  q_0^{-1} B_{k-l} + q_0^{-1} \sum_{i=\max(k-l+1,z)}^{\lfloor\frac{2k-l}{2}\rfloor}\hspace{-5mm}q_{2i-k+l}
         - q_0^{-1} \sum_{i=\lceil\frac{k-l}{2}\rceil}^{\min(k-l,z-1)}\hspace{-5mm}q_{2i-k+l} \\
\f -B_l - q_l A_0 - \sum_{i=\max(l+1,z)}^{\lfloor\frac{l+k}{2}\rfloor}\hspace{-5mm}q_{2i-l}
        +\sum_{i=\lceil\frac{l}{2}\rceil}^{\min(l,z-1)}\hspace{-5mm}q_{2i-l} \\
\f -\sum_{i=\lceil\frac{l}{2}\rceil}^{l-1}(q_0^{-1}q_{k-2i+l}+q_{2i-l})
            +\sum_{i=z}^{\lfloor\frac{l+k}{2}\rfloor}(q_0^{-1}q_{k-2i+l}+q_{2i-l})
\end{align*}
Hence
\[
    h_l'\in q_0^{-1}B_{k-l}+\Z[q^{\pm 1},\l^{\pm 1},q_0^{\pm 1},q_1,\ldots,q_{k-1},A_0^{\pm 1},B_1,B_2,\ldots,B_{z-1}],
\]
for $1\leq l\leq z-1$. Thus
\[
    \Omega_1 := \Omega/\langle h_1',h_2',\ldots,h_{z-1}'\rangle_{_{\Omega}}
\cong\Z[q^{\pm 1},\l^{\pm 1},q_0^{\pm 1},q_1,\ldots,q_{k-1},A_0^{\pm 1},B_1,B_2,\ldots,B_{z-\e}].
\]
Indeed, if $1 \leq l \leq z-1$, then if $k$ is even (hence $\e = 0$ and $z+1 \leq k-l \leq k-1$), quotienting by the $h_l'$ expresses the elements $B_{k-1}, \o, B_{z+1}$, respectively, as elements of the ring $\Z[q^{\pm 1},\l^{\pm 1},q_0^{\pm 1},q_1,\ldots,q_{k-1},A_0^{\pm 1},B_1,B_2,\ldots, B_{z-1}]$; similarly, if $k$ is odd (hence $\e = 1$ and $z \leq k-l \leq k-1$), then $B_z, \o, B_{k-1}$ may also be expressed in the quotient as elements of the ring $\Z[q^{\pm 1},\l^{\pm 1},q_0^{\pm 1},q_1,\ldots,q_{k-1},A_0^{\pm 1},B_1,B_2,\ldots, B_{z-1}]$ .

In particular, $q+\tau$ is not a zero divisor in $\Omega_1$. By using Proposition \ref{zerodiv} recursively, we aim to show that $q+\tau$ does not become a zero divisor, as we quotient by further generators of $I_\sigma$. \\
If we set $q + \tau = 0$ then $\d = 0$, hence by (\ref{eqn:hl}), $h_i = \l^{-1}(q_i + q_0^{-1} q_{k-i})$, for any $i$. So for any $1\leq l\leq z-1$, we have that
\begin{eqnarray*}
    \Omega/\langle h_1',\ldots,h_{z-1}',h_1,\ldots,h_{l-1},q+\tau\rangle_{_{\Omega}}\hspace{-50mm}\\
        &\cong&\Omega_1/\langle(q_1+q_0^{-1}q_{k-1}),(q_2+q_0^{-1}q_{k-2})\ldots, (q_{l-1}+q_0^{-1}q_{k-l+1}),q+\tau\rangle_{_{\Omega_1}}\\
        &\cong&\Z[\l^{\pm 1},q_0^{\pm 1},q_1,\ldots,q_{k-l},A_0^{\pm 1},B_1,B_2,\ldots,B_{z-\e}].
\end{eqnarray*}
Certainly $h_l\equiv\l^{-1}(q_l+q_0^{-1}q_{k-l})$ is not a zero divisor in this ring, so repeated application of Proposition \ref{zerodiv} proves that $q+\tau$ is not a zero divisor in
\[
    \Omega_2 := \Omega_1/\la{h_1, \o, h_{z-1}}_{_{\Omega_1}} = \Omega/\langle h_1',h_2',\ldots,h_{z-1}',h_1,h_2,\ldots,h_{z-1}\rangle_{_{\Omega}}.
\]
Moreover, the above argument (with $l = z$) says that
\[
    \Omega_2/\langle q+\tau\rangle_{_{\Omega_2}}\cong\Z[\l^{\pm 1},q_0^{\pm 1},q_1,\ldots,q_{z-\e},A_0^{\pm 1},B_1,B_2,\ldots,B_{z-\e}].
\]
Observe that $\langle h_0,\beta_\sigma\rangle_{_{\Omega}}=\langle h_0',\beta_\sigma\rangle_{_{\Omega}}$, where
\[
    h_0':=h_0-q_0^{-1}\beta=\l^{-1}(q_0^{-2}-1)+\d(A_0-1-(1-\e)q_0^{-1}).
\]
Suppose first that $k$ is odd, then $R_\sigma = \Omega_2/\la{h_0,\beta_\sigma} = \Omega_2/\la{h_0',\beta_\sigma}$. Certainly, we know that $h_0'\equiv\l^{-1}(q_0^{-2}-1)$ is not 
a zero divisor in the polynomial ring $\Omega_2/\langle q+\tau\rangle_{_{\Omega_2}}$. Moreover,
\begin{eqnarray*}
    \Omega_2/\langle h_0',q+\tau\rangle_{_{\Omega_2}}
        &\cong&\Z[\l^{\pm 1},q_0^{\pm 1},q_1,\ldots,q_{z-\e},A_0^{\pm 1},B_1,B_2,\ldots,B_{z-\e}]/\langle q_0^{-2}-1\rangle\\
        &=&(\Z[q_0^{\pm 1},q_1,\ldots,q_{z-\e},A_0^{\pm 1},B_1,B_2,\ldots,B_{z-\e}]/\langle q_0^{-2}-1\rangle)[\l^{\pm 1}]
\end{eqnarray*}
is a Laurent polynomial ring in $\l$. Now, in this ring, $\beta_\sigma$ is one of $q_0\l-q_0^{-1}\l^{-1}$, $q_0\l-1$ or $q_0^{-1}\l^{-1}-1$. In every case, 
it has an invertible leading coefficient as a polynomial in $\l$, so it is not a zero divisor in  $\Omega_2/\langle h_0',q+\tau\rangle_{_{\Omega_2}}$. Now, because $q+\tau$ is not a zero divisor in $\Omega_2$ and $h_0'$ is not 
a zero divisor in the polynomial ring $\Omega_2/\langle q+\tau\rangle_{_{\Omega_2}}$, Proposition \ref{zerodiv} implies that $q+\tau$ is not a zero divisor in $\Omega_2/\la{h_0'}_{_{\Omega_2}}$. Then applying Proposition \ref{zerodiv} again shows that $q+\tau$ is not a zero divisor in $\Omega_2/\la{h_0',\beta_\sigma}_{_{\Omega_2}} = R_\sigma$, so we have proven (a) when $k$ is odd.

Now suppose $k$ is even, then $R_\sigma = \Omega_2/\la{\beta_\sigma,h_0,h_z,h_z'}_{_{\Omega_2}}$. Equation (\ref{hiprime}) then gives
\begin{align*}
    \d h_z'
        &\h (q_0^{-1}-1)h_z+\beta q_0^{-1}q_z-h_0q_z\\
        &\rel{eqn:hBl} \d\left[(q_0^{-1}-1)B_z-(A_0-1-q_0^{-1})q_z\right] \text{ in } \Omega.
\end{align*}
Since $\d$ is not a zero divisor in $\Omega$, this implies
\begin{equation}\label{hmprimep4}
    h_z'=(q_0^{-1}-1)B_z-(A_0-1-q_0^{-1})q_z.
\end{equation}
We first aim to show that $q+\tau$ is not a zero divisor in $\Omega_2/\langle h_z,h_z',h_0'\rangle_{_{\Omega_2}}$. Suppose
\[
    (q+\tau)x=ah_z+bh_z'+ch_0',
\]
for some $a,b,c\in\Omega_2$. For $d\in\Omega_2$, let $\bar d$ denote the image of $d$ in $\Omega_2/\langle q+\tau\rangle_{_{\Omega_2}}$. Then $\bar{h}_z = \l^{-1} (q_z + q_0^{-1} q_z)$, $\bar{h}_z' = (q_0^{-1}-1)B_z-(A_0-1-q_0^{-1})q_z$, and $h_0' = \l^{-1}(q_0^{-2} - 1)$. So we
have
\[
    \bar a\l^{-1}(q_0^{-1}+1)q_z+\bar b\left[(q_0^{-1}-1)B_z-(A_0-1-q_0^{-1})q_z\right]+\bar c\l^{-1}(q_0^{-2}-1)=0.
\]
In particular, $q_0^{-1}+1$ divides $\bar b\left[(q_0^{-1}-1)B_z-(A_0-1-q_0^{-1})q_z\right]$. Because $\Omega_2/\langle q+\tau\rangle_{_{\Omega_2}}$ is just 
the polynomial ring $\Z[\l^{\pm 1},q_0^{\pm 1},q_1,\ldots,q_z, A_0^{\pm 1},B_1,B_2,\ldots,B_z]$ (shown above), $q_0^{-1}+1$ divides $\bar b$. Thus $\bar b=\bar b_1\l^{-1}(q_0^{-1}+1)$, for some $b_1\in\Omega_2$, and so 
\[
    \bar a q_z+\bar b_1\left[(q_0^{-1}-1)B_z-(A_0-1-q_0^{-1})q_z\right]+\bar c(q_0^{-1}-1)=0.
\]
Rearranging then gives
\[
    (q_0^{-1}-1)\left[\bar b_1 B_z+\bar c\right]=q_z\left[\bar b_1(A_0-1-q_0^{-1})-\bar a\right].
\]
Now $q_0^{-1}-1$ and $q_z$ are coprime as elements of $\Omega_2/\la{q+\tau}_{\Omega_2}$, so there exists a $c_1\in\Omega_2$ such that \begin{align*}
    \bar b_1 B_z+\bar c \- \bar c_1q_z\quad \text{ and}\\
    \bar b_1(A_0-1-q_0^{-1})-\bar a \- \bar c_1(q_0^{-1}-1).
\end{align*}
We may now write
\begin{align*}
    a \- b_1(A_0-1-q_0^{-1})-c_1(q_0^{-1}-1)+(q+\tau)a_2,\\
    b \- \l^{-1}(q_0^{-1}+1)b_1+(q+\tau)b_2 \quad \text{ and}\\
    c \- c_1q_z-b_1B_z+(q+\tau)c_2,
\end{align*}
for some $a_2,\,b_2,\,c_2\in\Omega_2$. Thus
\begin{align*}
    (q+\tau)x
        \- ah_z+bh_z'+ch_0'\\
        \- b_1\left[(A_0-1-q_0^{-1})h_z+\l^{-1}(q_0^{-1}+1)h_z'-B_zh_0'\right]\\
        \f +c_1\left[q_zh_0'-(q_0^{-1}-1)h_z\right]
            +(q+\tau)(a_2h_z+b_2h_z'+c_2h_0').
\end{align*}
It is straightforward to verify that $\left[(A_0-1-q_0^{-1})h_z+\l^{-1}(q_0^{-1}+1)h_z'-B_zh_0'\right] = 0$, using the definition of $h_0'$ and equations (\ref{eqn:hBl}) and (\ref{hmprimep4}). Also, by definition of $h_z'$ and $h_0'$, we know that $\d h_z' = (q_0^{-1}-1)h_z - q_zh_0'$. Hence the above reduces to
\begin{align*}
(q+\tau)x \- - \d c_1h_z'+(q+\tau)(a_2h_z+b_2h_z'+c_2h_0') \\
\- (q+\tau) \left[ -q^{-1} (q-\tau) c_1 h_z' + a_2h_z+b_2h_z'+c_2h_0' \right].
\end{align*}
as $\d = q^{-1}(q+\tau)(q-\tau)$.
Earlier we showed that $q+\tau$ is not a zero divisor in $\Omega_2$, so
\[
    x=a_2h_z+b_2h_z'+c_2h_0'- q^{-1}(q-\tau)c_1h_z'\in\langle h_z,h_z',h_0'\rangle_{_{\Omega_2}}.
\]
That is, $q+\tau$ is not a zero divisor in $\Omega_2/\langle h_z,h_z',h_0'\rangle_{_{\Omega_2}}$. Finally, by a similar reasoning as in the odd case, $\beta_\sigma$ is not a zero divisor in $\Omega_2/\langle h_z,h_z',h_0',q+\tau\rangle_{_{\Omega_2}}$. A final application of Proposition \ref{zerodiv} 
therefore shows that $q+\tau$ is not a zero divisor in $\Omega_2/\la{\beta_\sigma,h_0,h_z,h_z'}_{_{\Omega_2}}$, thereby completing the proof of (a).

%%%%%%%%%%%%%%%%%%%%%%%%%%%%%%%%%%%%%%%%%%

\textbf{(b)} For the moment, $\sigma \in \{0,+,-\}$. We now give a concrete realisation of the ring $R_\sigma[\d^{-1}]$. Let $I_\sigma[\d^{-1}]$ denote the ideal of $\Omega[\d^{-1}]$ generated by $I_\sigma$.
A standard argument shows that
\[
    R_\sigma[\d^{-1}]=(\Omega/I_\sigma)[\d^{-1}]\cong \Omega[\d^{-1}]/I_\sigma[\d^{-1}].
\]

By equation (\ref{hiprime}), $h_{k-l} = \d q_0 h_l' + q_0 h_l - \beta q_l - q_0 h_0q_l \in I_\sigma[\d^{-1}]$,  for $1\leq l\leq z-\e$. Thus the ideal in $\Omega[\d^{-1}]$ generated by $\beta_\sigma$ and all $h_0,h_1,\o,h_{k-1}$ must also be contained in $I_\sigma[\d^{-1}]$.
Conversely, since $\d$ is invertible in $\Omega[\d^{-1}]$, equation (\ref{hiprime}) also shows that
\[
    h_l'\in\langle\beta_\sigma,h_0,h_1,\ldots,h_{k-1}\rangle_{_{\Omega[\d^{-1}]}}\text{,\quad for } 1\leq l\leq z-\e. 
\]
Thus
\[
    I_\sigma[\d^{-1}]=\langle\beta_\sigma,h_0,h_1,\ldots,h_{k-1}\rangle_{_{\Omega[\d^{-1}]}}.
\]
In $\Omega[\d^{-1}]$, the equations $h_l=0$ are equivalent to
\begin{eqnarray*}
    A_0&=& \d^{-1}\l^{-1}-\d^{-1}\l + 1,\\
    \text{and\quad}B_l&=&-\d^{-1}\l^{-1}(q_l+q_0^{-1}q_{k-l})\text{,\quad  for }1\leq l\leq k-1.
\end{eqnarray*}
Let $\Omega_3 := \Omega[\d^{-1}]/\langle h_0,h_1,\ldots,h_{k-1}\rangle_{_{\Omega[\d^{-1}]}}$. Then we have expressed $A_0$ and $B_1, \ldots, B_{k-1}$ in $\Omega_3$ as polynomials in $q^{\pm 1},\l^{\pm 1},q_0^{\pm 1},q_1,\ldots,q_{k-1}$ and $\d^{-1}$.
Therefore, by (\ref{eqn:Lambdaalt}),
\[
 \Omega_3 \cong \Z[q^{\pm 1},\l^{\pm 1},q_0^{\pm 1},q_1,\ldots,q_{k-1}][\d^{-1}] [(\delta^{-1} \lambda^{-1} - \delta^{-1} \lambda + 1)^{-1}].
\]
Moreover,
\[
R_\sigma[\d^{-1}] \cong \Omega[\d^{-1}]/I_\sigma[\d^{-1}] = \Omega[\d^{-1}]/\la{\beta_\sigma,h_0,h_1,\ldots,h_{k-1}}_{_{\Omega[\d^{-1}]}} \cong
\Omega_3/\langle\beta_\sigma\rangle_{_{\Omega_3}}.
\]
Now suppose $\sigma=\pm$. Then $\beta_\sigma$ can be ``solved'' for $q_0^{\pm 1}$, so
\[
    R_\sigma[\d^{-1}]\cong\Z[q^{\pm 1},\l^{\pm 1},q_1,\ldots,q_{k-1}][\d^{-1}] [(\delta^{-1} \lambda^{-1} - \delta^{-1} \lambda + 1)^{-1}],
\]
completing the proof of (b).

\textbf{(c)} Observe that the ring $\Z[q^{\pm 1},\l^{\pm 1},q_1,\ldots,q_{k-1}][\d^{-1}] [(\delta^{-1} \lambda^{-1} - \delta^{-1} \lambda + 1)^{-1}]$ above is obtained from an integral domain via localisation, hence $R_\sigma[\d^{-1}]$ is also an integral domain. 
Now, we have already proven in part (a) that $\d$ is not a zero divisor in $R_\sigma$, therefore the map $R_\sigma\rightarrow R_\sigma[\d^{-1}]$ is injective. Since $R_\sigma$ embeds into $R_\sigma[\d^{-1}]$, statement (c) now follows immediately.

\textbf{(d)} Because $\beta_0 = \beta_+ \beta_-$, it is clear that $\beta_0 \in I_+ \cap I_-$, thus all generators of the ideal $I_0$ are in $I_+ \cap I_-$. Hence $I_0 \subseteq I_+ \cap I_-$. 

Finally, note that $\Omega_3 \cong \Z[q^{\pm 1},\l^{\pm 1},q_0^{\pm 1},q_1,\ldots,q_{k-1}][\d^{-1}] [(\delta^{-1} \lambda^{-1} - \delta^{-1} \lambda + 1)^{-1}]$ is obtained from a UFD (unique factorisation domain) by localisation, and is therefore also a UFD. Suppose that 
$x\in I_+\cap I_-\subseteq\Omega$. Then $x$ vanishes in $R_\pm$ and hence in $R_\pm[\d^{-1}] \cong \Omega_3/\la{\beta_\pm}_{_{\Omega_3}}$. So the image $\bar x$ of $x$ in $\Omega_3$ satisfies
\[
    \bar x\in\langle\beta_+\rangle_{_{\Omega_3}} \cap \langle\beta_-\rangle_{_{\Omega_3}} = \langle\beta_+\beta_-\rangle_{_{\Omega_3}} = \langle\beta\rangle_{_{\Omega_3}}, 
\]
since $\beta_+$ and $\beta_-$ are coprime in $\Omega_3$. Thus $x$ maps to $0$ in $R_0[\d^{-1}] \cong \Omega_3/\la{\beta}_{_{\Omega_3}}$. Since $R_0$ embeds into $R_0[\d^{-1}]$, the image 
of $x$ in $R_0$ must also be $0$. That is, $x\in I_0$, hence $I_+\cap I_-\subseteq I_0$, completing the proof of (d).
\end{proof}

\begin{defn} \label{defn:adm}
Let $R$ be as in the definition of $\bnk$ (see Definition \ref{defn:bnk}). The family of parameters $\left( A_0, \ldots, A_{k-1}, q_0, \ldots, q_{k-1}, q, \l\right)$ is called \textbf{\emph{admissible}} if 
\[
\beta=h_0=h_1=\ldots=h_{z-\e}=h_1'=h_2'=\ldots=h_{z-\e}'=0.
\]
\end{defn}

For all $\sigma \in \{0,+,-\}$, $R_\sigma$ is a ring with admissible parameters, by definition of $I_\sigma$ in Lemma \ref{lemma:Rsigma}.
Furthermore, $R_0$ is a ``universal'' ring with admissible parameters as demonstrated in the following proposition. This result allows us to deduce future results by initially proving them for $R_0$ and then specialising to another ground ring with admissible parameters.

\begin{prop} \label{prop:universalmap}
  Let $R$ be as in Definition \ref{defn:bnk} with admissible parameters $A_0$, \ldots, $A_{k-1}$, $q_0, \ldots$, $q_{k-1}$, $q$ and $\l$. Then there exists a unique map $R_0\rightarrow R$ which respects the parameters.
\end{prop}
\begin{proof}
  There is a unique ring map $\rho: \Omega\rightarrow R$ which respects the parameters. Furthermore, the admissibility of the parameters in $R$ is equivalent to \[
\rho(\langle\beta,h_0,h_1,\ldots,h_{z-\e},h_1',\ldots,h_{z-\e}'\rangle _\Omega)=0.
\]
That is, the map $\rho$ kills $I_0 \subseteq \Omega$ and hence factors through $R_0$.
\end{proof}

%%%%%%%%%%%%% proving V is a module now %%%%%%%%%%%%%%%

We now proceed to prove that the free $R_0$-module $V$ is in fact a $\bt(R_0)$-module.
For now let us assume that we are working over $R_0$ and denote $\bt(R_0)$ simply by $\bt$.

Recall the maps $\b{Y}$, $\b{X}$ and $\b{E}: V \rightarrow V$ defined at the beginning of this section. Now, by equation (\ref{hiprime}), we have $h_{k-l}=0$ in $R_0$ for $1\leq l\leq z-\e$. Thus $h_l=0$ for all $0\leq l\leq k-1$. Then, by our argument earlier, $\b{Y}^{-1}\b{W}\b{Y}^{-1}=\b{X}$ holds on all of 
$V$ and hence rearranging gives
\[
\b{Y}\b{X}\b{Y} = \b{W} = \b{X} - \d + \d\b{E}.
\]
Thus $\b{Y}\b{X}\b{Y}\b{X} - \b{X}\b{Y}\b{X}\b{Y} = 
[\b{Y}\b{X}\b{Y},\b{X}] = [\d \b{E}, \b{X}]$, where 
$[\ \text{,}\ ]$ denotes the standard commutator of two maps.

Since $\im(\b{E})=\la{v_0}$ and $\b{X}v_0=\l v_0$ by definition,
\begin{equation} \label{D1}
\im([\d \b{E},\b{X}]) \subseteq \la{v_0}.
\end{equation}

\noi Let $N := \b{Y}\b{X}\b{Y}\b{X} - 1$. Then 
\be
[N,\b{Y}] &= N\b{Y} - \b{Y}N \\ 
&= \b{Y}\b{X}\b{Y}\b{X}\b{Y} - \b{Y} - \b{Y}\b{Y}\b{X}\b{Y}\b{X} + \b{Y} \\
&= \b{Y}(\b{X}\b{Y}\b{X}\b{Y} - \b{Y}\b{X}\b{Y}\b{X}) \\
&= -\b{Y}[\d \b{E},\b{X}]
\end{align*}
and
\[
[N,\b{Y}^{-1}] = -\b{Y}^{-1} [N,\b{Y}] \b{Y}^{-1} = [\d \b{E},\b{X}] \b{Y}^{-1}.
\]
Therefore, by (\ref{D1}),
\begin{equation} \label{D2}
  \im([N,\b{Y}]) \subseteq \la{v_1} \text{\quad and \quad}  \im([N,\b{Y}^{-1}]) \subseteq \la{v_0}.
\end{equation}
Observe that 
\[
  \b{Y}\b{X}\b{Y}\b{X}v_0 = \l \b{Y}\b{X}\b{Y}v_0 = \l \b{Y}\b{X}v_1 = \l \b{Y} (\l^{-1} \b{Y}^{-1} v_0) = v_0
\]
and
\[
  \b{Y}\b{X}\b{Y}\b{X}v_1 = \b{Y}\b{X}\b{Y} (\l^{-1} \b{Y}^{-1} v_0) = \l^{-1} \b{Y}\b{X} v_0 = \l^{-1} \b{Y} (\l v_0) = v_1,
\]
hence $Nv_0=Nv_1=0$.

\begin{lemma} \label{lem:N}
\begin{equation} \label{D3}
Nv_l \in \la{v_1, v_2, \ldots, v_{l-1}} \tf l \geq 1,
\end{equation}
and
\begin{equation} \label{D4}
Nv_{-m} \in \la{v_0,v_{-1}, \ldots, v_{-(m-1)}} \tf m \geq 0.
\end{equation}
\end{lemma}

\begin{proof}
  We have already established $Nv_0 = Nv_1 = 0$ above. To prove the first assertion, we argue by induction on $l$. Assume that $l \geq 2$ and $Nv_{l-1} \in \la{v_1, v_2, \ldots, v_{l-2}}$. Then
\[
Nv_l = [N,\b{Y}] v_{l-1} + \b{Y}N v_{l-1} 
\in \la{v_1} + \la{v_2, \ldots, v_{l-1}},
\]
by (\ref{D2}) and the inductive hypothesis. Thus $Nv_l \in 
\la{v_1, \ldots, v_{l-1}}$ for all $l \geq 1$.

 The second assertion is similar. Assume $m \geq 1$ 
and $Nv_{-(m-1)} \in \la{v_0,v_{-1}, \ldots, v_{-(m-2)}}$. Then
\be
Nv_{-m} &= [N,\b{Y}^{-1}] v_{-(m-1)} \,+\, \b{Y}^{-1} N v_{-(m-1)} \\
  &\in \la{v_0} + \la{v_{-1}, v_{-2}, \ldots v_{-(m-1)}},
\end{align*}
by (\ref{D2}) and the inductive hypothesis.
So $Nv_{-m} \in \la{v_0,v_{-1}, \ldots, v_{-(m-1)}}$ for all $m \geq 0$.

\end{proof}

\begin{lemma} \label{lem:YXYX}
$\b{Y}\b{X}\b{Y}\b{X}$ is the identity map on $V$.
\end{lemma}
\begin{proof}
We are required to show that $N v_i = 0$ for $i = 0,1,\ldots, k-1$. We proceed by induction on $i$. The cases $i=0$ and $i=1$ have been established above. Suppose that $2 \leq i \leq k-1$ and 
\begin{equation} \label{D5}
  Nv_0 = Nv_1 = \ldots = N v_{i-1} = 0.
\end{equation}
Since $\{v_{i-1}, v_{i-2}, \ldots, v_{i-k}\}$ is a basis for $V$, by (\ref{eqn:vbasis}), we have that
\[
v_i \in \la{v_{i-1},v_{i-2}, \ldots, v_0,v_{-1}, \ldots, v_{i-k}}.
\]
Together with our inductive hypothesis (\ref{D5}), this implies 
\be
Nv_i &\in N \la{v_0,v_{-1}, \ldots, v_{i-k}} \\
&\subseteq \la{v_0,v_{-1}, \ldots, v_{i-k+1}}\text{,\quad by } (\ref{D4}).
\end{align*}
However (\ref{D3}) states that $Nv_i \in \la{v_1, v_2, \ldots, v_{i-1}}$. 
Again using that $\{v_{i-1},v_{i-2}, \ldots, v_{i-k+1},v_{i-k}\}$ forms a basis for $V$, 
\[
\la{v_1, v_2, \ldots, v_{i-1}} \cap \la{v_0,v_{-1}, \ldots, v_{i-k+1}} = 0.
\]
Therefore $Nv_i = 0$. So by induction on $i$, we have proved that $N v_i = 0$ for all $0 \leq i \leq k-1$, as required.
\end{proof}

\begin{lemma} \label{lem:V} (cf. Lemma 25 of H\"aring-Oldenburg \cite{HO01}) \\
The following relations hold on $V$: 
\begin{eqnarray}
\sum_{l=0}^{k} q_l \b{Y}^l \= 0 \label{eqn:VYk}\\
\b{X}\b{W} \= \b{W}\b{X} \,\,\,=\,\,\, 1 \label{eqn:Vxw}\\
\b{X} \b{E} \= \l \b{E} \,\,\,=\,\,\, \b{E} \b{X} \label{eqn:Vxe}\\
\b{Y}\b{X}\b{Y}\b{X} \= \b{X}\b{Y}\b{X}\b{Y} \label{eqn:Vyxyx}\\
\b{E} \b{Y}^m \b{E} \= A_m \b{E} \tf  \tw{m}, \label{eqn:eyev}\\
\b{E} \b{Y}\b{X}\b{Y} \= \b{Y}\b{X}\b{Y} \b{E}\,\,\,=\,\,\,\l^{-1} \b{E} 
\label{eqn:Veyxy}
\end{eqnarray} 

\noi Furthermore, $V$ is a $\bt(R_0)$-module with the 
actions of $Y$, $X$, $X^{-1}$ and $e$ on $V$ given by the maps $\b{Y}$, $\b{X}$, $\b{W}$ and $\b{E}$, respectively.
\end{lemma}

\begin{proof}
The $k^{\rm th}$ order relation (\ref{eqn:VYk}) is immediate from (\ref{kthV}).
As a consequence of Lemma \ref{lem:YXYX},
\[
  \b{W}\b{X}=\b{Y}\b{X}\b{Y}\b{X}=1.
\]
Clearly, this implies
\[
\b{X}\b{W} = \b{X}\b{Y}\b{X}\b{Y} = \iy (\b{Y}\b{X}\b{Y}\b{X}) \b{Y} = \iy (1) \b{Y} = 1.
\]
Hence we have proved (\ref{eqn:Vxw}) and (\ref{eqn:Vyxyx}). Moreover,
\[
[\d\b{E}, \b{X}] = [\b{W},\b{X}] = 0.
\]
Since $\delta$ is not a zero divisor in $R_0$, by part (a) of Lemma \ref{lemma:Rsigma}, and since $\End_{R_0}(V)$ is a free $R_0$-module, it follows that $\b{X}$ commutes with $\b{E}$.
Thus, using (\ref{eqn:Ev}) and (\ref{eqn:Xv0}),
\[
\b{E}\b{X} = \b{X}\b{E} = \l \b{E},
\]
proving (\ref{eqn:Vxe}).
Equation (\ref{eqn:eyev}) follows easily from (\ref{eqn:Yv}) and
(\ref{eqn:Ev}). Furthermore, since $\b{Y}\b{X}\b{Y} = \b{W}$,
(\ref{eqn:Vxw}) and (\ref{eqn:Vxe}) imply (\ref{eqn:Veyxy}).
The last assertion of Lemma \ref{lem:V} now follows immediately.
\end{proof}

\begin{thm} \label{thm:V}
The map $\varphi: V \rightarrow \mathrm{span}_{R_0}\{Y^i e \mid 0 \leq i \leq k-1\}$ which maps $v_i$ to $Y^i e$ defines a $\bt 
(R_0)$-module isomorphism.
\end{thm}
\begin{proof}
Let $U:=\mathrm{span}_{R_0}\{Y^i e \mid 0 \leq i \leq k-1\}$. It is clear that $\varphi: V \rightarrow U$ is a surjective $R_0$-module homomorphism. Recall that $U$ is a left ideal in $\bt$. We can therefore define a $\bt$-module homomorphism $\psi: U\rightarrow V$ by
\[
  \psi(a)=a(A_0^{-1}v_0), \quad \text{for all } a \in U.
\]
Then
\[
\psi(\varphi(v_i))=Y^ie(A_0^{-1}v_0)=Y^iv_0=v_i,
\]
for $0\leq i\leq k-1$, so that $\psi\varphi$ is the identity. Since $\varphi$ is surjective, it follows that $\psi$ and $\varphi$ are inverses. Therefore they are both $\bt$-module isomorphisms.
\end{proof}

Hence, as a direct consequence of Theorem \ref{thm:V}, the set $\{Y^i e \mid 0 \leq i \leq k-1\}$ is linearly independent. A more direct proof of this is given as follows. Suppose $\sum_{i=0}^{k-1}\beta_i Y^i e = 0$, where $\beta_i \in R_0$. Considering the action of both sides on $v_0$ gives
\[
  A_0 \sum_{i=0}^{k-1}\beta_i v_i = 0.
\]
But $A_0$ is invertible and $v_0, \ldots, v_{k-1}$ are linearly independent over $R_0$, so each $\beta_i$ must be $0$.

We have established that $V \cong R_0^k$ is a free $\bt(R_0)$-module.
Now, for a general ring $R$ with admissible parameters $A_0, \ldots, A_{k-1}, q_0, \ldots, q_{k-1}, q, \l$, the map given in Proposition \ref{prop:universalmap} allows us to specialise from $R_0$ to $R$. Thus $V \otimes_{R_0} R \cong R^k$ is a $\bt(R_0)\otimes_{R_0}R$-module. It therefore becomes a $R$-free $\bt(R)$-module via the map 
$\bt(R)\rightarrow \bt(R_0)\otimes_{R_0}R$. In other words, we have now established the following result.
\begin{cor}
 The $R$-module $V \otimes_{R_0} R$, which shall also be denoted by $V$, is a $R$-free $\bt(R)$-module.
\end{cor}

Henceforth we assume $R$ to be as in Definition \ref{defn:bnk} with admissible parameters $A_0, \ldots, A_{k-1}$, $q_0, \ldots, q_{k-1}$, $q$ and $\l$.

%%%%%%%%%%%%%%%%%%%%%%%%%%%%%%%%%%%%%%%%%%%%%%%%%%%%%%%%%%%

\section{Comparison of Admissibility Conditions} \label{section:comparison}

It is important to remark here that our definition of admissibility differs from the various notions of admissibility arising in H\"aring-Oldenburg \cite{HO01} and Goodman and Hauschild Mosley \cite{GH107,GH207}. However, if $\delta$ is not a zero divisor, the above equations are equivalent to those obtained by H\"aring-Oldenburg. In contrast, Goodman and Hauschild Mosley impose infinitely many relations which are polynomials in the $A_i$'s. We now shed some light on their relations and demonstrate how our admissibility conditions relate to theirs. Specifically, we prove that all of their relations hold in $R_0$. This then allows us to replace their invalid generic ground ring with our $R_0$ and proceed with a similar argument from \cite{GH107} to prove linear independence of our spanning set for general $n$.

The algebras $\bnk$ are defined slightly differently by Goodman and Hauschild Mosley in \cite{GH107}. Specifically, suppose $R'$ is a commutative unital ring containing units $\theta_0, p_0, \ldots, p_{k-1}, q, \l$ and further 
elements $\theta_j$, for $j \geq 1$, such that $\l - \l^{-1} = \d(1-\theta_0)$ holds, where $\d = q - q^{-1}$. They initially consider the affine BMW algebra over $R'$, in which $e_1 Y^{j} e_1 = \theta_{j} e_1$, for all $j \geq 1$, and  define the cyclotomic BMW algebra to be the quotient of this by the ideal generated by the $k^\mathrm{th}$ order relation $\prod_{i=0}^{k-1}(Y-p_i) = 0$. (In this situation, as noted in Chapter \ref{chap:bnkintro}, the $q_i$ in relation (\ref{eqn:ycyclo}) would become the signed elementary symmetric polynomials in the $p_i$, where $q_0 = (-1)^{k-1} \prod_i p_i$ is invertible).
They then proceed to define $\theta_{-j} \in R'$ for $j \geq 1$ so that
\[
e_1 Y^{-j} e_1 = \theta_{-j} e_1.
\]
More precisely, from our proof of Lemma \ref{lemma:eype}, we know that
\[
X_1 Y^{-p} e_1 = \l Y^p e_1 + \d \pum{s} Y^{p-2s} e_1 - \d \pum{s} Y^{p-s} e_1 Y^{-s} e_1,
\]
for all $p \geq 1$.
Then multiplying on the left hand side by $e_1$ gives
\[
\l e_1 Y^{-p} e_1 = \l e_1 Y^p e_1 + \d \pum{s} e_1 Y^{p-2s} e_1 - \d \pum{s} \theta_{p-s} e_1 Y^{-s} e_1.
\]
Now, applying a change of summation $s \mapsto p-s$, we obtain
\[
\l e_1 Y^{-p} e_1 = \l \theta_p e_1 + \d \sum_{s=0}^{p-1} e_1 Y^{2s-p} e_1 - \d \sum_{s=0}^{p-1} \theta_{s} e_1 Y^{s-p} e_1.
\]
Moreover, since $\l^{-1} = \l - \d + \d \theta_0$ in $R'$, this implies
\begin{equation}\label{eqn:eypnege}
e_1 Y^{-p} e_1 = \l^2 \theta_p e_1 + \d\l \sum_{s=1}^{p-1} e_1 Y^{2s-p} e_1 - \d\l \sum_{s=1}^{p-1} \theta_{s} e_1 Y^{s-p} e_1.
\end{equation}
Thus defining $\theta_{-p}$ for $p \geq 1$ inductively by 
\[\theta_{-p} := \l^2 \theta_p + \d\l \sum_{s=1}^{p-1} \theta_{2s-p} - \d\l \sum_{s=1}^{p-1} \theta_{s} \theta_{s-p},
\]
we see that $e_1 Y^{-j} e_1 = \theta_{-j} e_1$, for all $j \geq 1$.

\noi On the other hand, recall that $Y$ satisfies the following equation
\begin{equation} \label{eqn:yk}
\sum_{r=0}^k q_r Y^r = 0.
\end{equation}
Therefore for \emph{all} integers $s$,
\begin{equation} \label{eqn:yks}
\sum_{r=0}^k q_r e_1 Y^{r+s} e_1 = 0.
\end{equation}

If we assume $e_1 \neq 0$ in $\bt(R')$ and that $\bt(R')$ is torsion free then the above results would imply certain relations hold amongst the infinite number of parameters $\theta_j$, where $j \in \Z$. This motivates the following definition, which is precisely the definition of admissibility given by Goodman and Hauschild-Mosley in \cite{GH107}. In order to distinguish the two different notions of admissibility, we shall rename theirs to ``weak admissibility''.

\begin{defn} \label{defn:GHadm}
Let $R'$ be a commutative unital ring containing units $\theta_0, q_0, q, \l$ and further 
elements $q_1, \ldots, q_{k-1}$ and $\theta_j$, for $j \geq 1$. The parameters of $R'$ are \emph{weak admissible} if the following relations hold:
\begin{enumerate}
  \item[(i)] $\l - \l^{-1} = \d(1-\theta_0)$, where $\d = q - q^{-1}$; \\
  \item[(ii)] $\sum_{r=0}^k q_r \theta_{r+s} = 0$, for all $s \in \Z$, where for $p \geq 1$, the element $\theta_{-p}$ is defined by
\begin{equation}\label{thetaneg}
\theta_{-p} := \l^2 \theta_p + \d\l \sum_{s=1}^{p-1} \theta_{2s-p} - \d\l \sum_{s=1}^{p-1} \theta_{s} \theta_{s-p}.
\end{equation}
\end{enumerate}
\end{defn}

Now we are able to show that by extending the parameters $A_j$ to \emph{all} integers $j$ appropriately, the $A_j$ satisfy the equations on $\theta_j$ given by (\ref{thetaneg}).

Recall the relation (\ref{eqn:eye}) of the algebra which states that
\[
e_1 Y^m e_1 = A_m e_1, \quad \text{for all } m = 0, \o, k-1.
\]
and that $Y$ satisfies the relation
\[
\sum_{r=0}^k q_r Y^r = 0.
\]
Therefore for \emph{all} integers $s$,
\[
\sum_{r=0}^k q_r e_1 Y^{r+s} e_1 = 0.
\]
This motivate us to make the following definition:
\begin{defn} \label{defn:Aextended}
  Define $A_j$ inductively for $j <0$ and $j \geq k$ so that
\begin{equation} \label{eqn:Aextended}
\sum_{j=0}^k q_j A_{j+s} = 0, \quad \text{for all } s \in \mathbb{Z}.
\end{equation}
\end{defn}

\begin{prop} \label{prop:WYimpliesGH}
  Let $R$ be as in Definition \ref{defn:bnk} with a set of admissible parameters $A_0$, \ldots, $A_{k-1}$, $q_0$, \ldots, $q_{k-1}$, $q$ and $\l$ (see Definition \ref{defn:adm}). Furthermore, for $j <0$ and $j \geq k$, let $A_j \in R$ be given by Definition \ref{defn:Aextended}. Then the parameters of $R$ are weak admissible, in the sense of Definition \ref{defn:GHadm}, (by identifying $\theta_j$ with $A_j$, for all $j \in \Z$).
\end{prop}
\begin{proof}
 We are now required to prove the following recursive relations hold for all $p \geq 1$:
\begin{equation} \label{eqn:GHadm}
  A_{-p} = \l^2 A_p + \d \l\sum_{s=1}^{p-1} A_{2s-p} - \d \l\sum_{s=1}^{p-1} A_{s} A_{s-p}.
\end{equation}

Recall our $k$-dimensional $\bt(R)$-module V, with basis $\{v_0,\o,v_{k-1}\}$, where the action of the generators $Y$, $X_1$ and $e_1$ were described by the maps $\b{Y}$, $\b{X}$ and $\b{E}$, respectively. 
%As we have shown that these maps respect the relations of the algebra, we may now omit the bars.

\noi \textbf{Claim:} For all $s \in \Z$, $\b{E} v_{s}$ = $\b{E} \b{Y}^{s} v_0 = A_{s} v_0$.

\noi By equation (\ref{eqn:yk}), $\sum_{r = 0}^k q_r Y^{r+s} v_0 = 0$, for all $s \in \Z$. Since $v_i = Y^i v_0$ by definition, this implies that $\sum_{r = 0}^k q_r v_{r+s} = 0$, hence $\sum_{r = 0}^k q_r \b{E} v_{r+s} = 0$, for all $s \in \Z$. Together with equation (\ref{eqn:Aextended}), this shows that
\[
\sum_{r = 0}^k q_r \left( \b{E} v_{r+s} - A_{r+s} v_0 \right) = 0\text{, for all } s \in \Z.
\]
But we know, by definition of $\b{E}$, $\b{E}v_s - A_s v_0 = 0$, for all $s = 0, \o, k-1$. Thus, by induction on $s$, it must be zero for all integers $s$, verifying our claim. Note that, in particular, this implies $\b{E} v_{-p} = A_{-p} v_0$, for all $p \geq 1$. 

Now, since we have proved $V$ is a $\bt(R)$-module, in particular $\b{X}\b{E} v_p = \b{E} \b{X} v_p$, for all $p \geq 1$.
Using equation (\ref{eqn:magic}) and Lemma \ref{lem:YXYX}, a straightforward calculation shows
\begin{align*}
  \b{X} v_p \- \l^{-1} v_{-p} + \d \sum_{i=1}^{p-1} A_{p-i} v_{-i} -  \d \sum_{i=1}^{p-1} v_{p-2i} \\
\- \l^{-1} v_{-p} + \d \sum_{s=1}^{p-1} A_{s} v_{s-p} -  \d \sum_{s=1}^{p-1} v_{2s-p}.
\end{align*}  
Therefore, by our claim above,
\[
\b{E} \b{X} v_p = \l^{-1} A_{-p} v_0 + \d \sum_{s=1}^{p-1} A_{s} A_{s-p} v_0 -  \d \sum_{s=1}^{p-1} A_{2s-p} v_0.
\]
On the other hand,
\[
\b{X} \b{E} v_p = A_p \b{X} v_0 \rel{eqn:Xv0} \l A_p v_0.
\]
Thus, as $v_0$ is a basis element, $\b{X}\b{E} v_p = \b{E} \b{X} v_p$ implies that
\begin{align*}
  \l A_p \- \l^{-1} A_{-p} + \d \sum_{s=1}^{p-1} A_{s} A_{s-p} - \d \sum_{s=1}^{p-1} A_{2s-p} \\
\Leftrightarrow \quad A_{-p} \- \l^2 A_p + \d \l\sum_{s=1}^{p-1} A_{2s-p} - \d \l\sum_{s=1}^{p-1} A_{s} A_{s-p},
\end{align*}
which is precisely equation (\ref{eqn:GHadm}).
Hence, by setting
$\theta_j=A_j$ for $j\geq0$, we have proven the parameters $q$, $\l$, $q_0$, \ldots, $q_{k-1}$ and $\theta_j$ are weak admissible, completing the proof of the Proposition.
\end{proof}

\chapter{The Freeness of $\bnk$} \label{chap:basis}

Recall in the introduction we mentioned that the BMW algebras $\bmw_n$ are isomorphic to the Kauffman tangle algebras $\mathbb{KT}_n$ and are of rank $(2n-1)!! = (2n-1) \cdot (2n-3) \cdot \o \cdot 1$, the same as that of the Brauer algebras.
The Brauer algebras were introduced by Brauer \cite{B37} as a device for studying the representation theory of the symplectic and orthogonal groups, and are typically defined to have a basis consisting of Brauer diagrams, which basically look like tangles except over and under-crossings are not distinguished. The BMW algebras are a deformation of the Brauer algebras comparable to the way the Iwahori-Hecke algebras of type $A_{n-1}$ are a deformation of the group algebras of the symmetric group $\mathfrak{S}_n$. Alternatively, the Brauer algebra is the ``classical limit" of the BMW algebra or Kauffman tangle algebra in the sense one just ``forgets" the notion of over and under crossings in tangles diagrams and so tangle diagrams consisting only of vertical strands degenerate into permutations. In fact, by starting with the set of Brauer $n$-diagrams (see Section \ref{section:cyclobrauer}), together with a fixed ordering of the vertices and a rule for which strands cross over which, one may easily write down a diagrammatic basis of the BMW algebra $\mathscr{C}_n$; for more details on this construction, we refer the reader to Morton and Wasserman \cite{MW89} and Halverson and Ram \cite{HR95}.

The affine and cyclotomic Brauer algebras were first introduced by H\"aring-Oldenburg \cite{HO01}, as classical limits of their BMW analogues in the above sense. The cyclotomic case is studied by Rui and Yu in \cite{RY04} and Rui and Xu in \cite{RX07}. (There is also the notion of a $G$-Brauer algebra for an arbitrary abelian group $G$, introduced by Parvathi and Savithri in \cite{PS02}). The cyclotomic Brauer algebra is, by definition, free of rank $k^n (2n-1)!!$. Therefore we would expect the cyclotomic BMW algebras to be of this rank too.
The main aim of this chapter is to establish the linear independence of our spanning set, obtained in Chapter \ref{chap:spanning}, initially over $R_0$, the ``universal'' ground ring with admissible parameters constructed in Lemma \ref{lemma:Rsigma}. To achieve this goal, we use a modification of the arguments made in Section 6 of Goodman and Hauschild Mosley~ \cite{GH107}, which were adapted from Morton and Traczyk~\cite{MT90} and Morton and Wasserman~ \cite{MW89}. We require the existence of a trace on $\bnk(R_0)$ which, using its specialisation into the cyclotomic Brauer algebra, is shown to be non-degenerate. This will prove the linear independency of our spanning set.

This chapter is set out as follows. We begin with a brief introduction on tangles and affine tangles and define the cyclotomic Kauffman tangle algebras through the affine Kauffman tangle algebras. From here, we introduce the cyclotomic Brauer algebras and associate with it a trace map. Using the nondegeneracy of this trace, given by a result of Parvathi and Savithri \cite{PS02}, we are then able to establish the nondegeneracy of a trace on the cyclotomic BMW algebras over a specific quotient ring of $R_\sigma$, for $\sigma \in \{0,+,-\}$. From this, we then deduce that the same result holds for these three rings. In particular, this implies that we have a basis of $\bnk(R_0)$, thereby proving $\bnk(R)$ is $R$-free for \emph{all} rings $R$ with admissible parameters. Moreover, as a consequence of these results, we also prove the cyclotomic BMW algebras are isomorphic to the cyclotomic Kauffman tangle algebras. These results verify the claims made in Goodman and Hauschild Mosley \cite{GH107} for rings with \emph{weak admissible} parameters.

\begin{defn}
  An \emph{\textbf{$n$-tangle}} is a piece of a link diagram,  consisting of a union of arcs and a finite number of closed cycles, in a rectangle in the plane such that the end points of the arcs consist of $n$ points located at the top and $n$ points at the bottom in some fixed position.
\end{defn}

An $n$-tangle may be diagrammatically presented as two rows of $n$ vertices and $n$ strands connecting the vertices so that every vertex is incident to precisely one strand, and over and under-crossings and self-intersections are indicated. In addition, this diagram may contain finitely many closed cycles.
 
\begin{defn}
  Two tangles are said to be \emph{\textbf{ambient isotopic}} if they are related by a sequence of \emph{\textbf{Reidemeister moves}} of types I, II and III (see Figure \ref{fig:reidemeister}), together with an isotopy of the rectangle which fixes the boundary. They are \emph{\textbf{regularly isotopic}} if the Reidemeister move of type I is omitted from the previous definition.
\end{defn}

\begin{figure}[h!]
  \begin{center}
  \begin{tabular}{rc}
\raisebox{.5cm}{I}\qquad\qquad\qquad &   \includegraphics{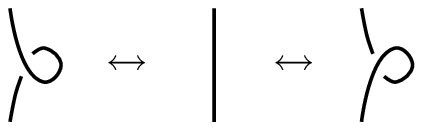} \\[1cm]
\raisebox{.55cm}{II}\qquad\qquad\qquad &   \, \includegraphics{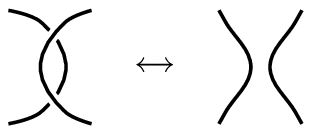} \\[1cm]
\raisebox{.65cm}{III}\qquad\qquad\qquad &    \, \includegraphics{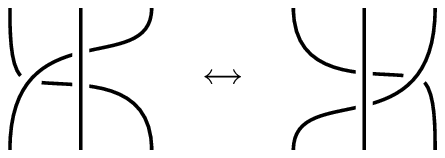}
\end{tabular}
    \caption{Reidemeister moves of types I, II and III.} \label{fig:reidemeister}
  \end{center}
\end{figure}

One obtains a monoid structure on the regular isotopy equivalence classes of $n$-tangles where composition is defined by concatenation of diagrams.
As mentioned in the introduction chapter, the algebra of these tangles together with a skein relation coming from the Kauffman link invariant is a diagrammatic formulation of the original BMW algebra. 

The tangles which appear in the topological intepretation of the affine and cyclotomic BMW algebras feature in type $B$ braid/knot theory. Affine braids or braids of type $B$ (see Lambropoulou \cite{L93}, tom Dieck \cite{T94} and Allcock \cite{A02}) are commonly depicted as braids in a (slightly thickened) cylinder, or a (slightly thickened) annulus, or as braids with a flagpole. Essentially, braids on $n$ strands of type $B$ are just ordinary braids (of type $A$) on $n+1$ strands in which the first strand is pointwise fixed; this single fixed line is usually presented as a ``flagpole'', a thickened vertical segment, on the left and the other strands may loop around this flagpole.
 
\begin{defn}
  An \emph{\textbf{affine $n$-tangle}} is an $n+1$-tangle with a monotonic path joining the first top and bottom vertex, which is diagramatically presented by the flagpole mentioned above.
\end{defn}

\noi The diagram shown in Figure \ref{fig:afftang} is an example of an affine $2$-tangle.
\begin{figure}[h!]
  \begin{center}
    \includegraphics[width=40mm,height=35mm] {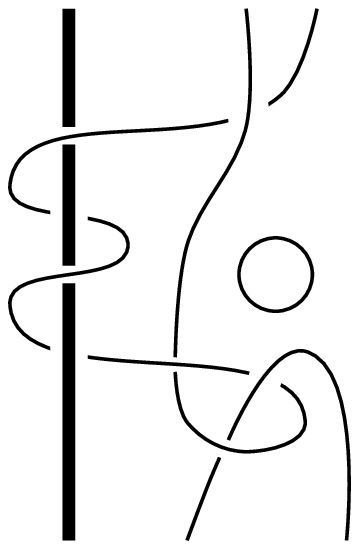}
\caption{Affine $2$-tangle diagram.}
\label{fig:afftang}
  \end{center}
\end{figure} 

Two affine $n$-tangles are ambient (regularly, respectively) isotopic if they are ambient (regularly, respectively) isotopic as $n+1$-tangles. Note that, as the flagpole is required to be a fixed monotonic path, the Reidemiester move of type I is never applied to the flagpole in an ambient isotopy. As with ordinary tangles, the equivalence classes of affine $n$-tangles under regular isotopy carry a monoid structure under concatenation of tangle diagrams. Let $\wh{\mathbb{T}}_n$ denote the monoid of the regular isotopy equivalence classes of affine $n$-tangles.

For $j \geq 0$, let us denote by $\Theta_j$ (the regular isotopy equivalence class of) the non-self-intersecting closed curve which winds around the flagpole in the `positive sense' $j$ times. These are special affine $0$-tangles which will feature in definitions later. Observe that $\Theta_0$ is represented by a closed curve that does not interact with the flagpole. The following figure illustrates $\Theta_3$:
\begin{figure}[h!]
  \begin{center}
    \includegraphics
    {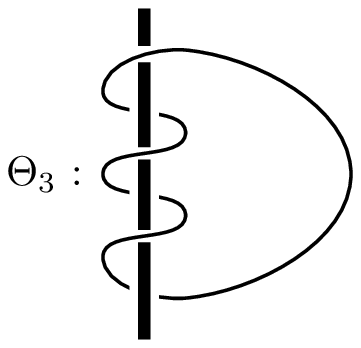}
\caption{$\Theta_3$.}
  \end{center}
\end{figure} 

Using this monoid algebra of affine $n$-tangles, we may now define the affine and cyclotomic BMW algebras. We remark here that the definitions differ slightly to those given in Goodman and Hauschild \cite{GH107}; the difference is that their initial ground ring involves an infinite family of $A_j$'s, for every $j \geq 0$. Instead, here we define the algebras over a ring $R$ under the same assumptions as in the definition of $\bnk$ and take $A_j$, where $j \geq k$, to be elements of $R$ defined by equation (\ref{eqn:Aextended}).

The figures in relations given in the following two definitions indicate affine tangle diagrams which differ locally only in the region shown and are identical otherwise. 

\begin{defn} \label{defn:affineKT}
Let $R$ be as in Definition \ref{defn:bnk}; that is, a commutative unital ring containing units $A_0, q_0, \ldots, q_{k-1},q,\l$ and further elements $A_1, \ldots, A_{k-1}$ such that $\l - \l^{-1} = \d(1-A_0)$ holds, where $\d = q - q^{-1}$. Moreover, let $A_j \in R$, for all $j \geq k$, be defined by equation (\ref{eqn:Aextended}).

  The \emph{\textbf{affine Kauffman tangle algebra}}  $\wh{\mathbb{KT}}_n(R)$ is the monoid $R$-algebra $R\wh{\mathbb{T}}_n$ modulo the following relations:
\pagebreak
\begin{enumerate}
\item[(1)] (Kauffman skein relation)
  \begin{center}
    \includegraphics{pspictures/kauff_scaled.eps}
  \end{center}
\item[(2)] (Untwisting relation) 
\begin{center}
        \includegraphics{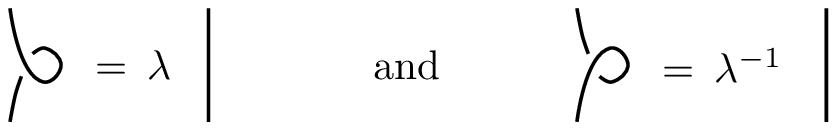}
\end{center}
\item[(3)] (Free loop relations) \\
For $j \geq 0$, 
\[ T \amalg \Theta_j = A_j T, \]
where $T \amalg \Theta_j$ is the diagram consisting of the affine $n$-tangle $T$ and a copy of the loop $\Theta_j$ defined above, such that there are no crossings between $T$ and $\Theta_j$.
\end{enumerate}
\end{defn}

\textbf{Remark:} 
An important case to consider is the affine $0$-tangle algebra $\wh{\mathbb{KT}}_0(R)$. If relation~(3) is removed from the above definition, a result of Turaev \cite{T88} shows that the affine $0$-tangle algebra is freely generated by the $\Theta_j$, where $j \geq 0$, and embeds in the center of the affine $n$-tangle algebra. This motivates relation~(3) in the above definition. This then shows $\wh{\mathbb{KT}}_0(R) \cong R$. 

Recall the affine BMW algebra $\wh{\mathscr{B}}_n(R)$ over $R$ is simply the cyclotomic BMW algebra $\bnk(R)$ with the cyclotomic $k^\mathrm{th}$ order relation on the generator $Y$ omitted. Goodman and Hauschild \cite{GH06} prove the maps given in Figure \ref{fig:dghom} determine an $R$-algebra isomorphism $\wh{\dg}$ between the affine BMW and affine Kauffman tangle algebras.
\begin{figure}[h!]
  \begin{center}
    \includegraphics{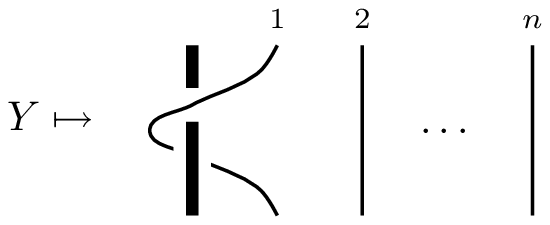} \vskip 1cm
    \includegraphics{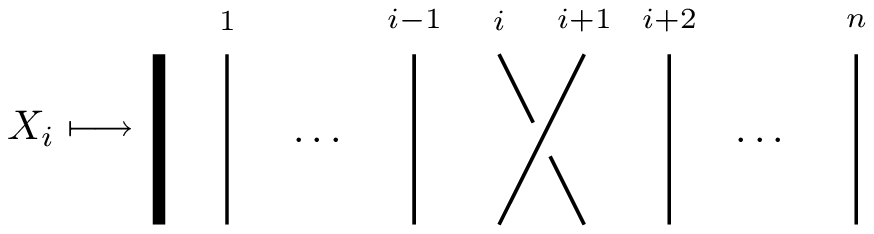}
\vskip 1cm
    \includegraphics{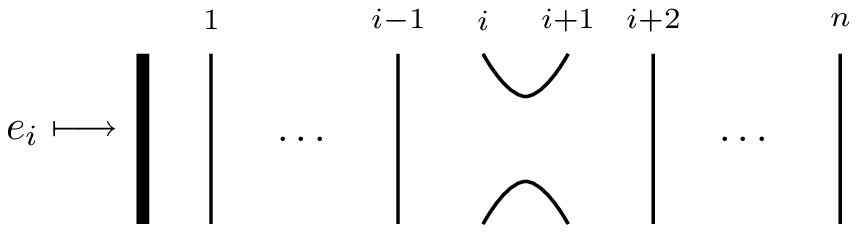}
    \caption{The isomomorphism between the affine BMW and affine Kauffman tangle algebras.}
\label{fig:dghom}
  \end{center}
\end{figure}

We write $\mathcal{Y}$, $\mathcal{X}_i$ and $\mathcal{E}_i$ for the images of the generators $Y$, $X_i$ and $e_i$ under $\wh{\dg}$, respectively.
In particular, the image of $Y_i'$ is exemplified in Figure \ref{fig:y4dash}.
\begin{figure}[h!]
  \begin{center}
    \includegraphics[scale=0.9]{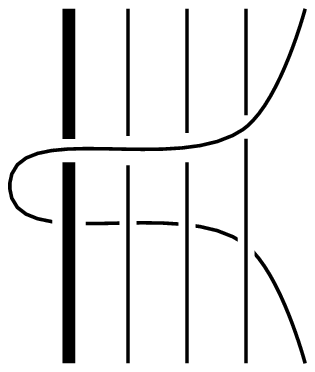}
    \caption{The affine $4$-tangle associated with the element $Y_4'$.}
    \label{fig:y4dash}
  \end{center}
\end{figure} 

\begin{defn}
  Let $R$ be as in Definition \ref{defn:affineKT}.
The \emph{\textbf{cyclotomic Kauffman tangle algebra}} $\ckt{n}(R)$ is the affine Kauffman tangle algebra $\wh{\mathbb{KT}}_n(R)$ modulo the cyclotomic skein relation:
  \begin{center}
\includegraphics[scale=1.2]{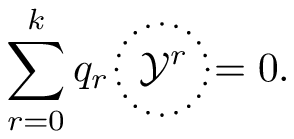}
  \end{center}
The interior of the disc shown in the above relation represents part of an affine tangle diagram isotopic to
$\mathcal{Y}^r$ (see Figure \ref{fig:dghom} for $\mathcal{Y}$).
The sum in this relation is over affine tangle diagrams which differ only in the interior of the disc shown and are otherwise identical. 
\end{defn}

By definition, there is a natural projection $\pi_t: \wh{\mathbb{KT}}_n(R) \twoheadrightarrow \ckt{n}(R)$ and a natural projection $\pi_b: \wh{\mathscr{B}}_n(R) \twoheadrightarrow \bnk(R)$. Moreover, $\hat{\dg}: \wh{\mathscr{B}}_n(R) \rightarrow \wh{\mathbb{KT}}_n(R)$ induces an $R$-algebra homomoprhism $\dg: \bnk(R) \rightarrow \ckt{n}(R)$ such that the following diagram of $R$-algebra homomorphisms commutes. 
\[
\begin{CD}
\wh{\mathscr{B}}_n(R) @>\hat{\dg}>> \wh{\mathbb{KT}}_n(R) \\
@V{\pi_b}VV    @VV{\pi_t}V \\
\bnk(R) @>\dg>> \ckt{n}(R)
\end{CD}
\]
\\
Moreover, because $\wh{\dg}$ is an isomorphism, this implies $\dg: \bnk(R) \rightarrow \ckt{n}(R)$ is surjective. \label{defn:dg}
Furthermore, the homomorphism $\dg$ commutes with specialisation of rings. More precisely, given a parameter preserving ring homomorphism $R_1 \rightarrow R_2$, we can consider $\bnk(R_2)$ as an $R_1$-algebra and construct the $R_1$-algebra homomorphism $\eta_b: \bnk(R_1) \rightarrow \bnk(R_2)$, which sends generator to generator. The ring homomorphism also extends to a $R_1$-algebra homomorphism $\eta_t: \ckt{n}(R_1) \rightarrow \ckt{n}(R_2)$. Then it is easy to verify that $\dg \circ \eta_b = \eta_t \circ \dg$ holds on the generators of $\bnk(R_1)$, hence we have the following commutative diagram of $R_1$-algebra homomorphisms:
\begin{equation} \label{cd:dgspec}
\begin{CD}
\bnk(R_1) @>\eta_b>> \bnk(R_2) \\
@V{\dg}VV    @VV{\dg}V \\
\ckt{n}(R_1) @>\eta_t>> \ckt{n}(R_2)
\end{CD}
\end{equation}

\textbf{Remark:} The tangle analogue of the (\ref{eqn:star}) anti-involution described at the beginning of Chapter \ref{chap:bnkintro} is then just  the anti-automorphism of $\ckt{n}(R)$ which flips diagrams top to bottom. In particular, it fixes $\mathcal{Y}$, $\mathcal{X}_i$ and $\mathcal{E}_i$. 
Furthermore, the map of affine tangles that reverses all crossings, including crossings of strands with the flagpole, determines an isomorphism from $\ckt{n}(q,\l,A_i,q_i) \rightarrow \ckt{n}(q^{-1},\l^{-1},A_{-i},-q_{k-i}q_0^{-1})$.

\section{Construction of a Trace on $\bnk$}

We now work our way towards a Markov trace on the cyclotomic BMW algebras, via maps on the affine BMW algebras, as described in Goodman and Hauschild \cite{GH06}.
Observe that, there is a natural inclusion map $\iota$ from the set of affine $n-1$-tangles to the set of affine $n$-tangles defined by simply adding an additional strand on the right without imposing any further crossings, as illustrated below.  
\begin{figure}[h!]
  \begin{center}
    \includegraphics{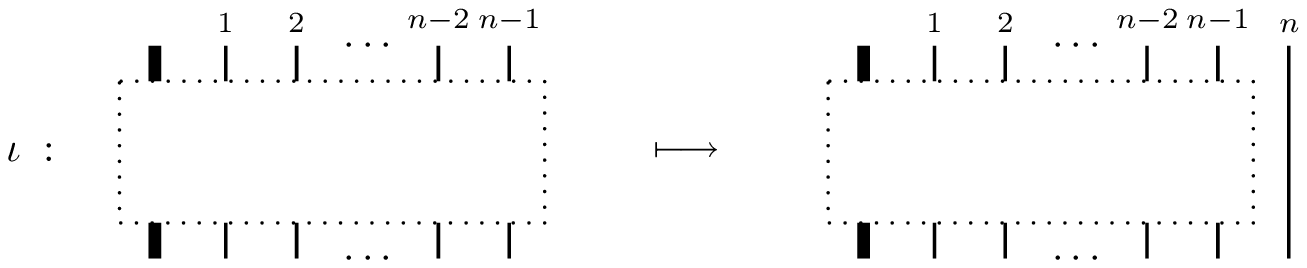}   
  \end{center}
\end{figure}
\\ 
Furthermore, the map $\iota$ respects regular isotopy, composition of affine tangle diagrams and the relations of $\akt{n}$, so induces an $R$-algebra homomorphism $\iota: \akt{n-1}(R) \rightarrow \akt{n}(R)$. Moreover, it respects the cyclotomic skein relation, hence induces an $R$-algebra homomorphism $\iota: \ckt{n-1}(R) \rightarrow \ckt{n}(R)$. 

There is also a ``closure'' map $\cl{n}$ for the affine Kauffman tangle algebras, from the set of affine $n$-tangles to the set of affine $n-1$-tangles, given by closure of the rightmost strand, as illustrated below.

We define $\varepsilon_n: \akt{n}(R) \rightarrow \akt{n-1}(R)$ to be the $R$-linear map given by 
\[
\varepsilon_n(T) := A_0^{-1} \cl{n}(T).
\]
Moreover, taking these closure maps recursively produces a \emph{trace} map on $\akt{n}$. 
We define $\varepsilon_{_{\mathbb{T}}}: \akt{n}(R) \rightarrow \akt{0}(R) \cong R$ by
\[
\varepsilon_{_{\mathbb{T}}} := \varepsilon_1 \circ \cdots \circ \varepsilon_n.
\]
 \begin{figure}[h!]
  \begin{center}
    \includegraphics{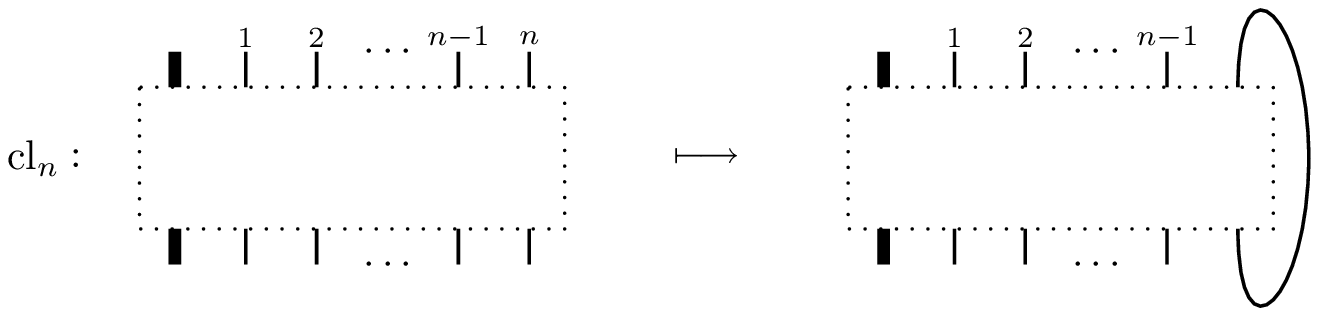}   
  \end{center}
\end{figure} 
\[
\hspace{5mm} \includegraphics{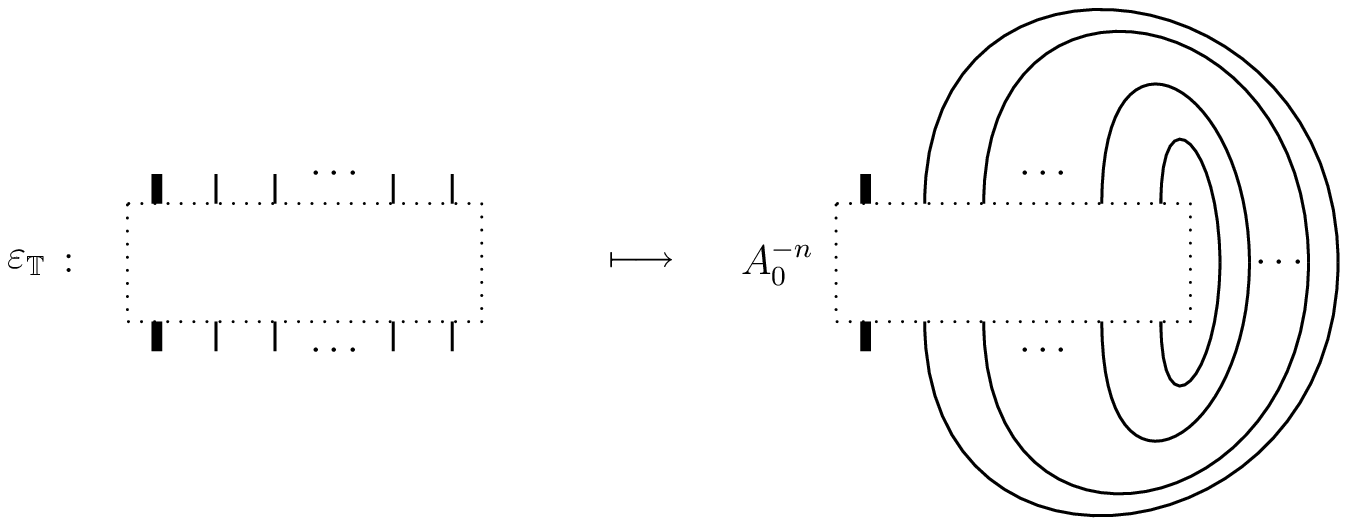}   
\]
\\
In particular, these closure and trace maps respect the cyclotomic skein relation, therefore induce analogous maps $\varepsilon_n: \ckt{n}(R) \rightarrow \ckt{n-1}(R)$ and $\varepsilon_{_{\mathbb{T}}}: \ckt{n}(R) \rightarrow \ckt{0}(R)$.
Now, using Turaev's result and the conditions for weak admissibility (see Definition \ref{defn:GHadm}), one can show that $\ckt{0}(S) \cong S$, for any ring $S$ with \emph{weak admissible} parameters.  Hence we have a trace map 
\[
\varepsilon_{_{\mathbb{T}}}: \ckt{n}(S) \rightarrow S.
\] In particular, by Proposition \ref{prop:WYimpliesGH}, this also holds for any ring with \emph{admissible} parameters.

\textbf{Remark:} Observe that $\varepsilon_n \circ \iota$ is the identity on $\ckt{n-1}$, hence the map $\iota: \ckt{n-1} \rightarrow \ckt{n}$ is injective and we may naturally regard $\ckt{n-1}$ as a subalgebra of $\ckt{n}$.
However, it is \emph{not} clear, \emph{a priori}, that the analogous map $\iota: \bmk \rightarrow \bnk$ defined by $Y \mapsto Y$, $X_i \mapsto X_i$ and $e_i \mapsto e_i$ is injective.

\textbf{Remark:} The map $\varepsilon_n$ is sometimes called a ``\emph{conditional expectation}''; see Goodman and Hauschild \cite{GH06} for more detail. Also, for a ring $S$ with admissible or weak admissible parameters, the trace map $\varepsilon_{_{\mathbb{T}}}: \ckt{n}(S) \rightarrow S$ satisfies certain properties and is usually called a \emph{Markov trace}. This terminology originated in Jones \cite{J83}. For more details, we refer the reader to Goodman and Hauschild \cite{GH06,GH107}.

Taking the closure of the usual braids of type $A$ (or tangles) produces links in $S^3$. In the type $B$ case, closure of affine braids and affine tangles yields links in a solid torus. This leads to the study of Markov traces on the Artin braid group of type $B$ and, moreover, invariants of links in the solid torus. Various invariants of links in the solid torus, analogous to the Jones and HOMFLY-PT invariants (for links in $S^3$), have been discovered using Markov traces on the cyclotomic Hecke algebras; see, for example, Lambropoulou \cite{L99} and references therein.
Kauffman-type invariants for links in the solid torus can be recovered from the Markov trace on the affine BMW algebra. This is discussed by Goodman and Hauschild in \cite{GH06}. 

The trace map on $\ckt{n}$ commutes with specialisation of ground rings, in the following sense.
Suppose we have two rings $S_1$ and $S_2$ with weak admissible parameters and there is a parameter preserving ring homomorphism $\nu: S_1 \rightarrow S_2$. Then the following diagram of $S_1$-linear maps commutes.
\begin{equation} \label{cd:ckttracespec}
\begin{CD}
\ckt{n}(S_1) @>\varepsilon_{_{\mathbb{T}}}>> S_1 \\
@V{\eta_t}VV  @VV{\nu}V \\
\ckt{n}(S_2) @>\varepsilon_{_{\mathbb{T}}}>> S_2
\end{CD}
\end{equation}
Indeed, because $\ckt{n}(S_1)$ is spanned over $S_1$ by affine $n$-tangle diagrams, it suffices to check $\nu \circ \varepsilon_{_{\mathbb{T}}} = \varepsilon_{_{\mathbb{T}}} \circ \eta_t$ on affine $n$-tangle diagrams, by the $S_1$-linearity of $\varepsilon_{_{\mathbb{T}}}$ and definition of $\eta_t$. This then follows because $\cl{n}$ and the isomorphism $\ckt{0}(S) \cong S$ commutes with specialisation. The former is easy to verify, as it suffices to check on diagrams, and the latter is clear since the isomorphism $\ckt{0}(S) \rightarrow S$ is inverse to the natural inclusion map $S \rightarrow \ckt{0}$. 

Using this trace on the cyclotomic Kauffman tangle algebras, we are now able to define a trace on the cyclotomic BMW algebras, over any ring $S$ with weak admissible parameters, by taking its composition with the diagram homomorphism $\dg: \bnk(S) \rightarrow \ckt{n}(S)$ described on page \pageref{defn:dg}.

\begin{defn}
  For any ring $S$ with weak admissible parameters, define the $S$-linear map 
\[\varepsilon_{\!_{\mathscr{B}}} := \varepsilon_{_{\mathbb{T}}} \circ \dg: \bnk(S) \rightarrow S.
\]
\end{defn}
Note that, in particular, we now have a trace map on $\bnk$ over any ring with \emph{admissible} parameters, by Proposition \ref{prop:WYimpliesGH}.

Now, by the commutative diagrams (\ref{cd:dgspec}) and (\ref{cd:ckttracespec}), it is easy to see that $\varepsilon_{\!_{\mathscr{B}}}$ commutes with specialisation of rings as well. In other words, if $S_1$ and $S_2$ are two rings with weak admissible parameters and $\nu: S_1 \rightarrow S_2$ is a parameter preserving ring homomorphism, then the following diagram of $S_1$-linear maps commutes.
\begin{equation} \label{cd:bnktracespec}
\begin{CD}
\bnk(S_1) @>\varepsilon_{\!_{\mathscr{B}}}>> S_1 \\
@V{\eta_b}VV  @VV{\nu}V \\
\bnk(S_2) @>\varepsilon_{\!_{\mathscr{B}}}>> S_2
\end{CD}
\end{equation}
In the next section, we will show that for a particular ring with admissible (and hence weak admissible) parameters, this trace map $\varepsilon_{\!_{\mathscr{B}}}$ is in fact non-degenerate. For this, we consider its relationship with a known non-degenerate trace on the cyclotomic Brauer algebras.

%%%%%%%%%% CYCLOTOMIC BRAUER ALGEBRA %%%%%%%%%%

\section{Cyclotomic Brauer Algebras} \label{section:cyclobrauer}

For a fixed $n$, consider the set of all partitions of the set $\{1,2,\o,n,1',2',\o,n'\}$ into subsets of size two. Any such partition can be represented by a \emph{Brauer $n$-diagram}; that is, a graph on $2n$ vertices with the $n$ top vertices marked by $\{1,2,\o,n\}$ and the bottom vertices $\{1',2',\o,n'\}$ and a strand connecting vertices $i$ and $j$ if they are in the same subset.
Given an arbitrary unital commutative ring $U$ and an element $A_0 \in U$, the \emph{Brauer algebra} $\EuScript{B}_n$ is defined to be the $U$-algebra with $U$-basis the set of Brauer $n$-diagrams. The multiplication of two Brauer $n$-diagrams in the algebra is defined as follows. Given two Brauer $n$-diagrams $D_1$ and $D_2$, let $D_3$ denote the Brauer $n$-diagram obtained by removing all closed loops formed in the concatenation of $D_1$ and $D_2$. Then the product of $D_1$ and $D_2$ is defined to be $A_0^r D_3$, where $r$ denotes the number of loops removed.

The Brauer algebras have been studied extensively in the literature. For example, the generic structure of the algebra and a criterion for semisimplicity of $\EuScript{B}_n$ have been determined; see Wenzl \cite{W88}, Rui \cite{R05} and Enyang \cite{E07} and references therein.

The Brauer algebra is the ``classical limit'' of the BMW algebra in the sense that $\EuScript{B}_n$ is a specialisation of the BMW algebra $\bmw_n$ obtained by sending the parameter $\d$ to $0$. Under this specialisation, the element $X_i$ is then identified with its inverse; this is equivalent to removing the notion of over and under-crossings in $n$-tangles.
In a similar fashion, the cyclotomic Brauer algebras (also termed the ``$\Z_k$-Brauer algebras'' in Goodman and Hauschild Mosley \cite{GH107}) may be thought of as the ``classical limit'' of the cyclotomic Kauffman tangle algebras $\ckt{n}$. The following definition is taken from \cite{GH107}.

\begin{defn}
  A \emph{\textbf{$k$-cyclotomic Brauer $n$-diagram}} (or $\Z_k$-Brauer $n$-diagram) is a Brauer $n$-diagram, in which each strand is endowed with an orientation and labelled by an an element of the cyclic group $\Z_k = \Z/k \Z$. Two diagrams are considered the same if the orientation of a strand is reversed and the $\Z_k$-label on the strand is replaced with its inverse (in the group $\Z_k$).
\end{defn}

\noi An example of a $\Z_6$-Brauer $5$-diagram is given in Figure \ref{fig:cyclobrauer}. 

\begin{figure}[h!]
  \begin{center}
    \includegraphics{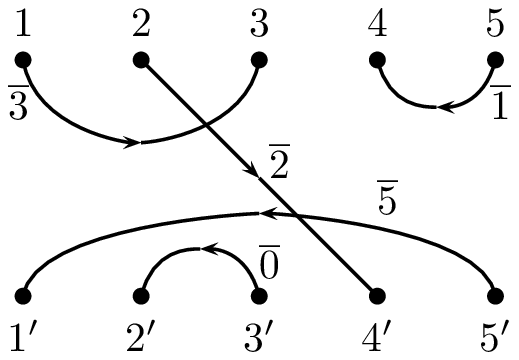}
  \end{center}
  \caption{Example of a $\Z_6$-Brauer $5$-diagram.}
  \label{fig:cyclobrauer}
\end{figure}

Now let $R_c$ denote the polynomial ring $\mathbb{Z}[A_0^{\pm 1}, A_1,\ldots, A_{\floor{k/2}}]$.
The following rules define a multiplication for these diagrams.
Firstly, given two $\Z_k$-Brauer $n$-diagrams $D_1$ and $D_2$, concatenate them as one would for ordinary Brauer $n$-diagrams. In the resulting diagram, horizontal strands, vertical strands and closed loops are formed. For each composite strand $s$, we arbitrarily assign an orientation to $s$ and make the orientations of the components of $s$ from the two diagrams agree with the orientation of $s$ by changing the $\Z_k$-labels for each component of $s$ accordingly. The label of $s$ is then the sum of its consisting component labels. Finally, for $d = 0, 1, \ldots, \floor{\frac{k}{2}}$, let $r_d$ be the number of closed loops with label $\pm d$, mod $k$. Let $D_1 \circ D_2$ denote the $\Z_k$-Brauer $n$-diagram obtained by removing all closed loops and define 
\[D_1 \cdot D_2 := \left( \prod_d A_d^{r_d} \right) D_1 \circ D_2.\]
\begin{defn}
  The \emph{\textbf{cyclotomic Brauer algebra}} (or $\mathbb{Z}_k$-Brauer algebra) $\EuScript{CB}_n^k(R_c)$ is the unital associative $R_c$-algebra with $R_c$-basis the set of $\Z_k$-Brauer diagrams, with multiplication defined by $\cdot$ above.
\end{defn}

We now proceed to show that $\varepsilon_{\!_{\mathscr{B}}}\!: \bnk(R_c) \rightarrow R_c$ is a nondegenerate trace, using an analogously defined trace on $\EuScript{CB}_n^k(R_c)$. 
Rather then considering this trace directly in \cite{GH107}, Goodman and Hauschild Mosley instead consider the trace $\varepsilon_{_{\mathbb{T}}}$ using a so-called `connector' map from the cyclotomic Kauffman tangle algebras to the cyclotomic Brauer algebras.

Just as in the context of tangle algebras, one has a closure map $\mathrm{cl}_n$ from $\Z_k$-Brauer $n$-diagrams to $\Z_k$-Brauer $(n-1)$-diagrams given by joining up the vertices $n$ and $n'$. In addition, any concatenated strands formed in the resulting diagram are labelled according to the same rule as for multiplication in the algebra. Also, if $n$ and $n'$ are joined in the original diagram, then closure will result in a closed loop with some label $\bar{d} \in \Z_k$, where $d = 0,1,\o,\floor{\frac{k}{2}}$ which is then removed and replaced by the coefficient $A_d$. We define $\varepsilon_n: \EuScript{CB}_n^k(R_c) \rightarrow \EuScript{CB}_{n-1}^k(R_c)$ by $\varepsilon_n(D) := A_0^{-1} \mathrm{cl}_n(D)$, for a $\Z_k$-Brauer $n$-diagram $D$. Then $\varepsilon_c := \varepsilon_1 \circ \cdots \circ \varepsilon_n$ is a trace map
\[
\varepsilon_c: \EuScript{CB}_n^k(R_c) \rightarrow \EuScript{CB}_0^k(R_c) = R_c.
\]

The following lemma is due to the work of Parvathi and Savithri \cite{PS02} on $G$-Brauer algebra traces, for finite abelian groups $G$. A sketch of its proof is given in Goodman and Hauschild Mosley \cite{GH107}.

\begin{lemma} \label{lemma:cycloBrauernondeg}
The trace $\varepsilon_c: \EuScript{CB}_n^k(R_c) \rightarrow R_c$ is nondegenerate. In other words, for every $d_1 \in \EuScript{CB}_n^k(R_c)$, there exists a $d_2 \in \EuScript{CB}_n^k(R_c)$ such that $\varepsilon_c(d_1d_2) \neq 0$. Equivalently, it says that the determinant of the matrix $\left( \varepsilon_c(D\!\cdot \!D')\right)_{D,D'}$, where $D,D'$ vary over all $\Z_k$-Brauer $n$-diagrams, is nonzero in $R_c$.
\end{lemma}

\section{Nondegeneracy of the Trace on $\bnk$ over $R_0$}
Let us fix $\sigma$ to be $+$ or $-$.
In the ring $R_c$, put $q := 1$, $\l := \pm 1$, depending on the sign of $\sigma$, $q_0 := 1$ and $q_i : =0$, for all $i = 1, \o, k-1$. Also, let $A_j$, where $j \notin \{0, \o, \floor{\frac{k}{2}}\}$, be such that $A_m = A_{m+k}$ and $A_{-m} = A_m$ hold for all $m \in \mathbb{Z}$. Then it is straightforward to verify that, by identifying $\theta_j$ with $A_j$, the conditions of weak admissibility in Definition \ref{defn:GHadm} are satisfied. Hence $R_c$ is a ring with weak admissible parameters $\l, q, q_i$ and $A_j$, as defined above.

Furthermore, recalling the definition of $R_{\sigma}$ from Lemma \ref{lemma:Rsigma}, we have natural surjective ring homomorphisms $R_{\sigma} \twoheadrightarrow R_c$, for $\sigma =\pm $, given by:
\begin{eqnarray*}
\varsigma: R_\sigma &\rightarrow& R_c \\ 
\l &\mapsto& \pm 1 \\
q &\mapsto& 1 \quad (\Rightarrow \d \mapsto 0) \\
q_0 &\mapsto& 1 \\
q_i &\mapsto& 0 \\
A_j &\mapsto& A_j.
\end{eqnarray*}

 Certainly, $\varsigma$ defines a map $\Omega \rightarrow R_c$. It is easy to verify that the image of $\beta_\sigma$ and $h_l$ under $\varsigma$ are zero, for all $l = 0,1, \o, z-~\epsilon$. Also, by (\ref{hiprime2}), the $h_l'$ are mapped to $A_{k-l} - A_l$ in $R_c$, but these are simply all zero, as $A_l = A_{-l} = A_{-l+k}$ in $R_c$. Thus the generators of $I_\sigma$ indeed vanish under the map $\varsigma$, hence the above defines a ring homomorphism $\varsigma: R_{\sigma} \twoheadrightarrow R_c$. As an immediate consequence of this, we also have a surjective map $R_0 \twoheadrightarrow R_c$, which factors through $R_{\sigma}$. Observe that, by Definition \ref{defn:adm}, the existence of this map is equivalent to the admissibility of the parameters in $R_c$ chosen at the beginning of this section. Hence, by Proposition \ref{prop:WYimpliesGH}, this gives an alternative proof that $R_c$ is a ring with weak admissible parameters.

We have an $R_c$-algebra homomorphism $\xi: \bnk(R_c) \rightarrow \EuScript{CB}_n^k(R_c)$, given by Figure \ref{fig:bnk2brauer}, in which only non-zero labels on strands have been indicated.
\begin{figure}[h!]
  \begin{center}
    \includegraphics{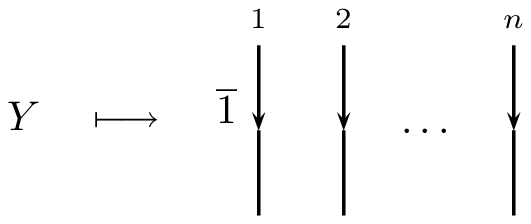} \vskip 1cm
    \includegraphics{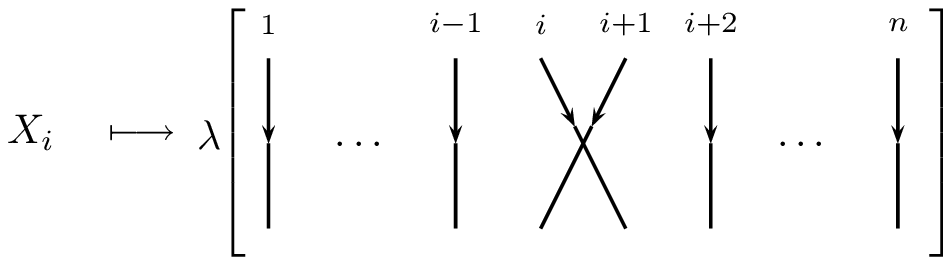}
\vskip 1cm
    \includegraphics{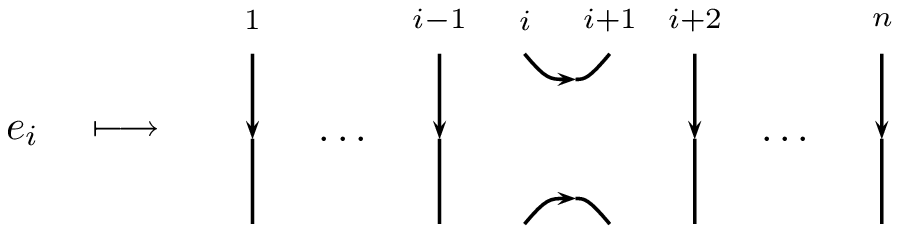}
    \caption{The isomorphism $\xi: \bnk(R_c) \rightarrow \EuScript{CB}_n^k(R_c)$.}
\label{fig:bnk2brauer}
  \end{center}
\end{figure}
\\
It is known that the cyclotomic Brauer algebra $\EuScript{CB}_n^k := \EuScript{CB}_n^k(R_c)$ is generated by the diagrams given in Figure \ref{fig:bnk2brauer}. (We refer the reader to Rui and Xu \cite{RX07} for a full presentation). Hence $\xi$ is surjective. Moreover, we already have a spanning set of size $k^n (2n-1)!!$ of $\bnk(R_c)$, given by Theorem \ref{thm:span}. Since $\xi$ is surjective, this maps onto a spanning set of $\EuScript{CB}_n^k$. But $\EuScript{CB}_n^k$ is of rank $k^n (2n-1)!!$, by definition, hence the image of our spanning set of $\bnk(R_c)$ is in fact a $R_c$-\emph{basis} of $\EuScript{CB}_n^k$. Thus, as $R_c$-algebras, $\bnk(R_c) \cong \EuScript{CB}_n^k$, under $\xi$.

We now compile the above information into the following diagram.
\begin{equation} \label{cd:bmwbrauer}
\begin{CD}
\bnk(R_c) @>{\xi}>{\cong}> \EuScript{CB}_n^k(R_c) \\
@V{\dg}VV  @VV{\varepsilon_c}V \\
\ckt{n}(R_c) @>\varepsilon_{_{\mathbb{T}}}>> R_c
\end{CD}
\end{equation}
Let us now prove that the diagram above commutes. Take a diagram $D \in \EuScript{CB}_n^k$. Then the map $\varepsilon_{_{\mathbb{T}}} \circ \dg \circ \xi^{-1} (D)$ is essentially the trace of the diagram obtained by `seperating' all the non-zero labels from the rest of the diagram and replacing them with appropriate analogous $\mathcal{Y}$-type diagrams. In $\mathbb{KT}_0^k(R_c)$, over and under-crossings do not matter, so $\varepsilon_{_{\mathbb{T}}} \circ \dg \circ \xi^{-1} (D)$ is reduced to disjoint closed loops (around the flagpole), which are then all identified (as $\mathbb{KT}_0^k(R_c) \cong R_c$) with a product of $A_i$'s, where $i$ depends on the original labels in $D$. This produces precisely the same result as taking the trace $\varepsilon_c$ of $D$, as required.

Our aim now is to utilise the above information to show that our spanning set of $\bnk(R_c)$ is linearly independent, by proving $\varepsilon_{\!_{\mathscr{B}}}\!: \bnk(R_c) \rightarrow R_c$ is nondegenerate. From this, we are then finally able to prove that $\bnk(R_0)$ is $R_0$-free of rank $k^n(2n-1)!!$.

\begin{lemma} \label{lemma:bnknondegtrace}
  Suppose $\{T_D\}$ is a spanning set of $\bnk(R_c)$ consisting of $k^n(2n-1)!!$ elements $T_D$ such that $\xi(T_D) = D$, where $D$ is a $\Z_k$-Brauer diagram. Then $\mathrm{det}\left( \varepsilon_{\!_{\mathscr{B}}}(T_{D_1} T_{D_2})\right)_{D_1,D_2} \neq 0$, hence the trace $\varepsilon_{\!_{\mathscr{B}}}\!: \bnk(R_c) \rightarrow R_c$ is nondegenerate.
\end{lemma}
\begin{proof}
  Consider the matrix $\left( \varepsilon_{\!_{\mathscr{B}}}(T_{D_1} T_{D_2})\right)_{D_1,D_2}$, where $D_1$, $D_2$ vary over all $\Z_k$-Brauer diagrams. Recall $\varepsilon_{\!_{\mathscr{B}}} = \varepsilon_{_{\mathbb{T}}} \circ \dg$, by definition. Since $\varepsilon_{_{\mathbb{T}}} \circ \dg = \varepsilon_c \circ \xi$ (see (\ref{cd:bmwbrauer})),
therefore
\[
\mathrm{det}\left( \varepsilon_{\!_{\mathscr{B}}}(T_{D_1} T_{D_2})\right)_{D_1,D_2} = \mathrm{det}(\varepsilon_{_{\mathbb{T}}} \circ \dg(T_{D_1} T_{D_2})) = \mathrm{det}(\varepsilon_c \circ \xi(T_{D_1} T_{D_2})) = \mathrm{det}(\varepsilon_c(D_1 D_2)) \neq 0,
\]
by Lemma \ref{lemma:cycloBrauernondeg}. 
\end{proof}
Since $R_c$ is an integral domain, the matrix $\left( \varepsilon_{\!_{\mathscr{B}}}(T_{D_1} T_{D_2})\right)_{D_1,D_2}$ is invertible over the field of fractions of $R_c$. Thus the set $\{T_D \mid D\text{ is a }\Z_k\text{-Brauer diagram}\}$ is linearly independent over $R_c$. In particular, this implies the spanning set of $\bnk(R_c)$ given in Theorem \ref{thm:span} is a basis over $R_c$.

\begin{thm} \label{thm:thebigbasisone} 
 \begin{enumerate}
 \item[(1)] The spanning set $\mathbb{B}_{R_0}$ of $\bnk(R_0)$ given in Theorem \ref{thm:span} is a basis over $R_0$. Thus $\bnk(R_0)$ is $R_0$-free of rank $k^n (2n-1)!!$.
\item[(2)] $\dg: \bnk(R_0) \rightarrow \ckt{n}(R_0)$ is an $R_0$-algebra isomorphism. 
\end{enumerate}
\end{thm}
\begin{proof}
As discussed above, we have specialisation maps $R_0 \twoheadrightarrow R_\sigma \twoheadrightarrow R_c$. Also, we have a surjection $\bnk(R_0) \rightarrow \bnk(R_c) \cong \EuScript{CB}_n^k$. Thus, for every $\Z_k$-Brauer $n$-diagram $D$, let us choose elements $T_D \in \bnk(R_0)$ which are mapped to $D$ under this surjection. Then, by (\ref{cd:bnktracespec}) above and Lemma \ref{lemma:bnknondegtrace}, the determinant of the trace matrix $\omega := \mathrm{det}\left( \varepsilon_{\!_{\mathscr{B}}}(T_{D_1} T_{D_2})\right)_{D_1,D_2} \in R_0$ has a nonzero image in $R_c$, and hence in $R_\sigma$.

Now suppose $\omega x = 0$, for some $x\in R_0$. Because $R_\sigma$ is an integral domain, by part (c) of Lemma \ref{lemma:Rsigma}, the image of $x$ in $R_\sigma$ must be zero, for 
both $\sigma=+$ and $-$. Hence $x \in I_+ \cap I_-$. However, part (d) of Lemma \ref{lemma:Rsigma} states that $I_0=I_+\cap I_-$, so it follows that $x \in I_0$, hence $x=0$ in $R_0$. That is, $\omega$ is not a zero divisor in $R_0$.

\noi In addition, by definition of $\varepsilon_{\!_{\mathscr{B}}}$,
\[
\mathrm{det}(\varepsilon_{_{\mathbb{T}}} (\dg(T_{D_1}) \dg(T_{D_2})))
 = \mathrm{det}\left( \varepsilon_{\!_{\mathscr{B}}}(T_{D_1} T_{D_2})\right)_{D_1,D_2}
\]
is not a zero divisor in $R_0$. This implies that the set of all $\dg(T_D)$ is linearly independent over $R_0$. Hence $\ckt{n}(R_0)$ contains $k^n(2n-1)!!$ linearly independent elements. However, by surjectivity of $\psi$, the image of $\mathbb{B}_{R_0}$ in $\ckt{n}(R_0)$, under the map $\dg$, is also a spanning set of $\ckt{n}(R_0)$. Hence it must also be linearly independent and of size $k^n(2n-1)!!$. Thus $\mathbb{B}_{R_0}$ is a basis of $\bnk(R_0)$, proving (1) and (2).
\end{proof}

Recall Proposition \ref{prop:universalmap} which states that any ring $R$ with admissible parameters $A_0, \ldots, A_{k-1}$, $q_0$, \ldots, $q_{k-1}$, $q$ and $\l$, admits a unique map $R_0\rightarrow R$. Therefore it makes sense to consider the specialisation algebras $\bnk(R_0) \otimes_{R_0} R$ and $\ckt{n}(R_0) \otimes_{R_0} R$. The following result shows that these are in fact isomorphic to the algebras $\bnk(R)$ and $\ckt{n}(R)$, respectively, and hence $\bnk(R) \cong \ckt{n}(R)$.

\begin{cor}
  Let $R$ be a ring with admissible parameters $A_0, \ldots, A_{k-1}, q_0, \ldots, q_{k-1}$, $q$ and $\l$. Let $\mathbb{B}_R$ denote the set of all
\[
\alpha_{i_1j_1,n-1}^{s_1}\ldots\alpha_{i_mj_m,n-2m+1}^{s_m} \bar \chi^{(n-2m)}(\alpha_{s_mt_m,n-2m+1}^{r_m})^*\ldots(\alpha_{s_1t_1,n-1}^{r_1})^*,
\]
where $m\geq0$, $i_1>i_2>\ldots>i_m$, $s_m<s_{m-1}<\ldots<s_1$, and $\bar{\chi}^{(n-2m)}$ is an element of $\widetilde{\X}_{n-2m,k}$. Then $\mathbb{B}_R$ is an $R$-basis of $\bnk(R)$ and its image under $\dg: \bnk(R) \rightarrow \ckt{n}(R)$ is an $R$-basis of $\ckt{n}(R)$.

Furthermore, $\bnk(R) \cong \bnk(R_0) \otimes_{R_0} R$ and $\ckt{n}(R) \cong \ckt{n}(R_0) \otimes_{R_0} R$, as $R$-algebras, and $\dg: \bnk(R) \rightarrow \ckt{n}(R)$ is an $R$-algebra isomorphism.
\end{cor}

\begin{proof}

Consider the canonical $R$-algebra homomorphism $\bnk(R) \rightarrow \bnk(R_0) \otimes_{R_0} R$. Certainly, it is clear that this is a surjective homomorphism. Thus the image of the spanning set $\mathbb{B}_R$ of $\bnk(R)$ spans $\bnk(R_0) \otimes_{R_0} R$. However, by Theorem \ref{thm:thebigbasisone} above, we have already that $\bnk(R_0) \otimes_{R_0} R$ is $R$-free of rank $k^n(2n-1)!!$. Therefore $\mathbb{B}_R$ is a basis of $\bnk(R)$ and $\bnk(R)\cong \bnk(R_0) \otimes_{R_0} R$.

It is clear that $\ckt{n}(R)$ and $\ckt{n}(R_0) \otimes_{R_0} R$ are isomorphic, as $R$-algebras. Similar reasoning also shows the image of $\mathbb{B}_R$ under $\dg$ is a basis of $\ckt{n}(R)$. 

Finally, we have also proven that 
\[
\bnk(R) \cong \bnk(R_0) \otimes_{R_0} R \cong \ckt{n}(R_0) \otimes_{R_0} R \cong \ckt{n}(R).
\]\!
\end{proof}

Recall we noted earlier that it is not clear \emph{a priori} that $\bmk(R)$ is a subalgebra of $\bnk(R)$. This now follows as a direct consequence of the isomorphism $\bnk(R) \cong \ckt{n}(R)$. Furthermore, we may henceforth identify $\bl{l}$ with $\blk$.

Finally, to help the reader visualise our basis of the cyclotomic BMW algebras, the isomorphism $\dg$ maps $\a{4,6,6}{0} \a{1,3,4}{0}$, considered as an element of $\mathscr{B}_7^k$ to the tangle shown in Figure \ref{fig:alpha}.

\begin{figure}[h!]
   \begin{center}
   \includegraphics{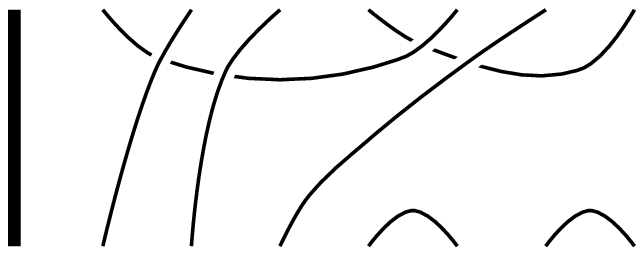}
        \caption{The affine $7$-tangle associated with  the element $\a{4,6,6}{0} \a{1,3,4}{0} \in \mathscr{B}_7^k$.}
     \label{fig:alpha}
   \end{center}
 \end{figure}   
 
%%%%%%%%%%%%%%%%%%%%%%%%%%%%%%%%%%%%%%%%%%%%%%%%%%%

\chapter{The Cellularity of $\bnk$} \label{chap:cellularity}

The theory of cellular algebras was developed in a well-known paper of Graham and Lehrer \cite{GL96}. 
Cellular algebras are a special class of associative algebras defined by a \emph{cell datum} which includes a distinguished basis with certain multiplicative properties that reflect the ideal structure of the algebra and an anti-involution of the algebra.
The principal motivation for their work comes from Kazhdan and Lusztig's study of Iwahori-Hecke algebras in types $A$ and $B$ \cite{KL79}. In particular, the multiplicative and combinatorial properties seen in the Kazhdan-Lusztig basis for Iwahori-Hecke algebras and the ``Robinson-Schensted correspondence'' in the type $A$ case is encapsulated in an axiom for a cell datum. 

A cellular basis of an algebra $\mathscr{A}$ gives rise to a filtration of $\mathscr{A}$, with composition factors isomorphic to the `cell modules' (or `standard modules') of $\mathscr{A}$. Moreover, all of the irreducible $A$-modules arise as quotients of these cell modules by the radical of a symmetric bilinear form defined by the structure constants of the cellular basis. Thus one obtains a complete parametrisation of (the isomorphism classes of) the irreducible $\mathscr{A}$-modules. In addition, important questions, for example, regarding the semisimplicity and quasi-heredity of the algebra, are reduced to linear algebra problems involving this bilinear form.

The general theory of cellular algebras allows one to deduce information about their representation theory, even in the non-semisimple case in most cases. One of the key features of cellular algebras is that the cellular structure is preserved under specialisation. Thus, in this way, the representation theory of an algebra under its generic semisimple setting may be used in many cases to understand non-semisimple specialisations of the algebra.

In \cite{GL96}, Graham and Lehrer prove the cellularity of the Brauer algebras, the Ariki-Koike algebras, the (generalised) Temperley-Lieb algebras and Jones' ``annular'' algebras.
Since then, cellular structures have been discovered for many other important algebras. Examples include the BMW algebras, cyclotomic $q$-Schur algebras, cyclotomic Nazarov-Wenzl algebras and diagram algebras such as the cyclotomic Brauer and Temperley-Lieb algebras and the partition algebras.  
(For example, see \cite{AMR06,DJM98,E04,M99,RX04,RY04,X99,X00}). Also, there is an alternative (equivalent) ring-theoretic and basis-free definition of cellularity, as given by K\"onig and Xi in \cite{KX98}. 

This chapter is concerned with the cellularity of the cyclotomic BMW algebras. As previously mentioned, Enyang \cite{E04} and Xi \cite{X00} utilise the already well-known cellular structure of the Iwahori-Hecke algebras of type $A$ to prove the BMW algebras are cellular. Given that the BMW algebras and Ariki-Koike algebras are cellular, one is naturally lead to ask whether the cyclotomic BMW algebras are also cellular.

Two different cellular bases of the Ariki-Koike algebra were produced in Graham and Lehrer \cite{GL96} and Dipper et al. \!\!\cite{DJM98}. By using a particular ``lifting'' of a slightly modified version of these bases for the $\{\chi^{(l)}\}$ of Theorem \ref{thm:span}, we show that our basis of $\bnk$ is a cellular basis.
A nice feature of our proof is that we need not be explicitly concerned with a particular cellular basis of $\ak{n}$; we need only use the fact that it is cellular with the natural anti-involution and the existence of this lifting map.

In order to use the cellularity of the Ariki-Koike algebras, we must work over a ring $\mathcal{R}$ in which the $k^\mathrm{th}$ order relation $\sum_{j=0}^k q_j y^j$ splits; that is, we require the relation $\prod_{i=0}^{k-1} (Y - p_i) = 0$ to hold, for invertible parameters $p_i \in \mathcal{R}$. (Thus the $q_i$ appearing in relation (\ref{eqn:ycyclo}) now become the signed elementary symmetric polynomials in the $p_i$; in particular, $q_0 = (-1)^{k-1} \prod_i p_i$ is invertible).
For the purposes of cellularity, we will therefore restrict our class of ground rings further to ``split admissible'' rings (see Definition \ref{defn:splitadm}).

\begin{defn} \label{defn:cellular}
  Let $\mathcal{R}$ be a unital commutative ring. An associative $\mathcal{R}$-algebra $\mathscr{A}$ is \emph{\textbf{cellular}} with cell datum $(\Lambda,M,\mathcal{C},\!*)$ if it satisfies the following conditions:
\begin{enumerate}
  \item[(C1)] $\Lambda$ is a finite partially ordered set (poset), and associated with each $\lambda \in \Lambda$ is a finite set $M(\lambda)$ such that the set
\[
\mathcal{C} = \left\{C_{\mathfrak{s},\mathfrak{t}}^{\l} \,|\, \l \in \Lambda\text{ and } \mathfrak{s},\mathfrak{t} \in M(\l) \right\}
\]
is an $\mathcal{R}$-basis of $\mathscr{A}$.
\item[(C2)] The map $^*$ is an anti-involution, i.e. an $\mathcal{R}$-linear involutory anti-automorphism, of $\mathscr{A}$ such that 
\[
\left( C_{\mathfrak{s},\mathfrak{t}}^{\l} \right)^* = C_{\mathfrak{t},\mathfrak{s}}^{\l}.
\]
\item[(C3)] If $\l \in \Lambda$ and $\mathfrak{s},\mathfrak{t} \in M(\l)$ then, for all $a \in \mathscr{A}$, 
\[
aC_{\mathfrak{s},\mathfrak{t}}^{\l} \,\,\equiv \sum_{\mathfrak{s}' \in M(\l)} r_a(\mathfrak{s}',\mathfrak{s}) C_{\mathfrak{s'},\mathfrak{t}}^{\l} \qquad \text{mod } \mathscr{A}(<\l)
\]
where the coefficients $r_a(\mathfrak{s}',\mathfrak{s}) \in \mathcal{R}$ do not depend on $\mathfrak{t}$ and where $\mathscr{A}(<\l)$ is the $\mathcal{R}$-submodule of $\mathscr{A}$ generated by $\{C_{\mathfrak{u},\mathfrak{v}}^{\mu} \mid \mu <\l\text{ and } \mathfrak{u},\mathfrak{v} \in M(\mu)\}$.
\end{enumerate}
\end{defn}

\textbf{Remark:} As an aside note, we remark here that it is possible to work with an extended definition of cellular algebras; Wilcox \cite{W05} removes the assumption that $A$ contains an identity element and that the indexing sets $M(\lambda)$ are finite. Wilcox also introduces the notion of ``conjugate cellular algebras'' (see \cite{W05}). Observe also that, by definition, cellular algebras are required to be finite dimensional. In Green \cite{G99}, the concept of cellular algebras is generalised to infinite dimensional cellular algebras. Also, recently K\"onig and Xi \cite{KX07} have introduced the notion of \emph{affine} cellular algebras.
\\

Let us now recall the definition of the Ariki-Koike algebras $\ak{n}$ from Chapter \ref{chap:spanning}, with the additional assumption that the $k^\mathrm{th}$-order relation on the generator $T_0$ splits over the ground ring. Suppose $\mathcal{R}$ is a unital commutative ring with invertible parameters $q, p_0, p_1\o, p_{k-1}$. Then $\ak{n} := \ak{n}(\mathcal{R})$ is the unital associative $\mathcal{R}$-algebra generated by $T_0^{\pm 1}$, $T_1^{\pm 1}$, \o, $T_{n-1}^{\pm 1}$ subject to the following relations:
\[
\begin{array}{rcll}
  T_0 T_1 T_0 T_1 &=& T_1 T_0 T_1 T_0& \\
  T_i T_{i \pm 1} T_i \= T_{i \pm 1} T_i T_{i \pm 1}&\quad \text{for } i = 1,\o,n-2\\
  T_i T_j \= T_j T_i&\quad \text{for } |i-j|\geq 2 \\
  \prod_{i=0}^{k-1}(T_0-p_i) &=& 0 &\\
   T_i^2 \= (q^2-1)T_i + q^2 &\quad \text{for } i = 1,\o,n-2.
\end{array}
\]
The following result is proved in Graham and Lehrer \cite{GL96} and Dipper et al. \!\!\cite{DJM98}. In both papers however, the invertibility of the parameters $p_0$, $p_1$, \o, $p_{k-1}$ is not required. We require this condition as we assume the generator $Y$ of $\bnk$ is invertible, which is given by the invertibility of $q_0 = (-1)^{k-1} \prod_i p_i$. 
\begin{thm} \label{thm:akcellular}
   The Ariki-Koike algebra $\ak{n}(\mathcal{R})$ is cellular for any unital commutative ring $\mathcal{R}$ with invertible parameters $q, p_0, p_1\o, p_{k-1}$.
\end{thm}
Now suppose $\mathcal{R}$ is as in Definition \ref{defn:bnk}, such that the $k^\mathrm{th}$ order relation $\sum_{j=0}^k q_j y^j$ splits in $\mathcal{R}$. Then recall $\ak{n}(\mathcal{R})$ is a quotient of $\bnk(\mathcal{R})$ under the following projection
\begin{eqnarray*}
  \pi_n: \bnk &\rightarrow& \ak{n} \\
  Y &\mapsto& T_0, \\
  X_i &\mapsto& q^{-1} T_i, \text{ \quad for }1 \leq i \leq n-1  \\
  e_i &\mapsto& 0.
\end{eqnarray*}

%We may consider elements of $\mathfrak{h}_{n,k}(R)$ as elements in $\bnk$ by picking a copy of $\mathfrak{h}_{n,k}$ in $\bnk$ as an $R$-module (\emph{not} an $R$-algebra). Let $\mathfrak{W}_{n,k}$ be any set mapping onto a basis of $\mathfrak{h}_{n,k}$.

Following Ariki and Koike \cite{AK94}, define elements $s_1, s_2, \o, s_n \in \ak{n}$ inductively by $s_1 := T_0$ and $s_i := q^{-2} T_{i-1} s_{i-1} T_{i-1}$ if $i > 1$. Hence $s_i = \pi_n(Y_i')$, for all $i$. Consequently, this shows that the $s_i$ are pairwise commutative.
For $t = 1, 2, \o, k$, let 
\[
  f_t(x):=\prod_{j=1}^{t-1}(x-p_j). 
\]
If $\tau: \{1,\o,n\} \rightarrow \{1,\o,k\}$ is any function, define 
\[
  p^\tau := \prod_{i=1}^n f_{\tau(i)}(s_i) \in \ak{n}.
\]
The symmetric group $\mathfrak{S}_n$ acts on the set of functions $\tau: \{1,\o,n\} \rightarrow \{1,\o,k\}$ by composition;
\[
\tau w := \tau \circ w.
\]
Consider the set $\mathbb{M}$ of non-increasing maps $\tau: \{1,\o,n\} \rightarrow \{1,\o,k\}$. Then $\mathbb{M}$ is a set of orbit representatives for this action. For each $\tau \in \mathbb{M}$, the stabiliser 
\[
S(\tau) := \{w \in \mathfrak{S}_n \mid \tau w = \tau\}
\]
is a standard parabolic subgroup of $\mathfrak{S}_n$. Observe that $S(\tau) \cong \mathfrak{S}_{|\tau^{-1}(1)|} \times \mathfrak{S}_{|\tau^{-1}(2)|} \times \o \times \mathfrak{S}_{|\tau^{-1}(k)|}$ and is generated by the simple transpositions $\{\sigma_i = (i\,\,i+1) \mid \tau(i) = 
\tau(i+1)\}$. Consider the set $D(\tau)$ of distinguished (shortest) left coset representatives for $S(\tau)$ in $\mathfrak{S}_n$. Then 
\[
D(\tau) = \{w \in \mathfrak{S}_n \mid w(i) < w(j)\text{ for all } i,j\text{ such that } i <j\text{ and } \tau(i) = \tau(j)\}. 
\]
Lastly, given a reduced expression $\sigma_{i_1} \o \sigma_{i_l}$ for $w \in \mathfrak{S}_n$, let $T_w := T_{i_1} \o T_{i_l}$; it is well-known from standard Coxeter group theory that the relations of the algebra ensure $T_w$ is independent of the choice of reduced expression for $w$. Furthermore, as for the cyclotomic BMW algebras, there is a natural anti-involution $^*$ of $\ak{n}$ (the same that appears in a cell datum of $\ak{n}$) determined by $T_i^* := T_i$, for all $i = 0, 1, \o, n-1$. Then $T_w^* = T_{w^{-1}}$ and $s_i^* = s_i$. And, since the $s_i$ are pairwise commutative, the $p^{\tau}$ are also fixed by $^*$.
(By abuse of notation, we do not distinguish between the $^*$ anti-involution on $\ak{n}$ with the $^*$ anti-involution of $\bnk$ defined by (\ref{eqn:star}) in Chapter \ref{chap:bnkintro}).
  
By the arguments given in Graham and Lehrer \cite{GL96} and Ariki and Koike \cite{AK94},
\[
  \B=\{T_{d_1}^*p^\tau T_wT_{d_2}\mid\tau\in\M,\;d_1,d_2\in D(\tau)\text{ and }w\in S(\tau)\}.
\]
is an $\mathcal{R}$-basis of $\ak{n}$. 
By Lemma (3.3) of \cite{AK94}, the element $p^\tau$ commutes with $T_w$ if $w\in S(\tau)$. 
Thus
\begin{equation} \label{eqn:b*}
  \left(T_{d_1}^*p^\tau T_wT_{d_2}\right)^*
    =T_{d_2}^*T_{w^{-1}}p^\tau T_{d_1}
    =T_{d_2}^*p^\tau T_{w^{-1}}T_{d_1},
\end{equation}
so $\B$ is setwise invariant under $^*$.

Recall, on page \pageref{defn:phiak}, we chose an arbitrary $\mathcal{R}$-module homomorphism $\phi_n:\ak{n}\rightarrow \bnk$ such that 
$\pi_n\phi_n={\rm id}_{\ak{n}}$. However, in general, the invariance of $\B$ under $^*$ may not be preserved under such a map $\phi_n$. Indeed, one may be tempted to consider the map which essentially `replaces' all $s_i$ and $T_i$ with $Y_i'$ and $X_i$, respectively, and $p^{\tau}$ and $T_w$ with their appropriate analogues $Y^{\tau}$ and $X_w$, respectively, in $\bnk$. Unfortunately, $Y^{\tau}$ and $X_w$ do not necessarily commute, hence the invariance of $\B$ under $^*$ would be lost when mapped to $\bnk$. 
%(\textbf{Stewart,with these $\phi$ maps - can't the subscript be omitted? i mean, now that everything is free, there are no more need for tildes and there can be just a pi and a phi map, right? as in, for example, $\phi_l$ and $\phi_n$ of an element of $\ak{l}$ should be the same, whether considered as an element in the lth or nth level of either algebra?})
 
Our aim now is to ``lift'' $\B$ from $\ak{n}$ to a set in $\bnk$ which is still compatible with the $^*$ anti-involution on $\bnk$ defined by (\ref{eqn:star}) on page \pageref{eqn:star}. In other words, we desire an $\mathcal{R}$-module map $\phi:\ak{n}\rightarrow \bnk$ such that $\pi_n\phi={\rm id}_{\ak{n}}$ and
\begin{equation}\label{phistar}
  \phi(b^*)=\phi(b)^*
\end{equation}
for all $b\in \ak{n}$. Due to the $\mathcal{R}$-linearity of $\phi$, it suffices to define $\phi$ on $\B$ such that (\ref{phistar}) holds for all $b\in\B$.

Now, in $\B$, equation (\ref{eqn:b*}) implies that the element $T_{d_1}^*p^\tau T_wT_{d_2}$ is invariant under $^*$ if and only if $d_1=d_2$ and $w$ is an involution in $S(\tau)$. Thus we may express $\B$ as a disjoint union $\B_1\amalg\B_2\amalg\B_3$, where $\B_2^*=\B_3$ and
\[
  \B_1:=\{T_d^*p^\tau T_wT_d\mid\tau\in\M,\;d\in D(\tau)\text{ and }w=w^{-1}\in S(\tau)\}.
\]
For $b\in\B_2$, let $\phi(b)$ be an arbitrary element of $\pi^{-1}(b)$. If $c\in\B_3$, then $c^*\in\B_2$, 
so we may define $\phi(c)=\phi(c^*)^*$. Then, for all $b\in\B_2$ we have
\[
  \phi(b^*)=\phi\big((b^*)^*\big)^*=\phi(b)^*,
\]
as $(b^*)^* = b$.
On the other hand, if $c\in\B_3$ then, by definition,
\[
\big(\phi(c)\big)^*=\big(\phi(c^*)^*\big)^*=\phi(c^*).
\]
Thus (\ref{phistar}) holds on $\B_2\amalg\B_3$. It now remains to define $\phi$ on $\B_1$. For this, we draw on the following standard result from Coxeter group theory, proved in \cite[Proposition 3.2.10]{GP00}: \\
if $w$ is an involution in a Coxeter group $W$ then there exists an element $u\in W$ such that $\ell(uwu^{-1})=\ell(w)-2\ell(u)$ and $uwu^{-1}$ is central in some standard parabolic subgroup.
Specifically, in the case of the symmetric group this can be restated as
follows.

\begin{prop} \label{prop:winvolution}
Let $\sigma_i$ denote the simple transposition $(i,i+1)\in \mathfrak{S}_n$. Suppose $w\in \mathfrak{S}_n$ is an involution. Then $w$ has an expression of the form
\[
  w=u^{-1}\sigma_{i_1}\sigma_{i_2}\ldots\sigma_{i_l}u,
\]
where $i_{m+1}\geq i_m+2$, for all $m=1,2,\ldots,l-1$, and $u$ is an element of $\mathfrak{S}_n$ such that $\ell(uwu^{-1})=\ell(w)-2\ell(u)$, where $\ell(v)=|\{\,(i,j)\mid \text{$i<j$ and $vi>vj$}\,\}|$.
\end{prop}

Let us fix $\tau \in \M$ and consider $b := T_d^*p^\tau T_wT_d\in\B_1$, where $w\in S(\tau)$ is an involution. By Proposition \ref{prop:winvolution}, $w$ has a reduced expression of the form
\[
  w=u^{-1}\sigma_{i_1}\sigma_{i_2}\ldots\sigma_{i_l}u,
\]
where $i_{m+1}\geq i_m+2$ for all $m=1,2,\ldots,l-1$.
Thus
\[
  b=T_d^* p^{\tau} T_{u^{-1}} T_{\sigma_{i_1}}T_{\sigma_{i_2}} \ldots T_{\sigma_{i_l}} T_u T_d.
\]
By definition of $^*$, $T_{u^{-1}} = T_u^*$. Since $u^{-1} \in S(\tau)$, we know that $p^{\tau} T_{u^{-1}} = T_{u^{-1}} p^{\tau}$ and so, as $d \in D(\tau)$ is a coset representative, the above expression for $b$ becomes
\begin{equation} \label{eqn:B1b}
  b=T_{ud}^* p^{\tau} T_{\sigma_{i_1}}T_{\sigma_{i_2}} \ldots T_{\sigma_{i_l}} T_{ud}.
\end{equation}

Moreover, since $S(\tau)$ is a parabolic subgroup of $\mathfrak{S}_n$, any reduced expression for $w\in S(\tau)$ involves only generators of $S(\tau)$. Hence we must have $\sigma_{i_m} \in S(\tau)$ and $\tau(i_m)=\tau(i_m+1)$, for all $m=1,2,\ldots,l$. Thus $f_{\tau(i_m)}(s_{i_m})f_{\tau(i_m+1)}(s_{i_m+1})$ is a symmetric polynomial in $s_{i_m}$ and $s_{i_m+1}$. We can therefore rewrite $p^{\tau}$ as
\[
  p^\tau=\Bigg[\prod_{i\neq i_m,i_m+1 \atop \text{for all } m}f_{\tau(i)}(s_i)\Bigg]
    \left[\prod_{m=1}^lg_m(s_{i_m}+s_{i_m+1},s_{i_m}s_{i_m+1})\right]
\]
for some polynomials $g_m(x,y)$. Thus, substituting into (\ref{eqn:B1b}) above gives
\[
  b= T_{ud}^*\Bigg[\prod_{i\neq i_m,i_m+1 \atop \text{for all } m}f_{\tau(i)}(s_i)\Bigg]
\left[\prod_{m=1}^lg_m(s_{i_m}+s_{i_m+1},s_{i_m}s_{i_m+1})T_{i_m}\right]T_{ud}.
\]

As $X_{i_m}$ does not necessarily commute with $Y^\tau$ in $\bnk$, we now instead use the following result to help us define $\phi(b)$ in order to satisfy (\ref{phistar}).
\begin{lemma} \label{lemma:almostcomm}
For all $i$, we have $[X_i,Y_i'Y_{i+1}']=0$ and $  [X_i,Y_i'+Y_{i+1}']=\delta[Y_{i+1}',e_i]$, where $[\ \text{,}\ ]$ denotes the standard commutator of two elements in $\bnk$.
\end{lemma}
\begin{proof}
We have
\[
X_iY_i'Y_{i+1}'=Y_{i+1}'X_i^{-1}Y_{i+1}'=Y_{i+1}'Y_i'X_i =Y_i'Y_{i+1}'X_i,
\]
by equation (\ref{eqn:prop1c}) of Proposition \ref{prop:conseq}.
Also,
\begin{eqnarray*}
  X_i(Y_i'+Y_{i+1}')
    &=&Y_{i+1}'X_i^{-1}+X_i^2Y_i'X_i\\
&\stackrel{(\ref{eqn:ix}),(\ref{eqn:xi2})}{=}&Y_{i+1}'X_i-\delta Y_{i+1}'+\delta Y_{i+1}'e_i+Y_i'X_i+\delta X_iY_i'X_i-\delta\lambda e_iY_i'X_i\\
    &\rel{eqn:untwist}& (Y_{i+1}'+Y_i')X_i+\delta [Y_{i+1}',e_i].
\end{eqnarray*}
\end{proof}

\noi Therefore, applying Lemma \ref{lemma:almostcomm} recursively, we obtain 
\begin{align*}
  [X_i,(Y_i'+Y_{i+1}')^r(Y_i'Y_{i+1}')^s]    &=\sum_{c=1}^r(Y_i'+Y_{i+1}')^{c-1}\delta[Y_{i+1}',e_i](Y_i'+Y_{i+1}')^{r-c}(Y_i'Y_{i+1}')^s\\
&=\left[Y_{i+1}',\delta\sum_{c=1}^r(Y_i'+Y_{i+1}')^{c-1}e_i(Y_i'+Y_{i+1}')^{r-c}\right],
\end{align*}
by equations (\ref{eqn:prop1c}) and (\ref{eqn:m9}) of Proposition \ref{prop:conseq}. Rearranging, this gives
\begin{eqnarray*}
  X_i(Y_i'+Y_{i+1}')^r(Y_i'Y_{i+1}')^s
    +\delta\left[\sum_{c=1}^r(Y_i'+Y_{i+1}')^{c-1}e_i(Y_i'+Y_{i+1}')^{r-c}\right]Y_{i+1}' \hspace{-105mm}\\
    &=&(Y_i'Y_{i+1}')^s(Y_i'+Y_{i+1}')^rX_i+\delta Y_{i+1}'\left[\sum_{c=1}^r(Y_i'+Y_{i+1}')^{r-c}e_i(Y_i'+Y_{i+1}')^{c-1}\right].
\end{eqnarray*}
Thus the right hand side of the above equation is an element of $\pi_n^{-1}((s_i+s_{i+1})^r(s_is_{i+1})^sT_i)$ which is invariant under $^*$. In addition, it is in 
the subalgebra generated by $X_i$, $e_i$, $Y_i'$ and $Y_{i+1}'$. Therefore, for each $m$, there exists
$b_m\in\langle X_{i_m},e_{i_m},Y_{i_m},Y_{i_m+1}\rangle$ which is invariant under $^*$, such that
\[
\pi_n(b_m)=g_m(s_{i_m}+s_{i_m+1},s_{i_m}s_{i_m+1})T_{i_m}.
\]
Now let us define
\[
  \phi(b):=X_{ud}^*\Bigg[\prod_{i\neq i_m,i_m+1 \atop \text{for all } m}f_{\tau(i)}(Y_i')\Bigg]
    \left[\prod_{m=1}^lb_m\right]X_{ud},
\]
where $X_v := X_{g_1} X_{g_2} \o X_{g_j} \in \bnk$ for a reduced expression $\sigma_{g_1} \sigma_{g_2} \o \sigma_{g_j}$ of $v \in \mathfrak{S}_n$. \\
As $i_{m+1}\geq i_m+2$, all the $b_m$ commute with one another and $\prod_{i\neq i_m,i_m+1}f_{\tau(i)}(Y_i')$ commutes with $\prod_{m=1}^lb_m$, by equations (\ref{eqn:prop1b}) and (\ref{eqn:prop1c}). Thus $\phi(b)$ is invariant under $*$ and maps to $b$ under $\pi$. Hence $\phi(b)^*=\phi(b)=\phi(b^*)$, as required.
Now that we have defined $\phi(b)$, for all $b \in \B$, we extend $\phi$ to all of $\ak{n}$ by $\mathcal{R}$-linearity.
To summarise, we now have a map $\phi: \ak{n} \rightarrow \bnk$ such that the following diagram commutes.
\[
\begin{CD}
\ak{n}(\mathcal{R}) @>>{\phi}> \bnk(\mathcal{R}) \\
@V{*}VV  @VV{*}V \\
\ak{n}(\mathcal{R}) @>>{\phi}> \bnk(\mathcal{R})
\end{CD}
\]
Since $n$ was arbitrary, we have maps $\phi_l: \ak{l} \rightarrow \blk$ satisfying $\pi_l \phi_l = {\rm id}_{\ak{l}}$ which commute with $^*$.
Now that we have the existence of such maps $\phi_l$ which are compatible with the $^*$ anti-involutions, we are able to proceed with proving the cellularity of $\bnk$.

\begin{defn} \label{defn:splitadm}
  Let $R$ be as in the definition of $\bnk$ (see Definition \ref{defn:bnk}). Then the family of parameters $\left( A_0, \ldots, A_{k-1}, q_0, \ldots, q_{k-1}, q, \l\right)$ is said to be \textbf{\emph{split \,admissible}} if they are admissible (see Definition \ref{defn:adm}) \emph{and} there exists units $p_i \in R$ such that $ y^k-\sum_{j=0}^{k-1}q_jy^j = \prod_{i=0}^{k-1} (y - p_i)$.
\end{defn}

\noi Henceforth, we restrict to ground rings $\mathcal{R}$ with split admissible parameters $A_0, \ldots, A_{k-1}$, $q_0, \ldots, q_{k-1}$, $q$ and $\l$. 

\noi Let $(\Lambda_l,M_l,\mathcal{C},\!*)$ be a cell datum for $\ak{l}$ and
\[
\{C_{\mathfrak{s},\mathfrak{t}}^\lambda\mid\lambda\in\Lambda_l,\;\mathfrak{s},\mathfrak{t}\in M_l(\lambda)\}
\]
be a cellular basis of $\ak{l}$. Let $\bar{C}_{\mathfrak{s},\mathfrak{t}}^\lambda$ denote the image of $\phi_l(C_{\mathfrak{s},\mathfrak{t}}^\lambda)$ in $\bnk$. Observe that, by our construction of $\phi_l$ above,
\begin{equation} \label{eqn:akcstar}
\left(\bar{C}_{\mathfrak{s},\mathfrak{t}}^\lambda\right)^* = \left(\phi_l\left(C_{\mathfrak{s},\mathfrak{t}}^\lambda\right)\right)^*=
\phi_l\left((C_{\mathfrak{s},\mathfrak{t}}^{\lambda})^*\right) =\phi_l(C_{\mathfrak{t},\mathfrak{s}}^\lambda)
=\bar{C}_{\mathfrak{t},\mathfrak{s}}^\lambda.
\end{equation}
Define
\[
\Lambda:=\coprod_{m\geq0}^{\floor{n/2}} \Lambda_{n-2m}.
\]
We extend the partial orders on $\Lambda_l$ to a partial order on $\Lambda$ by declaring that $\Lambda_{n-2m}<\Lambda_{n-2m'}$ for $m>m'$. 
For each $\lambda\in\Lambda_{n-2m}\subseteq\Lambda$, define $M(\lambda):=V_{n,m}\times M_{n-2m}(\lambda)$ and, for each pair of $(a,\mathfrak{s}),(b,\mathfrak{t}) \in M(\lambda)$, let 
\[
C_{(a,\mathfrak{s})(b,\mathfrak{t})}^\lambda:=a\bar{C}_{\mathfrak{s},\mathfrak{t}}^\lambda b^*.
\]

\begin{thm} \label{thm:cellularity}
Let $\mathcal{R}$ be as in Definition \ref{defn:splitadm}. The algebra $\bnk(\mathcal{R})$ is cellular with cell datum $(\Lambda,M,\mathcal{C},*)$ defined as above and cellular basis
\[
    \{C_{(a,\mathfrak{s})(b,\mathfrak{t})}^\lambda \mid\lambda\in\Lambda,\;(a,\mathfrak{s}),(b,\mathfrak{t}) \in M(\lambda)\}
\]
\end{thm}
\begin{proof}
By Theorem \ref{thm:span}, 
\[
    \{C_{(a,\mathfrak{s})(b,\mathfrak{t})}^\lambda \mid\lambda\in\Lambda,\;(a,\mathfrak{s}),(b,\mathfrak{t}) \in M(\lambda)\}.
\]
is a basis of $\bnk$. Hence (C1) is satisfied.
We already know that $^*$, as defined by (\ref{eqn:star}), is an anti-involution of $\bnk$. Moreover, by (\ref{eqn:akcstar}), we see that
\[    \left(C_{(a,\mathfrak{s})(b,\mathfrak{t})}^\lambda\right)^*=\left(a\bar{C}_{\mathfrak{s},\mathfrak{t}}^\lambda b^*\right)^*
        =b\bar{C}_{\mathfrak{t},\mathfrak{s}}^\lambda a^*
        =C_{(b,\mathfrak{t})(a,\mathfrak{s})}^\lambda,
\]
so that (C2) holds.
It now remains to prove the multiplication axiom (C3) holds. Suppose $x\in \bnk$. Let $l=n-2m$, and fix a $\lambda\in\Lambda_l\subseteq\Lambda$ and 
$(a,\mathfrak{s})\in M(\lambda) = V_{n,m} \times M_l(\l)$. By Lemma \ref{lemma:Vgleftideal}, $V_{n,m}\blk$ is a left ideal of $\bnk$, therefore
\[
    xa=\sum_ia_ix_i,
\]
for some $a_i\in V_{n,m}$ and $x_i\in\blk$. Because $\ak{l}$ is cellular, there exists 
$r_{\pi_l(x_i)}(\mathfrak{u},\mathfrak{s})\in \mathcal{R}$ such that
\[    \pi_l(x_i)C_{\mathfrak{s},\mathfrak{t}}^\lambda
\in\sum_{\mathfrak{u}\in M_l(\lambda)}r_{\pi_l(x_i)}(\mathfrak{u},\mathfrak{s})
C_{\mathfrak{u},\mathfrak{t}}^\lambda
+\langle C_{\mathfrak{s}'\mathfrak{t}'}^\mu\mid\mu\in\Lambda_l\text{ and }\mu<\lambda\rangle,
\]
for any $\mathfrak{t}\in M_l(\lambda)$. Recall that $\blk e_{l-1}\blk$ is the kernel of $\pi_l:\blk \rightarrow \ak{l} $. Therefore
\begin{eqnarray*}
    x_i\phi_l(C_{\mathfrak{s},\mathfrak{t}}^\lambda)
        &\in&\sum_{\mathfrak{u}\in M_l(\lambda)}r_{\pi_l(x_i)}(\mathfrak{u},\mathfrak{s})\phi_l(C_{\mathfrak{u},\mathfrak{t}}^\lambda)\\   &&{}+\langle\phi_l(C_{\mathfrak{s}'\mathfrak{t}'}^\mu)\mid\mu\in\Lambda_l\text{ and }\mu<\lambda\rangle+\blk e_{l-1}\blk.
\end{eqnarray*}
Hence
\[
    x_i\bar{C}_{\mathfrak{s},\mathfrak{t}}^\lambda
        \,\in\hspace{-1mm}\sum_{\mathfrak{u}\in M_l(\lambda)}r_{\pi_l(x_i')}(\mathfrak{u},\mathfrak{s})\bar{C}_{\mathfrak{u},\mathfrak{t}}^\lambda
            +\langle\bar{C}_{\mathfrak{s}'\mathfrak{t}'}^\mu\mid\mu\in\Lambda_l\text{ and }\mu<\lambda\rangle
            +\blk e_{l-1}\blk.
\]
Thus, for all $(b,\mathfrak{t})\in M(\lambda)$, we have
\begin{eqnarray}
    xC_{(a,\mathfrak{s})(b,\mathfrak{t})}^\lambda
        &=&xa\bar{C}_{\mathfrak{s},\mathfrak{t}}^\lambda b^* \nonumber\\
        &=&\sum_i a_ix_i\bar{C}_{\mathfrak{s},\mathfrak{t}}^\lambda b^* \nonumber\\
        &\in&\sum_i\sum_{\mathfrak{u}\in M_l(\lambda)}r_{\pi_l(x_i')}(\mathfrak{u},\mathfrak{s})a_i \bar{C}_{\mathfrak{u},\mathfrak{t}}^\lambda b^* \nonumber\\
        &&{}+\langle a_i\bar{C}_{\mathfrak{s}'\mathfrak{t}'}^\mu b^*\mid\mu\in\Lambda_l\text{ and }\mu<\lambda\rangle
            +a_i\blk e_{l-1}\blk b^*. \nonumber\\
        &\subseteq&\sum_i\sum_{\mathfrak{u}\in M_l(\lambda)}r_{\pi_l(x_i')}(\mathfrak{u},\mathfrak{s})C_{(a_i,\mathfrak{u})(b,\mathfrak{t})}^\lambda \nonumber\\
        &&{}+\langle C_{(a_i,\mathfrak{s}')(b,\mathfrak{t}')}^\mu\mid\mu\in\Lambda_l\text{ and }\mu<\lambda\rangle + V_{n,m}\blk e_{l-1}\blk V_{n,m}^*. \label{eqn:actionofx}
\end{eqnarray}
By parts (b) and (c) of Lemma \ref{Iprop},
\begin{eqnarray*}
    V_{n,m}\blk e_{l-1}\blk V_{n,m}^*
        &\subseteq&I_{n,m+1}\\
        &=&\langle C_{(a',\mathfrak{s}')(b',\mathfrak{t}')}^\mu \,|\,\mu\in\Lambda_{n-2m'}\text{ where }m'\geq m+1\rangle\\
        &\subseteq&\langle C_{(a',\mathfrak{s}')(b',\mathfrak{t}')}^\mu\mid \mu<\lambda\rangle,
\end{eqnarray*}
by definition of the partial ordering on $\Lambda$.
Combining this with (\ref{eqn:actionofx}) above, we therefore have 
\[
    xC_{(a,\mathfrak{s})(b,\mathfrak{t})}^\lambda
        \in\sum_i\sum_{\mathfrak{u}\in M_l(\lambda)}r_{\pi_l(x_i)}(\mathfrak{u},\mathfrak{s}) C_{(a_i,\mathfrak{u})(b,\mathfrak{t})}^\lambda
        +\langle C_{(a',\mathfrak{s}')(b',\mathfrak{t}')}^\mu\mid \mu<\lambda\rangle,
\]
proving (C3) and completing the proof of the Theorem.
\end{proof} 

At this point, it would be natural to use the general theory of cellular algebras given in Graham and Lehrer \cite{GL96} to further study the representation theory and structure of $\bnk$, including completely describing its irreducible representations over a field and determining a criterion for semisimplicity. This detailed study will hopefully feature in future work.

%\newpage

\bibliographystyle{plain}

\end{document}